\documentclass[a4paper, 11pt,reqno]{amsart}

\usepackage[T1]{fontenc}      
\usepackage[utf8]{inputenc}
\usepackage[foot]{amsaddr}    
\usepackage{comment}
\usepackage{siunitx}
\usepackage[british]{babel}   
\usepackage{csquotes}         
\usepackage[babel]{microtype} 
\usepackage[shortlabels]{enumitem}

\usepackage{mathtools}        
\usepackage{amssymb}          
\usepackage{amsthm}           
\usepackage{amstext}          
\usepackage{mleftright}       
\usepackage{gensymb}          
\usepackage[makeroom]{cancel} 
\usepackage{mathdots}         
\usepackage{mathrsfs}         
\usepackage{dsfont}           
\usepackage{stickstootext}
\usepackage[stickstoo,varbb]{newtxmath} 
\makeatletter
\DeclareFontFamily{U}{ntxmia}{\skewchar \font =127}
 \DeclareFontShape{U}{ntxmia}{m}{it}{
                        <-> \ntxmath@scaled ntxmia
                      }{}
                      \DeclareFontShape{U}{ntxmia}{b}{it}{
                        <-> \ntxmath@scaled ntxbmia
                      }{}
\makeatother

\usepackage{graphicx}         
\usepackage{caption}          
\usepackage{tikz}             
\usetikzlibrary{backgrounds}
\tikzset{vtx/.style={inner sep=1.7pt, outer sep=0pt, circle, fill,draw}}
\usepackage{pgfplots}
\pgfplotsset{compat=1.18}

\usepackage{booktabs}         

\usepackage[margin=1in]{geometry} 

\usepackage[textsize=footnotesize]{todonotes}

\usepackage[numbers,sort]{natbib}

\makeatletter
\def\NAT@spacechar{~}
\makeatother

\usepackage{hyperref}              
\makeatletter
\@ifundefined{newcounteralias}{%
  \usepackage{aliascnt}
}{}
\makeatother

\usepackage[capitalise]{cleveref}  
\hypersetup{colorlinks=true,
   citecolor=blue,
   filecolor=blue,
   linkcolor=blue,
   urlcolor=blue
}
\usepackage{url}

\crefname{figure}{Figure}{Figures}
\Crefname{figure}{Figure}{Figures}

\usepackage{algorithm}
\usepackage{algpseudocode}

\allowdisplaybreaks

\newtheorem{theorem}{Theorem}[section]

\crefname{theorem}{Theorem}{Theorems}
\Crefname{theorem}{Theorem}{Theorems}

\makeatletter
\@ifundefined{newcounteralias}{%
  \newcommand{\newcreftheorem}[5]{%
    \newaliascnt{#1}{theorem}%
    \newtheorem{#1}[#1]{#2}%
    \aliascntresetthe{#1}%
    \crefname{#1}{#3}{#4}%
    \Crefname{#1}{#2}{#5}%
  }%
}{%
  \newcommand{\newcreftheorem}[5]{%
    \newtheorem{#1}[theorem]{#2}%
    \crefname{#1}{#3}{#4}%
    \Crefname{#1}{#2}{#5}%
  }%
}
\makeatother

\newcreftheorem{proposition}{Proposition}{Proposition}{Propositions}{Propositions}
\newcreftheorem{prop}{Proposition}{Proposition}{Propositions}{Propositions}
\newcreftheorem{corollary}{Corollary}{Corollary}{Corollaries}{Corollaries}
\newcreftheorem{cor}{Corollary}{Corollary}{Corollaries}{Corollaries}
\newcreftheorem{lemma}{Lemma}{Lemma}{Lemmas}{Lemmas}
\newcreftheorem{lm}{Lemma}{Lemma}{Lemmas}{Lemmas}
\newcreftheorem{fact}{Fact}{Fact}{Facts}{Facts}
\newcreftheorem{claim}{Claim}{Claim}{Claims}{Claims}

\crefname{subclaim}{Subclaim}{Subclaims}
\Crefname{subclaim}{Subclaim}{Subclaims}
\newcreftheorem{conjecture}{Conjecture}{Conjecture}{Conjectures}{Conjectures}
\newcreftheorem{question}{Question}{Question}{Questions}{Questions}
\newcreftheorem{problem}{Problem}{Problem}{Problems}{Problems}

\newcreftheorem{algo}{Algorithm}{Algorithm}{Algorithms}{Algorithms}

\theoremstyle{definition}
\newcreftheorem{definition}{Definition}{Definition}{Definitions}{Definitions}
\newcreftheorem{defin}{Definition}{Definition}{Definitions}{Definitions}
\newcreftheorem{example}{Example}{Example}{Examples}{Examples}
\newcreftheorem{remark}{Remark}{Remark}{Remarks}{Remarks}

\newenvironment{claimproof}{%
\let\origqed=\qedsymbol%
\renewcommand{\qedsymbol}{$\blacktriangleleft$}%
\begin{proof}}{\end{proof}\let\qedsymbol=\origqed}

\numberwithin{equation}{section}

\renewcommand{\binom}[2]{\ensuremath{\mleft(\kern-.1em\genfrac{}{}{0pt}{}{#1}{#2}\kern-.1em\mright)}}    
\newcommand{\inbinom}[2]{\ensuremath{\bigl(\kern-.1em\genfrac{}{}{0pt}{}{#1}{#2}\kern-.1em\bigr)}} 

\newcommand*\nume{\ensuremath{\mathrm{e}}}

\newcommand{\cA}{\mathcal{A}}

\newcommand{\cC}{\mathcal{C}}

\newcommand{\cE}{\mathcal{E}}

\DeclareSymbolFont{sfletters}{OML}{cmbrm}{m}{it}
\DeclareMathSymbol{\salpha}{\mathord}{sfletters}{"0B}
\DeclareMathSymbol{\sbeta}{\mathord}{sfletters}{"0C}
\DeclareMathSymbol{\sgamma}{\mathord}{sfletters}{"0D}
\DeclareMathSymbol{\sdelta}{\mathord}{sfletters}{"0E}
\DeclareMathSymbol{\sepsilon}{\mathord}{sfletters}{"0F}
\DeclareMathSymbol{\szeta}{\mathord}{sfletters}{"10}
\DeclareMathSymbol{\seta}{\mathord}{sfletters}{"11}
\DeclareMathSymbol{\stheta}{\mathord}{sfletters}{"12}
\DeclareMathSymbol{\siota}{\mathord}{sfletters}{"13}
\DeclareMathSymbol{\skappa}{\mathord}{sfletters}{"14}
\DeclareMathSymbol{\slambda}{\mathord}{sfletters}{"15}
\DeclareMathSymbol{\smu}{\mathord}{sfletters}{"16}
\DeclareMathSymbol{\snu}{\mathord}{sfletters}{"17}
\DeclareMathSymbol{\sxi}{\mathord}{sfletters}{"18}
\DeclareMathSymbol{\spi}{\mathord}{sfletters}{"19}
\DeclareMathSymbol{\srho}{\mathord}{sfletters}{"1A}
\DeclareMathSymbol{\ssigma}{\mathord}{sfletters}{"1B}
\DeclareMathSymbol{\stau}{\mathord}{sfletters}{"1C}
\DeclareMathSymbol{\supsilon}{\mathord}{sfletters}{"1D}
\DeclareMathSymbol{\sphi}{\mathord}{sfletters}{"1E}
\DeclareMathSymbol{\schi}{\mathord}{sfletters}{"1F}
\DeclareMathSymbol{\spsi}{\mathord}{sfletters}{"20}
\DeclareMathSymbol{\somega}{\mathord}{sfletters}{"21}

\DeclareFontEncoding{LGR}{}{}
\DeclareSymbolFont{sfgreek}{LGR}{cmss}{m}{n}
\SetSymbolFont{sfgreek}{bold}{LGR}{cmss}{bx}{n}
\DeclareMathSymbol{\ssalpha}{\mathord}{sfgreek}{`a}
\DeclareMathSymbol{\ssbeta}{\mathord}{sfgreek}{`b}
\DeclareMathSymbol{\ssgamma}{\mathord}{sfgreek}{`g}
\DeclareMathSymbol{\ssdelta}{\mathord}{sfgreek}{`d}
\DeclareMathSymbol{\ssepsilon}{\mathord}{sfgreek}{`e}
\DeclareMathSymbol{\sszeta}{\mathord}{sfgreek}{`z}
\DeclareMathSymbol{\sseta}{\mathord}{sfgreek}{`h}
\DeclareMathSymbol{\sstheta}{\mathord}{sfgreek}{`j}
\DeclareMathSymbol{\ssiota}{\mathord}{sfgreek}{`i}
\DeclareMathSymbol{\sskappa}{\mathord}{sfgreek}{`k}
\DeclareMathSymbol{\sslambda}{\mathord}{sfgreek}{`l}
\DeclareMathSymbol{\ssmu}{\mathord}{sfgreek}{`m}
\DeclareMathSymbol{\ssnu}{\mathord}{sfgreek}{`n}
\DeclareMathSymbol{\ssxi}{\mathord}{sfgreek}{`x}
\DeclareMathSymbol{\ssomicron}{\mathord}{sfgreek}{`o}
\DeclareMathSymbol{\sspi}{\mathord}{sfgreek}{`p}
\DeclareMathSymbol{\ssrho}{\mathord}{sfgreek}{`r}
\DeclareMathSymbol{\sssigma}{\mathord}{sfgreek}{`s}
\DeclareMathSymbol{\sstau}{\mathord}{sfgreek}{`t}
\DeclareMathSymbol{\ssupsilon}{\mathord}{sfgreek}{`u}
\DeclareMathSymbol{\ssphi}{\mathord}{sfgreek}{`f}
\DeclareMathSymbol{\sschi}{\mathord}{sfgreek}{`q}
\DeclareMathSymbol{\sspsi}{\mathord}{sfgreek}{`y}
\DeclareMathSymbol{\ssomega}{\mathord}{sfgreek}{`w}
\DeclareMathSymbol{\ssvarsigma}{\mathord}{sfgreek}{`c}
\DeclareMathSymbol{\ssGamma}{\mathalpha}{sfgreek}{`G}
\DeclareMathSymbol{\ssDelta}{\mathalpha}{sfgreek}{`D}
\DeclareMathSymbol{\ssTheta}{\mathalpha}{sfgreek}{`J}
\DeclareMathSymbol{\ssLambda}{\mathalpha}{sfgreek}{`L}
\DeclareMathSymbol{\ssXi}{\mathalpha}{sfgreek}{`X}
\DeclareMathSymbol{\ssPi}{\mathalpha}{sfgreek}{`P}
\DeclareMathSymbol{\ssSigma}{\mathalpha}{sfgreek}{`S}
\DeclareMathSymbol{\ssUpsilon}{\mathalpha}{sfgreek}{`U}
\DeclareMathSymbol{\ssPhi}{\mathalpha}{sfgreek}{`F}
\DeclareMathSymbol{\ssPsi}{\mathalpha}{sfgreek}{`Y}
\DeclareMathSymbol{\ssOmega}{\mathalpha}{sfgreek}{`W}

\DeclareRobustCommand{\msf}[1]{%
  \ifcat\noexpand#1\relax\msfgreek{#1}\else\mathsf{#1}\fi
}
\makeatletter
\newcommand{\msfgreek}[1]{\csname ss\expandafter\@gobble\string#1\endcsname}
\makeatother
\newcommand{\vect}[1]{\boldsymbol{\msf{#1}}}

\makeatletter
\def\moverlay{\mathpalette\mov@rlay}
\def\mov@rlay#1#2{\leavevmode\vtop{%
  \baselineskip\z@skip \lineskiplimit-\maxdimen
  \ialign{\hfil$\m@th#1##$\hfil\cr#2\crcr}}}
\newcommand{\charfusion}[3][\mathord]{
    #1{\ifx#1\mathop\vphantom{#2}\fi
        \mathpalette\mov@rlay{#2\cr#3}
      }
    \ifx#1\mathop\expandafter\displaylimits\fi}
\makeatother
\newcommand{\cupdot}{\charfusion[\mathbin]{\cup}{\cdot}}

\newcommand{\mybar}[1]{\makebox[0pt]{$\phantom{#1}\overline{\phantom{#1}}$}#1}

\newcommand{\eps}{\epsilon}

\newcommand{\defi}[1]{\emph{\color{red!60!black}#1}}

\newcommand{\COMMENT}[1]{}
\newcommand{\COMNEW}[1]{}
\newcommand{\aed}[1]{\textcolor{red}{\textbf{[AED: } #1\textbf{]}}}

\title[Factors and powers of Hamilton cycles in the budget-constrained random graph process]{Graph factors and powers of Hamilton cycles\linebreak{} in the budget-constrained random graph process}

\author[A.~Espuny D\'iaz]{Alberto Espuny D\'iaz}
\email{\{espuny-diaz,garbe,smith\}@informatik.uni-heidelberg.de}
\address[Espuny D\'iaz, Garbe, Smith]{Institut f\"ur Informatik, Universit\"at Heidelberg, 69120 Heidelberg, Germany.}
\author[F.~Garbe]{Frederik Garbe}
\author[T.~Naia]{T\'assio Naia}
\email{tnaia@member.fsf.org}
\address[Naia]{Centre de Recerca Matem\`atica, 08193 Bellaterra, Spain.}
\author[Z.~Smith]{Zak Smith}

\thanks{The research leading to these results has been partially supported by the Deutsche Forschungsgemeinschaft (DFG, German Research Foundation) through projects no.\ 513704762 (A.~Espuny D\'iaz) and 428212407 (F.~Garbe, Z.~Smith). 
T.~Naia was supported by the Grant PID2023-147202NB-I00 funded by MICIU/AEI/10.13039/501100011033. 
This work is also partially supported by the Spanish State Research Agency,
through the Severo Ochoa and Mar\'ia de Maeztu Program for Centers
and Units of Excellence in R\&D (CEX2020-001084-M)}

\date{\today}

\begin{document}
\begin{abstract}
We consider the following budget-constrained random graph process introduced by Frieze, Krivelevich and Michaeli.
A player, called Builder, is presented with $t$ distinct edges of $K_n$ one by one, chosen uniformly at random.
Builder may purchase at most $b$ of these edges, and must (irrevocably) decide whether to purchase each edge as soon as it is offered.
Builder's goal is to construct a graph which satisfies a certain property; we investigate the properties of containing different $F$-factors or powers of Hamilton cycles.

We obtain general lower bounds on the budget~$b$, as a function of~$t$, required for Builder to obtain partial $F$-factors, for arbitrary $F$.
These imply lower bounds for many distinct spanning structures, such as powers of Hamilton cycles.
Notably, our results show that, if $t$ is close to the hitting time for a partial $F$-factor, then the budget $b$ cannot be substantially lower than $t$.
These results give negative answers to questions of Frieze, Krivelevich and Michaeli.

Conversely, we also exhibit a simple strategy for constructing (partial) $F$-factors, in particular showing that our general lower bound is tight up to constant factors.
The ideas from this strategy can be exploited for other properties.
As an example, we obtain an essentially optimal strategy for powers of Hamilton cycles.
In order to formally prove that this strategy succeeds, we develop novel tools for analysing multi-stage strategies, which may be of general interest for studying other properties.
\end{abstract}


\maketitle

\section{Introduction}

Given a positive integer $n$, let $M\coloneqq\inbinom{n}{2}$ and consider the \defi{random graph process}, that is, a random sequence of graphs $G_0\subseteq G_1\subseteq\ldots\subseteq G_M$ obtained by letting $G_0$ be the empty graph on a (labelled) set of $n$ vertices and, for each $i\in[M] = \{ 1, \ldots, M \} $, letting $G_i$ be obtained from $G_{i-1}$ by adding to it a new edge $ e_i $, chosen uniformly at random among all of its missing edges.
Note that each $G_m$ in the above process is distributed as $\hat{G}(n,m)$ (which we write to mean a graph chosen uniformly at random among the set of all $n$-vertex graphs with exactly $m$ edges), though of course the graphs in the sequence are not independent.
For a non-trivial increasing graph property~$\mathcal{P}$, there exists a unique $\tau_{\mathcal{P}}\in[M]$ such that $G_{\tau_{\mathcal{P}}}\in\mathcal{P}$ but $G_{\tau_{\mathcal{P}}-1}\notin\mathcal{P}$; such a $\tau_{\mathcal{P}}$ (which is a random variable) is called the \defi{hitting time} for $\mathcal{P}$ in the random graph process.

In general, even graph properties which require ``few'' edges to occur often have very ``large'' hitting times, meaning that, in order for a random graph to satisfy the desired property, one usually needs to have ``many more'' random edges than the minimum number of edges required for the property.
Let us see some classical examples of this behaviour.
We say that an event $\mathcal{A}$ (formally, a sequence of events $(\mathcal{A}_n)_{n\in\mathbb{N}}$, but we omit the dependence on $n$ from the notation) occurs \defi{asymptotically almost surely} (a.a.s.) if $ \mathbb{P}[\mathcal{A}] \to 1 $ as $ n \to \infty $.
\begin{itemize}
    \item A perfect matching is a set of $n/2$ disjoint edges, but a.a.s.\ will not appear in the random graph process until it has at least $(1-o(1))n\log n/2$ edges~\cite{ER66}.
    \item A Hamilton cycle (that is, a cycle containing every vertex of the graph) requires $n$ edges, but a.a.s.\ will not appear in the random graph process until it has more than $n\log n/2$ edges~\cite{KS83}.
    \item A triangle factor (which is a set of $n/3$ vertex-disjoint triangles) again requires only $n$ edges, but a.a.s.\ $G_t$ will not contain one if $t=o(n^{4/3}\log^{1/3}n)$~\cite{JKV08}.
\end{itemize}
Moreover, as edges arrive one by one through the random graph process, it is generally hard to know whether an edge will be crucial for the desired property.
These facts motivated \citet{FKM25} to introduce an ``online'' model for constructing graphs following the random graph process, but with a constraint on the ``budget'', that is, on the number of edges that the final constructed graph is allowed to have.
This ``online'' model is the main object of study in this paper.

Formally, the proposed model (to which we will refer as the \defi{budget-constrained random graph process}) constitutes a one-player game where the edges of the complete graph on $n$ vertices are given a uniformly random order and a player, \textbf{Builder}, is presented the first $t\in[M]$ edges one by one.
Each time she is presented an edge, and without knowing which edges will come next, Builder must irrevocably decide whether to ``purchase'' this edge or not.
The choices made by Builder result in a sequence of graphs $B_0\subseteq B_1\subseteq\ldots\subseteq B_t$, where $B_i\subseteq G_i$ for all $i\in\{0,1,\ldots,t\}$.
The goal of Builder is to construct a graph which satisfies a desired property $\mathcal{P}$; the catch is that she has a limited budget~$b\in[t]$, that is, the graph she constructs must satisfy that $|E(B_t)|\leq b$.
For any pair of positive integers~$t\in[M]$ and~$b\in[t]$, a \defi{$(t,b)$-strategy} for Builder is a (possibly random) function which, for each $i\in[t]$, given the graph $B_{i-1}$ and presented with a new edge at time $i$, decides whether to purchase the edge, under the limitation that $|E(B_i)|\leq b$.
In general, given a $(t,b)$-strategy $\mathcal{S}$ for Builder, we will refer to the (random) sequence of graphs $B_0\subseteq B_1\subseteq\ldots\subseteq B_t$ created by following this strategy on the random graph process as the \defi{budget-constrained random graph process under $\mathcal{S}$}.
Given a monotone graph property $\mathcal{P}$ (we will usually think of $\mathcal{P}$ as the containment of a subgraph isomorphic to a given graph), a $(t,b)$-strategy which a.a.s.\ produces a graph $B_t$ satisfying property $\mathcal{P}$ is called a \defi{successful strategy for $\mathcal{P}$}.

The general question that one would like to understand in the budget-constrained random graph process is the following: given a monotone increasing property $\mathcal{P}$, for which pairs $(t,b)$ does Builder have a successful $(t,b)$-strategy for $\mathcal{P}$?
Note that, by monotonicity, if $t_1<t_2$ and there exists a successful $(t_1,b)$-strategy for $\mathcal{P}$, then there also exists a successful $(t_2,b)$-strategy for $\mathcal{P}$.
One can also make the trivial observations that any successful strategy must use budget $b\geq\min_{G\in\mathcal{P}}|E(G)|$ and take time at least the (usual) hitting time for $\mathcal{P}$ in the random graph process.
But how close to both of these trivial bounds can one get?
Let us briefly present the results which are known for this model.

\subsection{Previous results}

A simple graph property is that of having minimum degree at least $k$, for some $k\geq1$.
Let us denote the hitting time for this property by $\tau_{k}$; it is well known that, for $k$ independent of $n$, a.a.s.\ $\tau_k=(1\pm o(1))n\log n/2$.
\citet{FKM25} showcased a successful $((1+o(1))n\log n/2,(1+o(1))kn/2)$-strategy for minimum degree at least $k$, which is clearly asymptotically optimal both with respect to the time and the budget.
One may improve the error term on the time constraint in exchange for allowing a larger (but still linear) budget:
\citet{FKM25} also proposed a $(\tau_{k},(1+o(1))C_kn)$-strategy which a.a.s.\ produces a graph $B_{\tau_{k}}$ of minimum degree at least $k$.
Here, $k/2<C_k\leq3k/4$ is an explicit constant.
Very recently, \citet{KL25} proposed an alternative strategy which improves the value of $C_k$ to~$k/2+2^{-k}$ for all $k\geq2$.
The optimal value that this constant may take remains unknown.

Let us now consider $k$-connectivity.
For $k=1$ there is a straightforward strategy that results in a connected graph with optimal time and budget: simply purchase any edge that does not create a cycle.
For $k\geq2$, \citet{Lichev22} showcased a successful $((1+o(1))n\log n/2,(1+o(1))kn/2)$-strategy for $k$-connectivity; this, again, is asymptotically optimal in both time and budget.
For $k\geq3$, the $(\tau_{k},(1+o(1))C_kn)$-strategy proposed by \citet{FKM25} for minimum degree at least $k$ a.a.s.\ also results in a $k$-connected graph.

\citet{FKM25} proposed a successful $(\tau_2,Cn)$-strategy for Hamiltonicity, where $C>1$ is an absolute constant (note that this implies that linear budget also suffices at the hitting time for $2$-connectivity).
\citet{Anastos22} showed that at the hitting time $\tau_2$ the constant $C$ cannot be arbitrarily close to $1$, but provided a successful $((1+o(1))n\log n/2,(1+o(1))n)$-strategy for Hamiltonicity (which is asymptotically optimal in both time and budget).
Similar results hold for perfect matchings~\cite{FKM25,Anastos22}.

Lastly, \citet{FKM25} considered the problem of constructing copies of small graphs, and provided asymptotically optimal strategies for constructing trees and cycles, as well as unicyclic graphs, showing a trade-off between the time allowed in the process and the budget needed to construct these graphs.
Recently, \citet{ILS24} obtained similarly tight results for the diamond as well as any number of triangles sharing a unique vertex, and \citet{AEPS26+} obtained such results for wheels and graphs obtained by adding a universal vertex to a tree.
All other fixed graphs remain an open problem.

In view of their results for spanning properties, \citet{FKM25} raised the following question.

\begin{question}[\citet{FKM25}]\label{question:FKMgeneral}
    Let $H$ be an $n$-vertex graph with bounded maximum degree (so $|E(H)|=O(n)$).
    Does Builder have a successful $(t,b)$-strategy for containing a copy of $H$ with $t$ being ``close'' to the hitting time for $H$ and $b$ being linear?
\end{question}

As a concrete example, they asked whether this can be achieved for the square of a Hamilton cycle.
(In this paper, for any positive integer $k$, the \defi{$k$-th power} of a graph $G$ is the graph obtained from $G$ by adding an edge between any pair of vertices whose graph distance in $G$ is at most $k$.)

\begin{question}[\citet{FKM25}]\label{question:FKMsquare}
    Is there a successful $(O(n^{3/2}),O(n))$-strategy for the square of a Hamilton cycle?
\end{question}

In this paper, we show that the answer to these questions is negative in general and that, in some cases, close to the hitting time, the budget cannot be asymptotically smaller than the cost of the trivial strategy of purchasing every edge presented to Builder.
We showcase this behaviour by considering \emph{graph factors}.
In more generality, we obtain new results about the interplay between~$t$ and~$b$ required to have successful strategies for constructing graphs which contain different graph factors or powers of a Hamilton cycle.
We achieve this by providing lower bounds for~$b$ as a function of~$t$ as well as successful $(t,b)$-strategies which provide (almost) matching upper bounds.
We detail our contributions in the coming sections.
It is worth noting that sometimes the strategies used are quite simple and greedy (though analysing the process may not be so simple).
For instance, in the strategy proposed by \citet{FKM25} for constructing a graph with minimum degree at least $k$, Builder accepts every edge one of whose endpoints has degree less than $k$.
The strategies that we showcase in this paper will also be quite simple in nature.

Before describing our results, we remark here that hitting times in the random graph process are deeply tied to \emph{thresholds} in binomial random graphs.
For a positive integer $n$ and $p\in[0,1]$, a \defi{binomial random graph} $G(n,p)$ is a (labelled) $n$-vertex graph obtained by including each of the~$M$ possible edges independently with probability $p$.
A function $p^*=p^*(n)$ is a \defi{threshold} (resp.\ \defi{sharp threshold}) for a non-trivial increasing property $\mathcal{P}$ if (for every $\eps>0$) $\mathbb{P}[G(n,p)\in\mathcal{P}]=o(1)$ whenever $p=o(p^*)$ (resp.\ $p\leq(1-\eps)p^*$) and $\mathbb{P}[G(n,p)\in\mathcal{P}]=1-o(1)$ whenever $p=\omega(p^*)$ (resp.\ $p\geq(1+\eps)p^*$).
Using standard coupling arguments (see \cref{lem:basiccoupling}), it follows that, if $p^*$ is a threshold for property~$\mathcal{P}$, then a.a.s.\ the hitting time $\tau_{\mathcal{P}}$ for~$\mathcal{P}$ in the random graph process satisfies that $\tau_{\mathcal{P}}=\Theta(n^2p^*)$, whereas if~$p^*$ is a sharp threshold for $\mathcal{P}$, then a.a.s.\ $\tau_{\mathcal{P}}=(1\pm o(1))Mp^*$.

\subsection{Factors}

Let $F$ be a fixed graph.
We say that a graph $G$ on $n$ vertices (where $|V(F)|$ divides~$n$) contains an \defi{$F$-factor} if its vertex set can be partitioned into $n/|V(F)|$ sets of equal size in such a way that the graph induced by $G$ on each of the sets contains a copy of $F$.
In more generality (and removing the restriction on the divisibility of $n$), for any $\alpha\in(0,1)$, we say that $G$ contains an \defi{$\alpha$-$F$-factor} if it contains a set of at least $\alpha n/|V(F)|$ vertex-disjoint copies of $F$.
We sometimes generally refer to these as \defi{partial $F$-factors}.

Finding $F$-factors in graphs is one of the classical problems in graph theory, particularly when $F=K_r$ is the complete graph on $r\geq2$ vertices, and has received much attention in the context of random graphs.
The case $F=K_2$, which corresponds to a perfect matching, was settled in a seminal paper of \citet{ER66}.
For other instances of~$F$, this problem was first considered independently by \citet{Ruc92} (who, among other results, obtained the threshold for $\alpha$-$F$-factors for \emph{all} fixed graphs $F$ and any $\alpha\in(0,1)$) and by \citet{AY93}.
The threshold for $F$-factors for all strictly $1$-balanced $F$ in random graphs was determined in the influential paper of \citet{JKV08} (see \cref{thm:JKVbound}).
This includes the case of clique-factors, and their result implies that a.a.s.\ the hitting time $\tau_{K_r}$ for containing a $K_r$-factor satisfies $\tau_{K_r}=\Theta(n^{2-2/r}\log^{1/\binom{r}{2}}n)$.
Subsequently, \citet{GM15} obtained the threshold when $F$ is not vertex-balanced (see \cref{thm:GMbound}).
The current best result in the area for clique factors is a ``hitting time result'' by \citet{HKMP23} (which is established via a coupling with a result of \citet{Kahn22,Kahn23} for perfect matchings in hypergraphs, building on earlier work of \citet{Riordan22} and \citet{Heckel21}), which in particular implies that the hitting time for $K_r$-factors is concentrated around a certain explicit value.
Very recently, \citet{BKMP24} announced an extension of this hitting time result to $F$-factors for graphs $F$ in a larger family of ``nice'' graphs, and \citet{BHKMP24} obtained the sharp threshold for all strictly $1$-balanced graphs.

Given the wealth of results about graph factors in the random graph process, it is very natural to consider this property in the budget-constrained random graph process.
For simplicity, consider here the particular case that $F=K_r$ with $r\geq2$.
For the property that $B_t$ contains a (partial) $K_r$-factor, we prove the following result for any successful $(t,b)$-strategy.

\begin{theorem}\label{thm:factors_lower}
    Let $r\geq2$ and $\alpha\in(0,1]$.
    Let $t=t(n)$ and $b=b(n)$ be such that there exists a successful $ (t, b) $-strategy for an $\alpha$-$ K_r $-factor.
    Then we must have $ b = \Omega(n^{r-1}/t^{r/2 - 1}) $.
\end{theorem}

In particular, ``close'' to the hitting time $\tau_{K_r}$ (that is, if $ t \le n^{2 - 2/r + o(1)} $), the budget cannot be improved by a polynomial factor from the trivial strategy (that is, $ b \ge n^{2 - 2/r - o(1)} $).
Moreover, when considering $\log_n b$ as a function of $\log_n t$, we observe a linear interpolation between the trivial strategy at the hitting time and the trivial strategy that uses linear budget with quadratic time (namely, buy only edges belonging to a fixed $K_r$-factor); see \cref{fig:Krfactors} for a depiction of this behaviour.
For $r\geq3$, this provides a negative answer to \cref{question:FKMgeneral} in a strong sense: for some graphs $H$, a linear budget cannot suffice unless $t$ is quadratic.

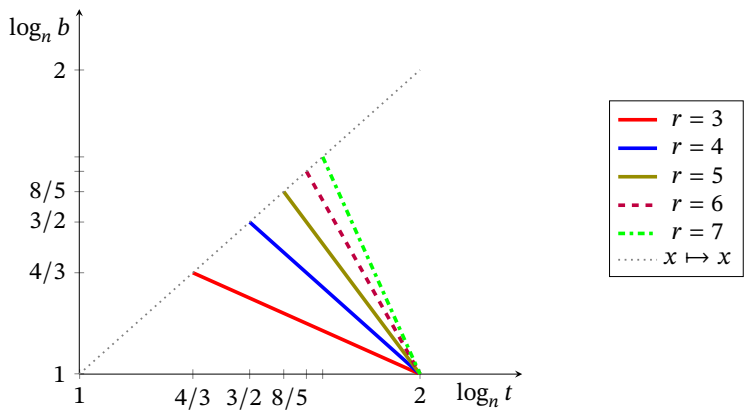
\begin{figure}
\begin{tikzpicture}[scale=0.85]
\begin{axis}[
    axis lines = middle,
    xlabel={$\log_{n}t$},
    ylabel={$\log_{n}b$},
    label style = {below left},
    legend style={at={(1.5,0.75)}},
    xmin=0, xmax=1.3, 
    xtick={0.001,1/3,1/2,3/5,2/3,5/7,1},
    xticklabels={$1$, $4/3$, $3/2\ \ $, $\ \ 8/5$, , , $2$},
    ymin=0, ymax=1.2,
    ytick={0.001,1/3,1/2,3/5,2/3,5/7,1},
    yticklabels={$1$, $4/3$, $3/2$, $8/5$, , , $2$}]
\addplot[red, ultra thick, domain=1/3:1] {1-(x+1)/2};
\addplot[blue, ultra thick, domain=1/2:1] {1-x};
\addplot[olive, ultra thick, domain=3/5:1] {3-3*(x+1)/2};
\addplot[purple, ultra thick, dashed, domain=2/3:1] {2-2*x};
\addplot[green, ultra thick, dashdotted, domain=5/7:1] {5-5*(x+1)/2};
\addplot[gray, thick, dotted, domain=0:1] {x};
\legend{$r=3$,$r=4$,$r=5$,$r=6$,$r=7$,$x\mapsto x$}
\end{axis}
\end{tikzpicture}
\caption{A depiction of the optimal budget $b$ for successful $(t,b)$-strategies for $K_r$-factors when $r\in\{3,4,5,6,7\}$, as follows from \cref{thm:factors_lower,thm:factor_upper}.}
\label{fig:Krfactors}
\end{figure}

In fact, our result is not limited to clique factors, and our method gives a lower bound for~$b$ in any successful strategy that creates (partial) $F$-factors, for any fixed graph $F$ (see \cref{thm:Ffactors_lower}).
This extends the negative answer to \cref{question:FKMgeneral} to a large class of graphs.
Indeed, let~$H$ be a bounded-degree graph which contains a (partial) $F$-factor for some fixed graph~$F$ satisfying the mild condition that its $1$-density satisfies $d^*(F)>1$; see \eqref{equa:1density_def}.\COMMENT{The reason why we need this condition is so that the bound we obtain on the budget is indeed superlinear.}
An application of \cref{thm:Ffactors_lower} then guarantees that there are no successful $(t,b)$-strategies for~$H$ with linear budget when $t$ is ``close'' to the hitting time for $H$, and in fact that a linear budget cannot suffice unless $t$ is quadratic.
We apply this method to powers of Hamilton cycles (see \cref{sect:intropowers}), and it could also be applied to other spanning structures, such as grids.

\Cref{thm:factors_lower} provides a lower bound on the budget of any successful $(t,b)$-strategy for containing $\alpha$-$K_r$-factors; we complement this by showcasing a successful $(t,b)$-strategy where $b$ coincides with this lower bound, up to a constant factor.
This shows that our results are tight.

\begin{theorem}\label{thm:partial_factor_upper}
    Let $r\geq2$ and $\alpha\in(0,1)$.
    For every $\delta>0$, there exists some $K>0$ such that, if $t=t(n)\geq Kn^{2-2/r}$ with $t\leq n^{2-\delta}$, there exists a $b=b(n)=O(n^{r-1}/t^{r/2-1})$ such that there exists a successful $(t,b)$-strategy for an $\alpha$-$K_r$-factor.
\end{theorem}

Moreover, we also consider complete $K_r$-factors (for which a lower bound follows from \cref{thm:factors_lower}).
In this case, our successful $(t,b)$-strategy uses a budget which is larger than the lower bound by only a polylogarithmic factor, thus proving essentially tight results in this case too.

\begin{theorem}\label{thm:factor_upper}
    Let $r\geq2$ and $n$ be a multiple of $r$.
    For every $\delta>0$, there exists some $K>0$ such that, if $t=t(n)\geq Kn^{2-2/r}\log^{1/\binom{r}{2}}n$ with $t\leq n^{2-\delta}$, there exists a $b=b(n)=O(n^{r-1}\log^{1/(r-1)}n/t^{r/2-1})$ such that there exists a successful $(t,b)$-strategy for a $K_r$-factor.
\end{theorem}

Similarly to the lower bounds, our upper bounds are in fact more general, and hold for (partial) $F$-factors for large families of graphs $F$.
See \cref{section:factors_upper} for the precise statements.

\subsection{Powers of Hamilton cycles}\label{sect:intropowers}

Recall that, for any integer $k\geq1$, the $k$-th power of a Hamilton cycle is obtained by adding an edge from each vertex to the $k$ vertices which precede it on the cycle.
For $k\geq2$, the threshold for containing the $k$-th power of a Hamilton cycle is of order $n^{-1/k}$.
For $k\geq3$, this follows from a general result of \citet{riordan00}, as observed by \citet{KO12}. 
For the case $k=2$, it took longer to determine the threshold.
\citet{KO12} showed that it is at most $n^{-1/2+o(1)}$.
Their result was subsequently improved by \citet{NS19} and by \citet{FSST22}.
Soon after, the threshold was determined by \citet{KNP21}.
Very recently, the sharp thresholds for powers of Hamilton cycles have been determined independently by \citet{MPPS25} and \citet{Z25}.
It is therefore natural to consider powers of Hamilton cycles in the budget-constrained random graph process.

As a corollary of \cref{thm:Ffactors_lower}, we can give a negative answer to \cref{question:FKMsquare}: we show in the following theorem that, for any successful $(t,b)$-strategy for the square of a Hamilton cycle, we must have that $b\geq n^{3-o(1)}/t$.

\begin{theorem} \label{thm:hamsquare_lower}
    Fix an integer $k\geq2$.
    Let $t=t(n)$ and $b=b(n)$ be such that there exists a successful $ (t, b) $-strategy for the $k$-th power of a Hamilton cycle.
    Then we must have $ b \geq n^{2k-1 - o(1)}/t^{k-1} $.
\end{theorem}

Since a.a.s.\ the hitting time for having the $k$-th power of a Hamilton cycle is of order $\Theta(n^{2-1/k})$, \cref{thm:hamsquare_lower} shows that, ``close'' to the hitting time, the budget cannot be improved by more than a subpolynomial factor from the trivial strategy.
Moreover, the behaviour of these lower bounds on the budget as $t$ increases is similar to the lower bound for clique factors (see \cref{fig:HamPowers}).

\begin{figure}
\begin{tikzpicture}[scale=0.85]
\begin{axis}[
    axis lines = middle,
    xlabel={$\log_{n}t$},
    ylabel={$\log_{n}b$},
    label style = {below left},
    legend style={at={(1.5,0.7)}},
    xmin=0, xmax=1.3, 
    xtick={0.001,1/2,2/3,3/4,4/5,1},
    xticklabels={$1$, $3/2$, $5/3$, , , $2$},
    ymin=0, ymax=1.2,
    ytick={0.001,1/2,2/3,3/4,4/5,1},
    yticklabels={$1$, $3/2$, $5/3$, , , $2$}]
\addplot[red, ultra thick, domain=1/2:1] {1-x};
\addplot[blue, ultra thick, domain=2/3:1] {2-2*x};
\addplot[olive, ultra thick, domain=3/4:1] {3-3*x};
\addplot[purple, ultra thick, dashed, domain=4/5:1] {4-4*x};
\addplot[gray, thick, dotted, domain=0:1] {x};
\legend{$k=2$,$k=3$,$k=4$,$k=5$,$x\mapsto x$}
\end{axis}
\end{tikzpicture}
\caption{A depiction of the optimal budget $b$ for successful $(t,b)$-strategies for the $k$-th power of a Hamilton cycle, $k\in\{2,3,4,5\}$, as given by \cref{thm:hamsquare_lower,thm:hamsquare_upper}.}
\label{fig:HamPowers}
\end{figure}
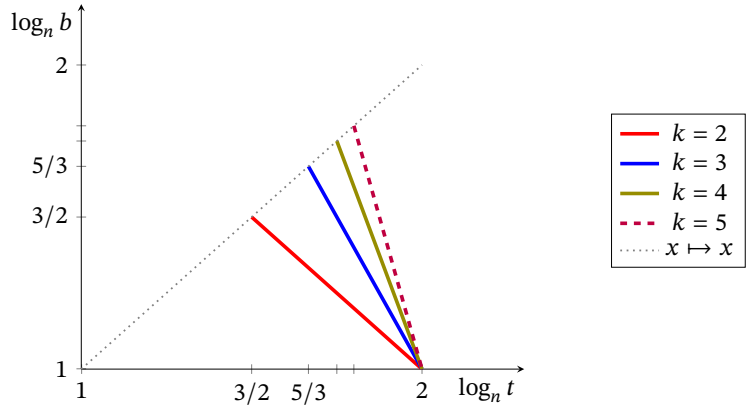

We also present a successful $(t,b)$-strategy for constructing powers of Hamilton cycles for parameters that almost match the lower bound in \cref{thm:hamsquare_lower}.

\begin{theorem}\label{thm:hamsquare_upper}
    Fix an integer $k\geq2$.
    If\/ $t=t(n)=\omega(n^{2-1/k})$, there exists a $b=n^{2k-1+o(1)}/t^{k-1}$ for which there is a successful $(t,b)$-strategy for the $k$-th power of a Hamilton cycle.
\end{theorem}

\subsection{General tools for the budget-constrained random graph process}\label{sect:toolsintro}

The budget-constrained random graph process is a very recent object of study, and we currently do not have many general tools to work specifically in this setting.
In particular, some strategies and ideas have been used without explicitly formalising the details involved in proving their success.
In this paper, we make two contributions to the general study of the model and towards this explicit formalisation.

We start with a simple observation about the nature of successful strategies.
When defining strategies in the budget-constrained random graph process, we noted that a strategy is a function which is allowed to make \emph{random} choices.
In \cref{sect:randomstrat}, we formalise our definition of a (random) $ (t, b) $-strategy, specifying the most general framework of the budget-constrained random graph process.
On the other hand, let us say, for now, that a strategy is \emph{deterministic} if it does not make use of any additional randomness (that is, external to the random graph process itself) in deciding whether to purchase a presented edge.
We will show that strategies that are allowed random choices, although seemingly stronger than deterministic strategies, do not give any real advantage to Builder.
In fact, our \cref{lem:deterministic} shows that, if there is any successful $ (t, b) $-strategy for property $ \mathcal{P} $, then there also exists a successful deterministic $ (t, b) $-strategy for $ \mathcal{P} $.
This may be useful for proving lower bounds, as it reduces the set of strategies that need to be analysed.

Our second contribution is to provide a formal approach to proving the success of our strategies.
Specifically, we consider \emph{multi-stage strategies}, in which the allocated time is split into intervals (of lengths $ t_i $), with the goal of purchasing a particular substructure in each stage, which may depend on the outcome of previous stages.
We want to treat the graph presented in each stage as behaving roughly like $ G(n, t_i / M) $, an idea which can be formalised using couplings.
However, when attempting to holistically prove the success of a multi-stage strategy, we cannot simply treat the stages as independent, as we implicitly condition upon having successfully built desired substructures in previous stages.
For example, it may be the case that some of the edges we would like to buy in a given stage have in fact appeared previously, and were not purchased.
We therefore need to show that this is not made substantially more likely by conditioning upon success in previous stages.
Overcoming these issues formally is one of the main technical challenges of this paper.

In \cref{sect:couplings}, we introduce a general tool to address this problem systematically for a particular family of strategies, by showing that the existence of desired structures in consecutive stages is in fact positively correlated.
Later, in \cref{sect:linking}, we exploit a different approach based upon the fact that, if $ t = o(M) $, the graph presented in any stage is very ``sparse'', in some appropriate sense, and so should not have a substantial effect on the distribution of the edges which are presented in subsequent stages.
A similar idea was used by \citet[Lemma~2.11]{KL25}.

\subsection{Organisation of the paper}

The remainder of this paper is organised as follows.
In \cref{section:prelims} we set our notation and list some standard probabilistic results which will be useful for us.
\Cref{sect:tools} is devoted to listing and proving several more auxiliary results that are key for our proofs, including the tools discussed in \cref{sect:toolsintro}.
We will prove our lower bounds on the budget of any successful $(t,b)$-strategy for graph factors (\cref{thm:factors_lower}) and powers of cycles (\cref{thm:hamsquare_lower}) in \cref{section:factors_low} as corollaries of a more general result (\cref{thm:lowerboundF}).
We prove our upper bounds for graph factors (\cref{thm:partial_factor_upper,thm:factor_upper}) in \cref{section:factors_upper}.
Lastly, \cref{thm:hamsquare_upper} is covered in \cref{sect:powers}.
We include intuition for the ideas behind each of the proofs in the relevant sections.
In \cref{sect:problems} we discuss several open problems and future lines of research.


\section{Preliminaries}\label{section:prelims}

\subsection{Notation}\label{sect:notation}

As usual in the area, for any $m\in\mathbb{R}$ we write $[m]\coloneqq\{i\in\mathbb{Z}:1\leq i\leq m\}$.
We will use standard $O$-notation for asymptotic comparisons.
When dealing with multiple parameters, we will sometimes use \emph{hierarchies} to introduce them.
When parameters are presented using a hierarchy, they are defined from right to left.
To be precise, if we say that a statement holds for $b$ sufficiently small with respect to $a$, denoted by $b\ll a$, we mean that for any $a>0$ there exists $b_0>0$ such that for any $0<b\leq b_0$ the subsequent statement holds.
This extends in a natural way to hierarchies involving more parameters.
If a parameter in a hierarchy is expressed as a fraction $1/a$, it should be interpreted that $a\in\mathbb{N}$.
To simplify the exposition, whenever we consider asymptotic statements we ignore rounding and treat reals as integers if this does not affect the validity of our claims.
Given the asymptotic nature of our results, we will often assume that $n$ is sufficiently large for our inequalities to hold without explicitly mentioning this. 
All our logarithms use base $\nume$, unless stated otherwise.

Most of our graph-theoretic notation is standard.
Given a graph $G$, we denote its number of vertices and of edges by $v(G)$ and $e(G)$, respectively.
We often identify graphs with their edge sets; for instance, given two graphs $G$ and $H$ (on not necessarily the same vertex sets), the notation $G\setminus H$ refers to the graph obtained from $G$ by deleting all edges of $H$.
Given a set $A\subseteq V(G)$, we will write~$G[A]$ to denote the subgraph of $G$ \defi{induced} by~$A$, that is, the graph on vertex set $A$ whose edges are all edges of $G$ which are contained in~$A$.
Given also $B\subseteq V(G)\setminus A$, we write $G[A,B]$ to denote the \defi{induced bipartite subgraph} of $G$ between $A$ and~$B$, that is, the graph on vertex set $A\cup B$ whose edges are all those of~$G$ which have one endpoint in~$A$ and the other in~$B$.
Given a vertex $v\in V(G)$, we write $N_G(v)\coloneqq\{u\in V(G):uv\in E(G)\}$ for the \defi{neighbourhood} of $v$ in $G$.
We define the \defi{neighbourhood} of a set $A\subseteq V(G)$ as $N_G(A)\coloneqq\left(\bigcup_{v\in A}N_G(v)\right)\setminus A$.

Given a graph $F$ on at least two vertices, we write 
\begin{equation}\label{equa:1density_def}
    d(F)\coloneqq \frac{e(F)}{v(F)-1}\qquad\text{ and }\qquad d^*(F)\coloneqq\max\{d(F'):F'\subseteq F, v(F')\geq2\}
\end{equation}
for the \defi{$1$-density} and the \defi{maximum $1$-density} of $F$, respectively.
We say that $F$ is \defi{strictly $1$-balanced} if for every $F'\subsetneq F$ with at least two vertices we have $d(F')<d(F)$.
For simplicity, we will assume throughout that all our fixed graphs $F$ are non-empty (which implies that $d^*(F)>0$).
We will sometimes want to count the number of copies of $ F $ in another graph $ G $; by this, we always mean labelled copies, that is, injective graph homomorphisms.

In some cases, we will partition the vertex set $[n]$ into smaller subsets and aim to find some spanning structure in each of these subsets.
We say that such a partition is an \defi{equipartition} into~$k$ parts if it is a partition into $k$ sets of sizes as equal as possible (that is, all subsets have size $\lfloor n/k\rfloor$ or~$\lceil n/k\rceil$).
In some occasions (say, if we are trying to construct an $F$-factor), we will require that the size of each set of the partition satisfies some divisibility conditions (in particular, that it is divisible by $v(F)$).
For simplicity, we will say that a partition of $[n]$ is an \defi{$F$-equipartition} into $k$ sets if it is a partition into $k$ sets, each of size divisible by $v(F)$, and such that the sizes of the sets are as equal as possible (that is, they are all equal to one of two values, which have difference $v(F)$).

For the budget-constrained random graph process, when considering $(t,b)$-strategies, it will often be useful to consider the ``density'' of the graphs that can be built, both in terms of the budget and the time.
Given a budget $ b $, we will often write $ d \coloneqq 2 b / n $ for the (maximum allowed) average degree of~$ B_t $.
Intuitively, we think of $ d = d(n) $ as the largest degree which ``most'' vertices can achieve.
Similarly, we write $ p \coloneqq t / M $ and think of $ p = p(n) $ as the probability that any given edge will be presented at some point throughout the process; this intuition can be formalised using couplings with $ G(n,p) $ (see \cref{sect:couplings}).

We will denote vectors using boldface type (for example, we may write $\vect{x}=(x_1,\ldots,x_k)$).
As previously discussed, a strategy is a function that takes all the information available to Builder at a given time and outputs a decision of whether to purchase the next presented edge or not.
It will be convenient to formalise this notion of ``available information''.
Recall that we denote the random edge presented to Builder at time $i$ by $e_i$.
Formally, when Builder is presented the edge $e_i$, the available information consists of a sequence of presented edges $\vect{e}^{(i)} \in E^{(i)} \coloneqq \{ \vect{e} \in \inbinom{[n]}{2}^i : e_j \ne e_{j'} \text{ for all } j \ne j' \} $, and a binary sequence $ \vect{b}^{(i - 1)} \in \{ 0, 1 \}^{i - 1} $ indicating which of these edges have been purchased.
Builder must make a binary decision as to whether the incoming edge $ e_i $ should be purchased, which determines the value of $b_i$ (we shall further formalise this in \cref{sect:randomstrat}).
Based on this, we refer to a pair $(\vect{e}^{(i)},\vect{b}^{(i - 1)})$ as a \defi{decision history} at time $i$.
In some parts of our proofs, it is useful to define sets \emph{after} Builder has made a decision, but \emph{before} the next edge has been presented. 
In these cases, it will be more useful to consider the \defi{update history} at time $i$, which refers to a pair of the form $(\vect{e}^{(i)},\vect{b}^{(i)})$ for some $i\in[t]$.


\subsection{Probabilistic tools}

We will often need to show that a.a.s.\ certain random variables are ``close'' to their expectation.
For this, we will use the following standard version of Chernoff's bound (see, e.g.,~\cite[Corollary~2.3 and Theorem~2.8]{JLR}).

\begin{lemma}[Chernoff's bound]\label{lem:Chernoff}
    Let\/ $X$ be a sum of mutually independent Bernoulli random variables, and let $\mu\coloneqq\mathbb{E}[X]$.
    Then, for all\/ $0<\delta<1$, we have that\/
    $\mathbb{P}[|X-\mu|\geq\delta\mu]\leq2\nume^{-\delta^2\mu/3}$.
\end{lemma}

The simplest case of Janson's inequality will also be useful for this purpose.
Given $ n \in \mathbb{N} $, $ r \in [2^{\binom{n}{2}}] $ and $ p\in[0,1] $, let $ F_1, \ldots, F_r \subseteq K_n $ be distinct subgraphs and let $ X \coloneqq \sum_{i = 1}^r \mathds{1}(F_i \subseteq G(n,p)) $ be the number of $ F_i $ which are contained in the binomial random graph $ G(n,p) $.
For $i,j\in[r]$, say that $ i \sim j $ if $ i \ne j $ and $ e(F_i \cap F_j) > 0 $, and write 
\[\Delta \coloneqq \frac{1}{2} \sum_{i\in[r]}\sum_{j \sim i} \mathbb{P}[F_i \cup F_j \subseteq G(n,p)].\]
The following inequality bounds the probability that $X$ is (much) smaller than its expectation (see, e.g.,~\cite[Theorem~2.14]{JLR}).

\begin{lemma}[Janson's inequality] \label{lem:janson}
    Let $X$ and $ \Delta $ be as above, and write $ \mu \coloneqq \mathbb{E}[X] $ and $ \overline{\Delta} \coloneqq \mu + 2 \Delta $.
    Then, for any $ \lambda \in [0, \mu] $, we have that $ \mathbb{P}[X \le \mu - \lambda] \le \nume^{-\lambda^2 / 2 \overline{\Delta}}$.
\end{lemma}

We will also make use of the following well-known correlation inequality, generally attributed to \citet{FKG71} (see also~\cite[Chapter~6]{AS16}).

\begin{lemma}[FKG inequality] \label{lem:fkg_inequality}
    Let $ p \in [0, 1] $, let $ \mathcal{G} $ denote the set of all graphs on $ [n] $, and
    let $ G \sim G(n, p) $.
    Suppose $ f, g\colon \mathcal{G} \to \mathbb{R} $ are increasing functions, that is, if $ G_1 \subseteq G_2 $, then $ f(G_1) \le f(G_2) $ and $ g(G_1) \le g(G_2) $.
    Then
    \[\mathbb{E}[f(G) g(G)] \ge \mathbb{E}[f(G)] \mathbb{E}[g(G)].\]
\end{lemma}


\section{Tools for the budget-constrained random graph process}\label{sect:tools}

In this section we introduce several tools which are useful for proving results in the budget-constrained random graph process.
In \cref{sect:randomstrat} we prove a result (\cref{lem:deterministic}) which is particularly useful for simplifying proofs of lower bounds for budgets of successful strategies.
In \cref{sect:couplings}, we introduce two useful tools for the analysis of multi-stage strategies (\cref{lemma:coupling_stages,lem:two_stage_real}), which we use in the proofs of our upper bounds on the optimal budget for successful strategies.

\subsection{Deterministic vs random strategies}\label{sect:randomstrat}

We previously mentioned that a strategy for the budget-constrained random graph process on $[n]$ is a function which is allowed to make random choices.
This may be interpreted in several ways; let us formalise our definition of a (random) $ (t, b) $-strategy.
Given a time $ i \in [t] $ and a decision history $(\vect{e}^{(i)},\vect{b}^{(i - 1)})$ (recall that here $\vect{e}^{(i)}$ is a sequence of presented edges $ \vect{e}^{(i)} \in E^{(i)} \coloneqq \{ \vect{e} \in \inbinom{[n]}{2}^i : e_j \ne e_{j'} \text{ for all } j \ne j' \} $ and $\vect{b}^{(i - 1)} \in \{ 0, 1 \}^{i - 1}$ is a binary sequence indicating which of these edges have been purchased), Builder must make a binary decision as to whether the incoming edge $ e_i $ should be purchased.
This decision may be a random one, but she cannot have access to any information about the future of the random graph process, and the sequence~$ \vect{e}^{(i)} $ has already been revealed; as such, we may assume that she simply outputs a number $ p_i \in [0, 1] $, and that the edge $e_i$ is then purchased with probability~$p_i$ \emph{independently of everything else in the process}.
In other words, we assume that Builder defines $\vect{b}^{(i)}$ by appending to~$\vect{b}^{(i-1)}$ an independent Bernoulli random variable of parameter~$p_i$ (whose outcome is revealed before proceeding to time~$i+1$).
With this in mind, we may write 
\[D \coloneqq \left\{ \big(i, \vect{e}^{(i)}, \vect{b}^{(i - 1)}\big) : i \in [t], \vect{e}^{(i)} \in E^{(i)}, \vect{b}^{(i - 1)} \in \{ 0, 1 \}^{i - 1},|\vect{b}^{(i - 1)}|<b \right\}\]
for the domain, and say that a \defi{$ (t, b) $-strategy} is a function $ \mathcal{S} \colon D \to [0, 1] $.
Recall from the introduction that, in a $(t,b)$-strategy, Builder's decisions are made subject to the constraint that $e(B_t)\leq b$; in other words, if at any time $i\in[t]$ we have a decision history $(\vect{e}^{(i)}, \vect{b}^{(i - 1)})\in E^{(i)}\times\{0,1\}^{i-1}$ such that $|\vect{b}^{(i - 1)}|=b$, we know that Builder will not accept any more edges.
This is the reason why we have the constraint $|\vect{b}^{(i - 1)}|<b$ in the definition of the domain for the strategy.
We say that a strategy $ \mathcal{S} $ is \defi{deterministic} if in fact $ \mathcal{S}(D) \subseteq \{ 0, 1 \} $.
Note also that we have only defined strategies for a specific value of~$n$.
However, we often abuse notation and refer to a sequence $ (\mathcal{S}_n)_{n\in\mathbb{N}} $ of strategies simply as ``a strategy''.

Although allowing for random decisions yields a larger family of strategies, there is convincing intuition to suggest that this does not in fact allow for any more efficiency.
Specifically, at each step, given the available information, since one of the two options should increase the probability that the budget-constrained random graph process under $\mathcal{S}$ succeeds, any optimal strategy $\mathcal{S}$ may as well make a deterministic decision.
The following lemma formalises this idea that, without loss of generality, we need only consider deterministic strategies.
Given a strategy~$\mathcal{S}$, let us write $ \mathcal{C}(\mathcal{S}) $ for the event that $ \mathcal{S} $ succeeds, that is, that it produces a graph with the desired property.

\begin{lemma} \label{lem:deterministic}
    Let $ \mathcal{S} $ be a $ (t, b) $-strategy for a property $ \mathcal{P} $.
    Then there exists a deterministic $(t,b)$-strategy $\mathcal{S}'$ such that
    \[ \mathbb{P}[\mathcal{C}(\mathcal{S}')] \ge \mathbb{P}[\mathcal{C}(\mathcal{S})] .\]
    In particular, if there exists a successful $ (t, b) $-strategy for $ \mathcal{P} $, then there also exists a successful deterministic $ (t, b) $-strategy for $ \mathcal{P} $.
\end{lemma}

\begin{proof} 
    Fix $ n \in \mathbb{N} $ and $t,b\in[M]$.
    Setting
    \[D^{(i)} \coloneqq \left\{ \big(j, \vect{e}^{(j)}, \vect{b}^{(j - 1)}\big)\in D : j \in [i] \right\}\]
    for any $i\in[t]$, we say that a $ (t, b) $-strategy $ \mathcal{S} $ is \defi{deterministic up to time $i$} if $ \mathcal{S}(D^{(i)}) \subseteq \{ 0, 1 \} $.
    In particular, a $(t,b)$-strategy is deterministic if it is deterministic up to time $t$.
    Given now an arbitrary $ (t, b) $-strategy $ \mathcal{S} $ for property $ \mathcal{P} $, our goal is to define a sequence of $(t,b)$-strategies $ \mathcal{S}_0, \mathcal{S}_1, \ldots, \mathcal{S}_t $ with $ \mathcal{S}_0 \coloneqq \mathcal{S} $ such that, for all $i\in[t]$, we have that $ \mathbb{P}[\mathcal{C}(\mathcal{S}_i)] \ge \mathbb{P}[\mathcal{C}(\mathcal{S}_{i - 1})] $ and $ \mathcal{S}_i $ is deterministic up to time $ i $.
    It is clear then that $ \mathcal{S}' \coloneqq \mathcal{S}_t $ will be the required deterministic $(t,b)$-strategy.
    We construct our sequence iteratively in the following natural way.
    
    Let $i\in[t]$ and let us assume, inductively, that $\mathcal{S}_{i-1}$ is deterministic up to time $i-1$.
    For all $(j, \vect{e}^{(j)}, \vect{b}^{(j - 1)})\in D$ with $ j \in [t] \setminus \{ i \} $, set $ \mathcal{S}_i(j, \vect{e}^{(j)}, \vect{b}^{(j - 1)}) \coloneqq \mathcal{S}_{i - 1}(j, \vect{e}^{(j)}, \vect{b}^{(j - 1)}) $, so it only remains to define $ \mathcal{S}_i(i, \vect{e}^{(i)}, \vect{b}^{(i - 1)}) $ for all $\vect{e}^{(i)}\in E^{(i)}$ and all $\vect{b}^{(i - 1)}\in\{0,1\}^{i-1}$ with $|\vect{b}^{(i - 1)}|<b$.
    Note that, since $\mathcal{S}_{i-1}$ is deterministic up to time $i-1$, for each possible $\vect{e}^{(i)}\in E^{(i)}$ there is a unique $\vect{b}^{(i-1)}\in\{0,1\}^{i-1}$ that corresponds to the decisions made by Builder during the first $i-1$ time steps if she is presented with the edges of $\vect{e}^{(i)}$; let us denote this unique $\vect{b}^{(i-1)}$ by $\vect{b}(\mathcal{S}_{i-1},\vect{e}^{(i)})$ (note that this unique $\vect{b}^{(i-1)}$ in fact only depends on the first $i-1$ entries of $\vect{e}^{(i)}$, even though this is not explicit in the notation).
    Given any $\vect{e}^{(i)}\in E^{(i)}$ and any $\vect{b}^{(i - 1)}\in\{0,1\}^{i-1}\setminus\{\vect{b}(\mathcal{S}_{i-1},\vect{e}^{(i)})\}$ with $|\vect{b}^{(i - 1)}|<b$, since we know the strategy will never be queried for the value of $\mathcal{S}_i(i,\vect{e}^{(i)},\vect{b}^{(i - 1)})$, we may simply set $\mathcal{S}_i(i,\vect{e}^{(i)},\vect{b}^{(i - 1)})\coloneqq0$ without altering the probability of success of the strategy.
    Now we must only consider the domain elements of the form $(i,\vect{e}^{(i)}, \vect{b}(\mathcal{S}_{i-1},\vect{e}^{(i)}))$.
    Recall that, if $|\vect{b}(\mathcal{S}_{i-1},\vect{e}^{(i)})|\geq b$, then we are outside the domain and there is nothing to define, so assume otherwise.
    If $ \mathcal{S}_{i - 1}(i, \vect{e}^{(i)}, \vect{b}(\mathcal{S}_{i-1},\vect{e}^{(i)})) \in \{ 0, 1 \} $, then we may set $ \mathcal{S}_i(i, \vect{e}^{(i)}, \vect{b}(\mathcal{S}_{i-1},\vect{e}^{(i)})) \coloneqq \mathcal{S}_{i - 1}(i, \vect{e}^{(i)}, \vect{b}(\mathcal{S}_{i-1},\vect{e}^{(i)})) $.
    Otherwise, let $ \mathcal{A}(\vect{e}^{(i)}) $ be the event that the random graph process outputs $\vect{e}^{(i)}$ up to time $ i $ and, for each $ j \in \{ 0, 1 \} $, let 
    \[p_j \coloneqq \mathbb{P}[\mathcal{C}(\mathcal{S}_{i - 1}) \mid \mathcal{A}(\vect{e}^{(i)}), b_i = j]\]
    be the probability that $ \mathcal{S}_{i - 1} $ succeeds depending on whether the $ i $-th edge is bought or not, and choose $ j^* \in \{0, 1\} $ to maximise this quantity, that is, such that $ p_{j^*} \ge p_{1 - j^*} $.
    Now, simply set $ \mathcal{S}_i(i, \vect{e}^{(i)}, \vect{b}(\mathcal{S}_{i-1},\vect{e}^{(i)})) \coloneqq j^* $.
    Since $ \mathcal{S}_{i - 1} $ is deterministic up to time $ i - 1 $, this definition ensures that~$ \mathcal{S}_i $ is deterministic up to time $ i $.

    Lastly, we verify the desired bounds on the probability of success.
    Fix any $i\in[t]$ and consider any $\vect{e}^{(i)}\in E^{(i)}$.
    If $\mathcal{S}_{i-1}(i, \vect{e}^{(i)}, \vect{b}(\mathcal{S}_{i-1},\vect{e}^{(i)}))\in\{0,1\}$, then clearly
    \[\mathbb{P}\left[\mathcal{C}(\mathcal{S}_i) \ \big|\  \mathcal{A}(\vect{e}^{(i)})\right] = \mathbb{P}\left[\mathcal{C}(\mathcal{S}_{i-1}) \ \big|\  \mathcal{A}(\vect{e}^{(i)})\right].\]
    Otherwise, by definition, we have that
    \begin{align*}
        \mathbb{P}\left[\mathcal{C}(\mathcal{S}_{i - 1}) \ \big|\  \mathcal{A}(\vect{e}^{(i)})\right] &= \mathcal{S}_{i - 1}(i, \vect{e}^{(i)}, \vect{b}(\mathcal{S}_{i-1},\vect{e}^{(i)})) \mathbb{P}\left[\mathcal{C}(\mathcal{S}_{i - 1}) \ \big|\  \mathcal{A}(\vect{e}^{(i)}), b_i = 1\right]\\
        &\phantom{=}+\left(1 - \mathcal{S}_{i - 1}(i, \vect{e}^{(i)}, \vect{b}(\mathcal{S}_{i-1},\vect{e}^{(i)}))\right) \mathbb{P}\left[\mathcal{C}(\mathcal{S}_{i - 1}) \ \big|\  \mathcal{A}(\vect{e}^{(i)}), b_i = 0\right]\\
        &\le \mathbb{P}\left[\mathcal{C}(\mathcal{S}_{i - 1}) \ \big|\  \mathcal{A}(\vect{e}^{(i)}), b_i = j^*\right]\\
        &= \mathbb{P}\left[\mathcal{C}(\mathcal{S}_i) \ \big|\  \mathcal{A}(\vect{e}^{(i)})\right].
    \end{align*}
    As this holds for all $ \vect{e}^{(i)} \in E^{(i)} $, it follows that $ \mathbb{P}[\mathcal{C}(\mathcal{S}_i)] \ge \mathbb{P}[\mathcal{C}(\mathcal{S}_{i - 1})] $ for each $ i \in [t] $, as required.
\end{proof}


\subsection{Couplings and stages}\label{sect:couplings}

As mentioned in \cref{sect:toolsintro}, each $ G_i $ in the random graph process can be thought of intuitively as behaving similarly to a binomial random graph with the appropriate expected number of edges.
It is often much easier to prove a statement for binomial random graphs (leveraging the independence between different edges) than it is to prove it (directly) in the random graph process, and thus it is very useful to be able to formalise this intuition.
In this section, we discuss several ways of doing so, particularly in the context of \emph{multi-stage} strategies.
Despite elaborate notation, the lemmas we prove capture quite intuitive ideas, but are imperative to be able to analyse our strategies in a formally correct way.

Recall firstly that the $m$-th graph of the random graph process, $G_m$, is distributed like the uniform random graph $\hat{G}(n,m)$.
The following standard coupling between binomial and uniform random graphs (which is a particular case of the more general \cref{lemma:coupling_stages}) states that a uniform random graph can be ``sandwiched'' between two binomial random graphs with similar parameters.

\begin{lemma}\label{lem:basiccoupling}
    Let $m\in[M]$ with $m=\omega(1)$.
    There exist $p=(1-o(1))m/M$ and $p'=(1+o(1))m/M$ and a coupling of random graphs $(H,\hat{G},H')$ such that $H\sim G(n,p)$, $\hat{G}\sim \hat{G}(n,m)$, $H'\sim G(n,p')$ and a.a.s.\ $H\subseteq \hat{G}\subseteq H'$.
\end{lemma}

In the budget-constrained random graph process, for many properties, the most straightforward way to construct and analyse a strategy is by dividing it into \emph{stages}.
Specifically, we partition the time into $ k $ intervals (which correspond to ``segments'' of the random graph process), and in the $ i $-th interval we purchase only edges belonging to a particular set $ E_i $, attempting to build one of a particular set $ \mathcal{F}_i $ of structures.
A strategy of this form may be adaptive, in the sense that, while $ E_i $ and $ \mathcal{F}_i $ must both be specified before the corresponding interval, they may depend on the outcome of previous stages.
For the sake of example, consider the following simple (and meaningless) two-stage strategy.
In the first stage, Builder attempts to construct a path of length at least $ (2/3 - \eps) n $ within the vertex set $ V_1 \coloneqq [2n/3]$, purchasing any edges presented in the set $ V_1 $.
Builder then fixes one such path $ P $ (assuming that one exists).
In the second stage, she now attempts to find a triangle factor on the remaining (at most) $ (1/3 + \eps) n $ vertices, purchasing only edges which are vertex-disjoint from $ P $.
Applying \cref{lem:basiccoupling} to the underlying random graph process is not so helpful in this case, because the fact that an edge appears in $ H $ tells us nothing about the stage of the process in which it is presented.
For example, even if a triangle factor on the appropriate vertex set (of size at most $ (1/3 + \eps) n $) appears in the graph $ G_t $, it may be that some of its edges were presented during the first stage and were not purchased.
For this reason, it is often useful to be able to regard the two stages as independent binomial random graphs.

We start by introducing a general context for multi-stage strategies, which we will use for the remainder of this section.
Let $k\in\mathbb{N}$, the number of stages, be fixed.
Let $ t_1, \ldots, t_k \in [M] $, the amounts of time allocated to each stage, be such that $\sum_{i=1}^k t_i\leq M$.
Let $s_0\coloneqq 0$ and, for each $ i \in [k] $, set $ s_i \coloneqq \sum_{j=1}^i t_j $.
Let $ G_0, G_1, G_2, \ldots $ denote a random graph process on $ [n] $ and, for each $ i \in [k] $, let $ \hat{G}_i \coloneqq G_{s_i} \setminus G_{s_{i - 1}} $ be the graph of all edges presented during the $ i $-th stage.
Observe that each $ \hat{G}_i $ has the same (marginal) distribution as $ \hat{G}(n, t_i) $, but that they are not independent (as they are pairwise disjoint).

The following general coupling lemma allows us to sandwich each of the stages $\hat{G}_1,\ldots,\hat{G}_k$ between coupled binomial random graphs analogously to \cref{lem:basiccoupling}.
The key difference is that, for all but the first stage, we must remove a few edges from the ``lower bound''; specifically, those edges of the previous stages of the random graph process which also happen to appear in the independent binomial random graph (since they cannot reappear in $ \hat{G}_i $).
Despite this technical detail, the proof of this lemma is still fairly standard.
For completeness, we include the proof in \cref{appen:multi_stage}.

\begin{lemma} \label{lemma:coupling_stages}
    Suppose that $t_i=\omega(1)$ and $t_i=o(M)$ for every $i\in[k]$.
    Then for each $i\in[k]$ there exist $p_i=(1 - o(1)) t_i / M$ and $\mybar{p}_i=o(t_i / M)$, and there exists a coupling of random graphs $ (H_i, \hat{H}_i, \mybar{H}_i)_{i \in [k]} $, such that
    \begin{enumerate}[label={\ensuremath{(\mathrm{\roman*})}}]
        \item\label{lem:couplingproperty1} $ H_i \sim G(n, p_i) $ and $ \mybar{H}_i \sim G(n, \mybar{p}_i) $ for each $i\in[k]$;
        \item\label{lem:couplingproperty1.5} $(\hat{H}_i)_{i\in[k]}$ is distributed like $(\hat{G}_i)_{i\in[k]}$;
        \item\label{lem:couplingproperty2} the graphs $\{H_i,\mybar{H}_i:i\in[k]\}$ are mutually independent, and
        \item\label{lem:couplingproperty3} a.a.s.\ $ H_i\setminus\bigcup_{j=1}^{i-1}\hat{H}_j \subseteq \hat{H}_i \subseteq H_i\cup \mybar{H}_i $ for all $ i \in [k] $. 
    \end{enumerate}
\end{lemma}

It turns out that, for rather subtle reasons, even this coupling lemma will not be sufficient to formally analyse our strategy in the proof of \cref{thm:hamsquare_upper}, and as such we require a stronger formal statement for a particular class of multi-stage strategies.
We introduce this more specific setup and prove such a statement in the remainder of this section.

Consider again the example two-stage strategy discussed above, and apply \cref{lemma:coupling_stages}.
Assume for now that we can prove that a.a.s.\ $ H_1 $ contains some long path as desired in the first stage, and (separately) that, for any fixed long path $ P $, a.a.s.\ $ H_2 \setminus (H_1 \cup \mybar{H}_1) $ (another binomial random graph) contains some triangle factor $ F $ which is vertex-disjoint from $P$, as desired in the second stage.
Then a.a.s.\ $ \hat{G}_1 $ contains some long path~$P$ and, for any given long path $ Q $, a.a.s.\ $ \hat{G}_2 $ contains some~$F$ which ``fits''~$Q$.
It may seem at first glance that this is sufficient to prove that the strategy succeeds, but in fact we need to show that a.a.s.~$ \hat{G}_1 $ contains some long path~$P$, in concert with $ \hat{G}_2 $ containing some~$F$ which ``fits''~$P$ specifically.
Note that this is a different statement in general, since conditioning on obtaining any specific path $ P $ in $ \hat{G}_1 $ may give us information about the graph $ \hat{G}_2 $.\COMMENT{Observe that, if we could prove that a.a.s., for every long path $ Q $ (simultaneously), $ H_2 \setminus (H_1 \cup \mybar{H}_1) $ contains some $F$ which ``fits'' $Q$, then this would also suffice; however, this is a much more powerful statement about the binomial random graph, requiring a union bound over $ Q $.}
We instead want to show that a.a.s., after the first stage, Builder has some way of choosing a path $ P $ such that a.a.s.\ some factor which ``fits'' $ P $ appears in the second stage.
We devote the remainder of this section to formalising one way of doing this, in a somewhat specific setup with three stages.
While the particular statements we give are tailored to our application (namely the proof of \cref{thm:hamsquare_upper}), the techniques employed may be applicable more generally.
It would also be very interesting to prove a more general result in this direction (see \cref{problem:multi_stage}).

We work in the setting of a multi-stage strategy as described above with $ k = 3 $; that is, $ t_1, t_2, t_3 $ are given.
Let $ m = m(n) \in\mathbb{N}$ and, for each $ i \in [m] $, let $ r_i = r_i(n) \in\mathbb{N}$.
Let $ \mathcal{E}^1 = \{ E_1, \ldots, E_m \} $, and $ \mathcal{E}^2_i = \{ E_{i, 1}, \ldots, E_{i, r_i} \} $, for each $ i \in [m] $, be collections of graphs on vertex set~$ [n] $.
Further, for each $ i \in [m] $ and each $ j \in [r_i] $, let $ \mathcal{F}_{i, j} $ be a collection of graphs on vertex set~$ [n] $.
We think of $ \mathcal{E}^1 $ as the set of all possible structures which Builder would be happy to construct in the first stage, and for each fixed $ i \in [m] $, we think of $ \mathcal{E}_i^2 $ as the set of all possible structures which she would be happy to construct in the second stage, having constructed $ E_i $ in the first stage (in our example strategy, $\cE^1$ would be the collection $\{P_1,\dots,P_m\}$ of all paths of length at least $(2/3-\eps)n$ on $[2n/3]$ and $\cE^2_i$ would be the set of all triangle factors on vertex set~$[n]\setminus V(P_i)$).
Likewise, having constructed $ E_i $ in the first stage and $ E_{i, j} $ in the second, Builder aims to build some structure of $ \mathcal{F}_{i, j} $ in the third stage.
For our application, we may suppose further that all possible structures for the first stage are edge-disjoint from all possible structures in the later stages, that is
\begin{equation} \label{eqn:zak6}
    K^2 \coloneqq \bigcup_{i \in [m]}\bigcup_{j \in [r_i]} \biggl( E_{i, j} \cup \bigcup_{F \in \mathcal{F}_{i, j}} F \biggr) \ \text{ is edge-disjoint from } \ K^1 \coloneqq \bigcup_{i \in [m]} E_i.
\end{equation}

Let $ \delta = \delta(n) = o(1) $, and assume that the following conditions hold whenever $ p_i = (1 - o(1)) t_i / M $ for each $ i \in [3] $ and $\mybar{p}_2 = o(t_2 / M) $.
\begin{enumerate}[label={\ensuremath{(\mathrm{P}\arabic*)}}]
    \item\label{cond:zak1} If $ G \sim G(n, p_1) $, with probability at least $ 1 - \delta $ there exists some $ i \in [m] $ for which $ E_i \subseteq G $.
    \item\label{cond:zak2} If $ G \sim G(n, p_2) $, then for each $ i \in [m] $, with probability at least $ 1 - \delta $, there exists some $ j \in [r_i] $ for which $ E_{i, j} \subseteq G $.\COMMENT{The reason why we introduce $\delta$ explicitly rather than stating things as simple `a.a.s.' statements is the following. In \ref{cond:zak2}, it could theoretically be possible that, for each $i\in[m]$, the probability of `success' is at least $1-\delta_i(n)$, where $\delta_i=o(1)$. However, this does not guarantee that $ \delta_{m(n)}(n) \not \to 0 $, which would be quite annoying to deal with. By explicitly defining a `global' $\delta$, we avoid this possible problem.}
    \item\label{cond:zak3} If $ G \sim G(n, p_3) \setminus (G(n, p_2) \cup G(n, \mybar{p}_2)) $,\COMMENT{Note that this is simply some $G(n,p^*)$ where $p^*=p_3(1-p_2-\mybar{p}_2+p_2\mybar{p}_2)$.} then for each $ i \in [m] $ and each $ j \in [r_i] $, with probability at least $ 1 - \delta $, there exists some $ F \in \mathcal{F}_{i, j} $ for which $ F\setminus E_{i,j} \subseteq G $.
\end{enumerate}

Let $ \mathcal{G} $ be the set of all graphs on vertex set $ [n] $, and let $ \mathcal{G}_1, \mathcal{G}_2 \subseteq \mathcal{G} $ be all those with exactly $ t_1 $ or~$ t_2 $ edges, respectively.
Let $ \mathcal{H} \coloneqq \{ (J_1, J_2) \in \mathcal{G}_1 \times \mathcal{G}_2: E(J_1)\cap E(J_2) = \varnothing \} $.
The following lemma formalises the idea that Builder can make suitable choices of structures found in the first two stages to a.a.s.\ obtain a suitable corresponding structure in the third stage.
This is the key lemma in formalising the proof of \cref{thm:hamsquare_upper}.

\begin{lemma} \label{lem:two_stage_real}
    Let $ (t_i)_{i \in [3]}$, $m$, $(r_i)_{i \in [m]}$, $\mathcal{E}^1$, $(\mathcal{E}^2_i)_{i \in [m]}$, $(\mathcal{F}_{i, j})_{i \in [m], j \in [r_i]} $, $K^1$ and $K^2$ be as in the setup described above.
    Suppose that for each $ i \in[3] $ and $ j \in [2] $ we have that $ t_i=\omega(M / e(K^j)) $.
    Then there exist functions $ f\colon \mathcal{G}_1 \to [m] $ and $ g\colon \mathcal{H} \to \mathbb{N} $ such that $ g(J_1, J_2) \in [r_{f(J_1)}] $ for all $ (J_1, J_2) \in \mathcal{H} $ and the following holds.
    Letting $ i \coloneqq f(\hat{G}_1) $ and $ j \coloneqq g(\hat{G}_1, \hat{G}_2) \in [r_i] $, we have that a.a.s.\ $ E_i \subseteq \hat{G}_1 $, $ E_{i, j} \subseteq \hat{G}_2 $, and there exists $ F \in \mathcal{F}_{i, j} $ for which $ F \setminus E_{i, j} \subseteq \hat{G}_3 $.
\end{lemma}

The proof consists of three parts.
Since \ref{cond:zak1}--\ref{cond:zak3} are statements about binomial random graphs, we first apply a coupling lemma to switch to this setting.
However, we also want to be able to exploit~\eqref{eqn:zak6}, which we achieve via the following stronger coupling lemma.

\begin{lemma} \label{lem:coupling_zak}
    Let $ K^1 $ and $ K^2 $ be edge-disjoint subgraphs of the complete graph $ K_n $, and suppose that for each $ i \in[3] $ and $ j \in [2] $ we have that $ t_i=\omega(M / e(K^j)) $.
    Then there exist $p_i=(1 - o(1)) t_i / M$ for each $i \in [3]$ and $\mybar{p}_2=o(t_2 / M)$, and there exists a coupling of random graphs $ ((H_i)_{i \in [3]}, (\hat{H}_i)_{i \in [3]}, \mybar{H}_2) $, satisfying the following properties.
    \begin{enumerate}[label={\ensuremath{(\mathrm{\roman*})}}]
        \item\label{cond:zak_coupling1} For each $i\in[3]$ we have that $ H_i \sim G(n, p_i) $, and $ \mybar{H}_2 \sim G(n, \mybar{p}_2) $.
        \item\label{cond:zak_coupling2} The sequence $(\hat{H}_i)_{i\in[3]}$ has the same distribution as $(\hat{G}_i)_{i\in[3]}$.
        \item\label{cond:zak_coupling3} The graphs $H_1,H_2,\mybar{H}_2,H_3$ are mutually independent.
        \item\label{cond:zak_coupling3_2} Let\/ $ X \subseteq \mathcal{G}\times\mathcal{G}\times\mathcal{G} $, and let\/ $ \Gamma_1, \Lambda_1 \in \mathcal{G} $ with $ e(\Gamma_1) = t_1 $ be such that 
        \[\mathbb{P}[\hat{H}_1 \cap K^1 = \Gamma_1 \cap K^1, H_1 = \Lambda_1]\neq 0.\]
        Then
        \begin{align}\label{cond:zak_coupling3_2.1}
            &\mathbb{P}\left[(H_2, \hat{H}_2, \hat{H}_3) \in X \ \middle| \ \hat{H}_1 \cap K^1 = \Gamma_1 \cap K^1, H_1 = \Lambda_1\right] \nonumber\\
            =\,& \mathbb{P}\left[(H_2, \hat{H}_2, \hat{H}_3) \in X \ \middle| \ \hat{H}_1 \cap K^1 = \Gamma_1 \cap K^1\right].
        \end{align}
        Similarly, let\/ $ Y \subseteq \mathcal{G} $, and let\/ $ \Gamma_1,\Gamma_2,\Lambda_2 \in \mathcal{G} $ be such that $ e(\Gamma_1) = t_1 $, $ e(\Gamma_2) = t_2 $, and
        \[\mathbb{P}[\hat{H}_1 \cap K^1 = \Gamma_1 \cap K^1, \hat{H}_2 = \Gamma_2, H_2 = \Lambda_2]\neq 0.\]
        Then
        \begin{align}\label{cond:zak_coupling3_2.2}
        &\mathbb{P}\left[\hat{H}_3 \in Y \ \middle| \ \hat{H}_1 \cap K^1 = \Gamma_1 \cap K^1, \hat{H}_2 = \Gamma_2, H_2 = \Lambda_2\right] \nonumber\\
        =\,& \mathbb{P}\left[\hat{H}_3 \in Y \ \middle| \ \hat{H}_1 \cap K^1 = \Gamma_1 \cap K^1, \hat{H}_2 = \Gamma_2\right].
        \end{align}
        \item\label{cond:zak_coupling4} A.a.s.\ we have that $ H_1 \cap K^1 \subseteq \hat{H}_1 \cap K^1 $, $ H_2 \cap K^2 \subseteq \hat{H}_2 \cap K^2 \subseteq (H_2 \cup \mybar{H}_2) \cap K^2 $, and $ (H_3 \cap K^2) \setminus (H_2 \cup \mybar{H}_2)  \subseteq \hat{H}_3 \cap K^2 $.
    \end{enumerate}
\end{lemma}

The idea for the coupling itself is standard and not dissimilar from \cref{lemma:coupling_stages}, but care must be taken in formally proving \ref{cond:zak_coupling3_2}.
Having proved \cref{lem:coupling_zak}, property \ref{cond:zak_coupling4} allows us to treat $ H_1 \cap K^1 $ completely independently from $ H_2 \cap K^2 $ and $ H_3 \cap K^2 $.
However, it is still the case that \ref{cond:zak2} and \ref{cond:zak3} are statements about independent binomial random graphs, whereas $ H_2 $ and $ H_3 \setminus (H_2 \cup \mybar{H}_2) $ are not independent.
To overcome this difficulty, the second part of the proof of \cref{lem:two_stage_real} is to prove the following lemma.

\begin{lemma} \label{lem:two_stage_indep}
    Let $ H_1 \sim G(n, p_1)$, $H_2 \sim G(n, p_2)$, $H_3 \sim G(n, p_3) $, and $ \mybar{H}_2 \sim G(n, \mybar{p}_2) $ be mutually independent.
    Let $ \mathcal{L}_1 $ be the event that there exists $ i \in [m] $ for which $ E_i \subseteq H_1 $, and in this case let $ i^* \in [m] $ be the minimum such $ i $.
    Let $ \mathcal{L}_2 $ be the event that $ \mathcal{L}_1 $ occurs, and there exists $ j \in [r_{i^*}] $ for which $ E_{i, j} \subseteq H_2 $, and in this case let $ j^* $ be the minimum such $ j $.
    Recall the function $ \delta $ from \ref{cond:zak1}--\ref{cond:zak3}.
    With probability at least $ 1 - 3 \delta $, the event $ \mathcal{L}_2 $ occurs and there exists $ F \in \mathcal{F}_{i^*, j^*} $ for which $ F \setminus E_{i^*, j^*} \subseteq H_3 \setminus (H_2 \cup \mybar{H}_2) $.
\end{lemma}

This roughly says that, even when we condition upon having found the correct structure in both of the binomial random graphs coupled to the first and second stages, which may in theory restrict the edges available in the third stage, we are still guaranteed to a.a.s.\ find the desired structure in the graph $ H_3 \setminus (H_2 \cup \mybar{H}_2) $.
Note that this is stronger than the properties \ref{cond:zak1}, \ref{cond:zak2} and \ref{cond:zak3} alone.

It is now intuitive that \cref{lem:two_stage_real} should follow from \cref{lem:coupling_zak,lem:two_stage_indep}; the formal arguments are given in the following proof.

\begin{proof}[Proof of \cref{lem:two_stage_real}]
    We start by working in the coupling of random graphs $((H_i)_{i\in[3]},(\hat{H}_i)_{i\in[3]},\mybar{H}_2)$ provided by \cref{lem:coupling_zak}, with disjoint edge sets as defined in \eqref{eqn:zak6}, and write $ H_2' \coloneqq H_2 \cup \mybar{H}_2 $.
    By \cref{lem:coupling_zak}~\ref{cond:zak_coupling2}, it suffices for us to derive the desired conclusion for $(\hat{H}_i)_{i\in[3]}$.
    By choosing $ \delta=\delta(n) $ suitably, we may ensure that all containments in \cref{lem:coupling_zak}~\ref{cond:zak_coupling4} hold with probability at least $ 1 - \delta $.
    Define a function $ f'\colon \mathcal{G} \to [m] $ by taking $ f'(J_1) $ to be the minimal $ i \in [m] $ for which $ E_i \subseteq J_1 $ (and $ f'(J_1) \coloneqq 1 $ if no such~$ i $ exists).
    Likewise, define a function $ g'\colon \mathcal{G} \times [m] \to \mathbb{N} $ by taking $ g'(J_2, i) $ to be the minimal $ j \in [r_i] $ such that $ E_{i, j} \subseteq J_2 $ (and, again, $ g'(J_2, i) \coloneqq 1 $ if no such $ j $ exists).

    For each $ i \in [m] $ and each $ j \in [r_i] $, let $ \mathcal{A}_i $ be the event that $ E_i \subseteq \hat{H}_1 $, let $ \mathcal{B}_{i, j} $ be the event that $ E_{i, j} \subseteq \hat{H}_2 $, and let $ \mathcal{C}_{i, j} $ be the event that there exists $ F \in \mathcal{F}_{i, j} $ such that $ F \setminus E_{i, j} \subseteq \hat{H}_3 $.
    Write
    \[ \mathcal{D}_{i, j} \coloneqq \mathcal{A}_i \cap \mathcal{B}_{i, j} \cap \mathcal{C}_{i, j}. \]
    Similarly, let $ \mathcal{A}_i' $ be the event that $ E_i \subseteq H_1 $, let $ \mathcal{B}_{i, j}' $ be the event that $ E_{i, j} \subseteq H_2 $, and let $ \mathcal{C}_{i, j}' $ be the event that there exists $ F \in \mathcal{F}_{i, j} $ such that $ F \setminus E_{i, j} \subseteq H_3 \setminus H_2' $.
    \Cref{lem:two_stage_indep} says exactly that
    \[ \mathbb{P}\left[\mathcal{A}_{f'(H_1)}' \cap \mathcal{B}_{f'(H_1), g'(H_2, f'(H_1))}' \cap \mathcal{C}_{f'(H_1), g'(H_2, f'(H_1))}'\right] \ge 1 - 3 \delta. \]
    By \cref{lem:coupling_zak}~\ref{cond:zak_coupling4}, this means that
    \begin{equation} \label{eqn:zak9}
        \mathbb{P}\left[\mathcal{D}_{f'(H_1), g'(H_2, f'(H_1))}\right] \ge 1 - 4 \delta.
    \end{equation}

    Given two graphs $ J, J' $ on vertex set $ [n] $, we write that $ J =_1 J' $ to mean that $ J \cap K^1 = J' \cap K^1 $.
    Note that $=_1$ defines an equivalence relation on $\mathcal{G}_1$, and let us denote by $\mathcal{G}_1/K^1$ a set of graphs containing one representative from each of the equivalence classes of this relation.
    
    Now define a function $ f\colon \mathcal{G}_1 \to [m] $ by taking $ f(J_1) $ to be any value of $ i \in [m] $ maximising the expression
    \[ \mathbb{P}\left[\mathcal{D}_{i, g'(H_2, i)} \ \middle|\ \hat{H}_1 =_1 J_1\right] \]
    (noting that this probability may be zero, for instance in the case that $ E_i \not \subseteq J_1 $ for every $ i \in [m] $).
    Similarly, define a function $ g\colon \mathcal{H} \to \mathbb{N} $ by taking $ g(J_1, J_2) $ to be any value of $ j \in [r_{f(J_1)}] $ maximising the expression
    \[ \mathbb{P}\left[\mathcal{D}_{f(J_1), j} \ \middle|\ \hat{H}_1 =_1 J_1, \hat{H}_2 = J_2\right] \]
    (noting that, again, this probability may be zero, such as in the case that $ E_{f(J_1), j} \not \subseteq J_2 $ for every $ j \in [r_{f(J_1)}] $).
    
    Given $ i \in [m] $ and $ J_1 \in \mathcal{G}_1 $, conditional upon $ \hat{H}_1 =_1 J_1 $, observe that $ g'(H_2, i) $ is defined only in terms of $ H_2 $, the occurrence of~$ \mathcal{A}_i $ is determined by $ J_1\cap K^1 $, and the occurrence of the events $ \mathcal{B}_{i, g'(H_2, i)} $ and $ \mathcal{C}_{i, g'(H_2, i)} $ is defined only in terms of $H_2$, $ \hat{H_2} $ and $ \hat{H}_3 $, whereas $ f'(H_1) $ is defined only in terms of~$ H_1 $.
    Therefore, by \eqref{cond:zak_coupling3_2.1} in \cref{lem:coupling_zak}~\ref{cond:zak_coupling3_2}, we see that\COMMENT{In case it helps, here is an expansion:
    \begin{align*}
        \mathbb{P}[\mathcal{D}_{i, g'(H_2, i)} \mid \hat{H}_1 =_1 G_1, f'(H_1) = i]
            &= \sum_{\substack{J_1 \in \mathcal{G}\\f'(J_1)=i}} \mathbb{P}[\mathcal{D}_{i, g'(H_2, i)} \mid \hat{H}_1 =_1 G_1, f'(H_1) = i, H_1 = J_1] \mathbb{P}[H_1 = J_1 \mid \hat{H}_1 =_1 G_1, f'(H_1) = i]\\
            &= \sum_{\substack{J_1 \in \mathcal{G}\\f'(J_1)=i}} \mathbb{P}[\mathcal{D}_{i, g'(H_2, i)} \mid \hat{H}_1 =_1 G_1, H_1 = J_1] \mathbb{P}[H_1 = J_1 \mid \hat{H}_1 =_1 G_1, f'(H_1) = i]\\
            &= \sum_{\substack{J_1 \in \mathcal{G}\\f'(J_1)=i}} \mathbb{P}[\mathcal{D}_{i, g'(H_2, i)} \mid \hat{H}_1 =_1 G_1] \mathbb{P}[H_1 = J_1 \mid \hat{H}_1 =_1 G_1, f'(H_1) = i]\\
            &= \mathbb{P}[\mathcal{D}_{i, g'(H_2, i)} \mid \hat{H}_1 =_1 G_1] \sum_{\substack{J_1 \in \mathcal{G}\\f'(J_1)=i}} \mathbb{P}[H_1 = J_1 \mid \hat{H}_1 =_1 G_1, f'(H_1) = i]\\
            &= \mathbb{P}[\mathcal{D}_{i, g'(H_2, i)} \mid \hat{H}_1 =_1 G_1].
    \end{align*}}
    \begin{equation} \label{eqn:zak7}
        \mathbb{P}\left[\mathcal{D}_{i, g'(H_2, i)} \ \middle|\  \hat{H}_1 =_1 J_1, f'(H_1) = i\right] = \mathbb{P}\left[\mathcal{D}_{i, g'(H_2, i)} \ \middle|\  \hat{H}_1 =_1 J_1\right].
    \end{equation}
    Using \eqref{eqn:zak7} and the definition of $ f $, we may compute\COMMENT{For completeness:
    \begin{align*}
        &\,\mathbb{P}\left[\mathcal{D}_{f'(H_1), g'(H_2, f'(H_1))}\right]\nonumber\\
            =&\, \sum_{J_1 \in \mathcal{G}_1/K^1} \sum_{i \in [m]} \mathbb{P}\left[\mathcal{D}_{i, g'(H_2, i)} \ \middle|\ \hat{H}_1 =_1 J_1, f'(H_1) = i\right] \mathbb{P}\left[f'(H_1) = i \ \middle|\ \hat{H}_1 =_1 J_1\right] \mathbb{P}\left[\hat{H}_1 =_1 J_1\right]\nonumber\\
            =&\, \sum_{J_1 \in \mathcal{G}_1/K^1} \sum_{i \in [m]} \mathbb{P}\left[\mathcal{D}_{i, g'(H_2, i)} \ \middle|\ \hat{H}_1 =_1 J_1\right] \mathbb{P}\left[f'(H_1) = i \ \middle|\ \hat{H}_1 =_1 J_1\right] \mathbb{P}\left[\hat{H}_1 =_1 J_1\right]\nonumber\\
            \le&\, \sum_{J_1 \in \mathcal{G}_1/K^1} \sum_{i \in [m]} \mathbb{P}\left[\mathcal{D}_{f(J_1), g'(H_2, f(J_1))} \ \middle|\ \hat{H}_1 =_1 J_1\right] \mathbb{P}\left[f'(H_1) = i \ \middle|\ \hat{H}_1 =_1 J_1\right] \mathbb{P}\left[\hat{H}_1 =_1 J_1\right]\nonumber\\
            =&\, \sum_{J_1 \in \mathcal{G}_1/K^1} \mathbb{P}\left[\mathcal{D}_{f(J_1), g'(H_2, f(J_1))} \ \middle|\ \hat{H}_1 =_1 J_1\right] \mathbb{P}\left[\hat{H}_1 =_1 J_1\right]\nonumber\\
            =&\, \mathbb{P}\left[\mathcal{D}_{f(\hat{H}_1), g'(H_2, f(\hat{H}_1))}\right].
    \end{align*}}
    \begin{align} \label{eqn:zak75}
        &\,\mathbb{P}\left[\mathcal{D}_{f'(H_1), g'(H_2, f'(H_1))}\right]\nonumber\\
            =&\, \sum_{J_1 \in \mathcal{G}_1/K^1} \sum_{i \in [m]} \mathbb{P}\left[\mathcal{D}_{i, g'(H_2, i)} \ \middle|\ \hat{H}_1 =_1 J_1, f'(H_1) = i\right] \mathbb{P}\left[f'(H_1) = i \ \middle|\ \hat{H}_1 =_1 J_1\right] \mathbb{P}\left[\hat{H}_1 =_1 J_1\right]\nonumber\\
            =&\, \sum_{J_1 \in \mathcal{G}_1/K^1} \sum_{i \in [m]} \mathbb{P}\left[\mathcal{D}_{i, g'(H_2, i)} \ \middle|\ \hat{H}_1 =_1 J_1\right] \mathbb{P}\left[f'(H_1) = i \ \middle|\ \hat{H}_1 =_1 J_1\right] \mathbb{P}\left[\hat{H}_1 =_1 J_1\right]\nonumber\\
            \le&\, \sum_{J_1 \in \mathcal{G}_1/K^1} \mathbb{P}\left[\mathcal{D}_{f(J_1), g'(H_2, f(J_1))} \ \middle|\ \hat{H}_1 =_1 J_1\right] \mathbb{P}\left[\hat{H}_1 =_1 J_1\right]\nonumber\\
            =&\, \mathbb{P}\left[\mathcal{D}_{f(\hat{H}_1), g'(H_2, f(\hat{H}_1))}\right].
    \end{align}

    Using instead the definition of $ g $ and \eqref{cond:zak_coupling3_2.2} in \cref{lem:coupling_zak}~\ref{cond:zak_coupling3_2}, by the same argument,\COMMENT{Here an expansion:
    \begin{align*}
        &\mathbb{P}[\mathcal{D}_{f(\hat{H}_1), g'(H_2, f(\hat{H}_1))}]\\
            &= \sum_{(J_1, J_2) \in \mathcal{H}} \sum_{j \in [r_{f(J_1)}]} \mathbb{P}[\mathcal{D}_{f(J_1), j} \mid \hat{H}_1 =_1 J_1, \hat{H}_2 = J_2, g'(H_2, f(J_1)) = j] \mathbb{P}[g'(H_2, f(J_1)) = j \mid \hat{H}_1 =_1 J_1, \hat{H}_2 = J_2] \mathbb{P}[\hat{H}_1 =_1 J_1, \hat{H}_2 = J_2]\\
            &= \sum_{(J_1, J_2) \in \mathcal{H}} \sum_{j \in [r_{f(J_1)}]} \mathbb{P}[\mathcal{D}_{f(J_1), j} \mid \hat{H}_1 =_1 J_1, \hat{H}_2 = J_2] \mathbb{P}[g'(H_2, f(J_1)) = j \mid \hat{H}_1 =_1 J_1, \hat{H}_2 = J_2] \mathbb{P}[\hat{H}_1 =_1 J_1, \hat{H}_2 = J_2]\\
            &\le \sum_{(J_1, J_2) \in \mathcal{H}} \sum_{j \in [r_{f(J_1)}]} \mathbb{P}[\mathcal{D}_{f(J_1), g(J_1, J_2)} \mid \hat{H}_1 =_1 J_1, \hat{H}_2 = J_2] \mathbb{P}[g'(H_2, f(J_1)) = j \mid \hat{H}_1 =_1 J_1, \hat{H}_2 = J_2] \mathbb{P}[\hat{H}_1 =_1 J_1, \hat{H}_2 = J_2]\\
            &= \sum_{(J_1, J_2) \in \mathcal{H}} \mathbb{P}[\mathcal{D}_{f(J_1), g(J_1, J_2)} \mid \hat{H}_1 =_1 J_1, \hat{H}_2 = J_2] \sum_{j \in [r_{f(J_1)}]} \mathbb{P}[g'(H_2, f(J_1)) = j \mid \hat{H}_1 =_1 J_1, \hat{H}_2 = J_2] \mathbb{P}[\hat{H}_1 =_1 J_1, \hat{H}_2 = J_2]\\
            &= \sum_{(J_1, J_2) \in \mathcal{H}} \mathbb{P}[\mathcal{D}_{f(J_1), g(J_1, J_2)} \mid \hat{H}_1 =_1 J_1, \hat{H}_2 = J_2] \mathbb{P}[\hat{H}_1 =_1 J_1, \hat{H}_2 = J_2]\\
            &= \mathbb{P}[\mathcal{D}_{f(\hat{H}_1), g(\hat{H}_1, \hat{H}_2)}].
    \end{align*}} it is also easy to check that
    \begin{equation} \label{eqn:zak8}
        \mathbb{P}\left[\mathcal{D}_{f(\hat{H}_1), g'(H_2, f(\hat{H}_1))}\right] \le \mathbb{P}\left[\mathcal{D}_{f(\hat{H}_1), g(\hat{H}_1, \hat{H}_2)}\right].
    \end{equation}

    Combining \eqref{eqn:zak75}, \eqref{eqn:zak8}, and \eqref{eqn:zak9}, we conclude that
    \[ \mathbb{P}\left[\mathcal{D}_{f(\hat{H}_1), g(\hat{H}_1, \hat{H}_2)}\right] \ge \mathbb{P}\left[\mathcal{D}_{f'(H_1), g'(H_2, f'(H_1))}\right] = 1 - o(1), \]
    as required.
\end{proof}

It remains to prove \cref{lem:coupling_zak,lem:two_stage_indep}. 
To show that \cref{lem:coupling_zak} holds, we make use of the following version of the functional representation lemma\COMMENT{At a quick glance, it may seem that both statements are the same. There is a subtle difference, however, in what is really proved.
We show that, given two random variables $X,Y\colon\Omega\to\mathbb{R}$ (say), there exist a third random variable $Z\colon\Omega'\to\mathbb{R}$ and a function $f\colon\mathbb{R}^2\to\mathbb{R}$ such that $X=f(Y,Z)$.
What is shown in \cite{GK11} is that there exist a random variable $Z\colon\Omega'\to\mathbb{R}$ and a function $f\colon\mathbb{R}^2\to\mathbb{R}$ such that $(X,Y)$ and $(f(Y,Z),Y)$ have the same distribution.} which has found multiple applications in information theory (see, e.g.,~\cite[p.~626]{GK11}).
For completeness, we include a short proof in \cref{appen:multi_stage}.

\begin{lemma}\label{claim:independence1}
    Let $ A $ and $ B $ be discrete random variables.
    Then there exist a random variable~$ C $, which is independent of~$ B $, and a function $f$ such that $ A = f( B, C ) $.
\end{lemma}

We also need the following slightly more complex statement, which gives a similar result to \cref{claim:independence1} in a case where several variables are involved.
Again, a proof can be found in \cref{appen:multi_stage}.

\begin{lemma}\label{claim:independence2}
    Let $ A, B, C $ be mutually independent discrete random variables, let $ F = f(A, B) $ and $ G = g(F, C) $, for some functions $ f, g $.
    Then there exist a random variable $ D $ which is independent of~$ (F, G) $ and a function $ a $ such that $ A = a(D, F) $.
\end{lemma}

With these tools at hand, we can now prove \cref{lem:coupling_zak}.

\begin{proof}[Proof of \cref{lem:coupling_zak}]
    By restricting to a subset, it suffices to consider the case that $ K^2 = K_n \setminus K^1 $.
    Fix probabilities $p_i=(1 - o(1)) t_i / M$, $i \in [3] $, and $\mybar{p}_2=o(t_2 / M)$, to be chosen later.
    Write $ N \coloneqq 2^M $.
    Let~$ (\hat{G}_i)_{i \in [3]} $ be a random variable distributed according to the random graph process, split into three stages according to $ (t_i)_{i \in [3]} $.
    We may then define the following mutually independent (discrete) random variables, which we call \defi{base variables}, where we treat $ (M_i)_{i\in[3]} $ as a single variable:
    \begin{enumerate}[label={\ensuremath{(\mathrm{V}\arabic*)}}]
        \item\label{basevariables1} $ (M_i)_{i \in [3]} \coloneqq (e(\hat{G}_i \cap K^1))_{i \in [3]} $;
        \item\label{basevariables2} $ H_i \sim G(n, p_i), i \in [3] $ and $ \mybar{H}_2 \sim G(n, \mybar{p}_2) $;
        \item\label{basevariables3} $ X_i $ sampled uniformly at random from $ [N!] $ for each $i \in [4] $.
    \end{enumerate}
    It is thus clear by definition from \ref{basevariables2} that \ref{cond:zak_coupling1} and \ref{cond:zak_coupling3} hold.

    We may define other random variables as functions of the base random variables, using \ref{basevariables3} as follows.
    Given $ r \in [N] $ and $ x \in [N!] $, partition the interval $ [N!] $ into $ r $ intervals $ I_1, \ldots, I_r $, each of size $ N!/r \in \mathbb{N} $, and write $ u_r(x) $ for the unique $ j\in[r] $ such that $ x \in I_j $; in particular, $ u_r(X_i) $ is uniformly distributed on $ [r] $ for each $i\in[4]$.
    Now, given a set $ Y $ with $|Y|=r$, label its elements arbitrarily as $ Y = \{ y_1, \ldots, y_r \} $ and write $ u_Y(x) \coloneqq y_{u_r(x)} $, so $ u_Y(X_i) $ is uniformly distributed on $ Y $.
    In particular, given any non-empty set $ Y $ with size at most $ N $, which in particular includes any set of graphs on $ n $ vertices, we may choose one of its elements uniformly at random entirely as a function of one of our independent random variables~$ X_i $; we refer to this as choosing a uniformly random element of $ Y $ \defi{according to $ X_i $}.

    Consider the following conditions:
    \begin{enumerate}[label={\ensuremath{(\mathrm{C}\arabic*)}}]
        \item\label{couplingcond1} $ e(H_1 \cap K^1) \le M_1 $;
        \item\label{couplingcond2} $ e(H_2 \cap K^2) \le t_2 - M_2 \le e((H_2 \cup \mybar{H}_2) \cap K^2) $;
        \item\label{couplingcond3} $ e(H_3 \cap K^2) \le t_3 - M_3 $.
    \end{enumerate}
    Using Chernoff's bound (\cref{lem:Chernoff}) and its version for hypergeometric random variables (see, e.g.,~\cite[Theorem~2.10]{JLR}), it is not hard to verify that, for suitably chosen $p_i=(1 - o(1)) t_i / M$ and $\mybar{p}_2=o(t_2 / M)$, a.a.s.\ \ref{couplingcond1}--\ref{couplingcond3} hold.

    Next, we sample the following graphs sequentially:
    \begin{enumerate}[label={\ensuremath{(\mathrm{G}\arabic*)}}]
        \item If \ref{couplingcond1} holds, then sample $ \hat{H}_1 \cap K^1 $ uniformly at random among all graphs $ G $ with exactly $ M_1 $ edges such that $ H_1 \cap K^1 \subseteq G \subseteq K^1 $, according to $ X_1 $.
            Otherwise, sample $ \hat{H}_1 \cap K^1 $ uniformly at random among all graphs $ G \subseteq K^1 $ with exactly $ M_1 $ edges, according to $ X_1 $.\label{stepg1}
        \item If \ref{couplingcond1} and \ref{couplingcond2} hold, then sample $ \hat{H}_2 $ uniformly at random among all graphs $ G $ such that $ e(G \cap K^1) = M_2 $, $ e(G) = t_2 $, $ H_2 \cap K^2 \subseteq G \cap K^2 \subseteq (H_2 \cup \mybar{H}_2) \cap K^2 $, and $ e(G \cap \hat{H}_1 \cap K^1) = 0 $, according to $ X_2 $.
            Otherwise, sample $ \hat{H}_2 $ uniformly at random among all graphs $ G $ such that $ e(G \cap K^1) = M_2 $, $ e(G) = t_2 $, and $ e(G \cap \hat{H}_1 \cap K^1) = 0 $, according to $ X_2 $.\label{stepg2}
        \item If \ref{couplingcond1}--\ref{couplingcond3} all hold, then sample $ \hat{H}_3 $ uniformly at random among all graphs $ G $ such that $ e(G \cap K^1) = M_3 $, $ e(G) = t_3 $, $ (H_3 \setminus \hat{H}_2) \cap K^2 \subseteq G \cap K^2 $, and $ e(G \cap ((\hat{H}_1 \cap K^1) \cup \hat{H}_2)) = 0 $, according to $ X_3 $.
            Otherwise, sample $ \hat{H}_3 $ uniformly at random among all graphs $ G $ such that $ e(G \cap K^1) = M_3 $, $ e(G) = t_3 $, and $ e(G \cap ((\hat{H}_1 \cap K^1) \cup \hat{H}_2)) = 0 $, according to $ X_3 $.\label{stepg3}
        \item Sample $ \hat{H}_1 \cap K^2 $ uniformly at random among all subgraphs of $ K^2 \setminus (\hat{H}_2 \cup \hat{H}_3) $ with exactly $ t_1-M_1 $ edges, according to $ X_4 $.\label{stepg4}
    \end{enumerate}
    Note that, using the assumed conditions in the respective cases, all of the sets of graphs from which we sample are non-empty and have size at most $ N $, so this is well defined.
    
    Now observe firstly that, assuming the coupling succeeds, all of the containments of \ref{cond:zak_coupling4} hold by definition.
    In order to see \ref{cond:zak_coupling2}, note that $ e(\hat{H}_i) = t_i $ and $ e(\hat{H}_i \cap K^1) = M_i = e(\hat{G}_i \cap K^1) $ for all $i\in[3]$, in all cases.
    Furthermore, conditional on any particular values of $ (M_i)_{i \in [3]} $, it follows from our definitions above that the distribution of $ (\hat{H}_i)_{i \in [3]} $ is uniform over all triples of edge-disjoint graphs $ (H'_i)_{i \in [3]} $ for which $ e(H'_i) = t_i $ and $ e(H'_i \cap K^1) = M_i $ for each $ i \in [3] $ (since any edge is equally likely to appear in any of the binomial random graphs).
    It thus only remains to prove that \ref{cond:zak_coupling3_2} holds.
    
    To show this, we use \cref{claim:independence1,claim:independence2}.
    Fix $ X \subseteq \mathcal{G}\times\mathcal{G}\times\mathcal{G} $ and $ \Gamma_1, \Lambda_1 \in \mathcal{G} $ with \mbox{$ e(\Gamma_1) = t_1 $} such that $\mathbb{P}[\hat{H}_1 \cap K^1 = \Gamma_1 \cap K^1, H_1 = \Lambda_1]\neq 0$.
    Since $ (M_i)_{i \in [3]} $ is independent of all other base variables, we may apply \cref{claim:independence1} with $ B = M_1 $ and $ A = (M_2, M_3) $ to obtain a random variable~$ Y_1 $, which is independent from $ M_1 $ and all base variables except $ (M_i)_{i \in [3]} $, such that $ (M_2, M_3) $ is a deterministic function of~$ (M_1, Y_1) $.
    Now, apply \cref{claim:independence2} with $ A = H_1$, $B = (M_1, X_1)$, $F = \hat{H}_1 \cap K^1$, \mbox{$G = (H_2, \hat{H}_2, \hat{H}_3)$}, and $C = (Y_1, H_2, H_3, \mybar{H}_2, X_2, X_3) $, for which \ref{stepg1} gives an $f$ such that $F=f(A,B)$, and \ref{stepg2} and~\ref{stepg3} give a $g$ such that $G=g(F,C)$.\COMMENT{Note that $ F $ is a deterministic function of~$ (A, B) $ in each case depending on whether there is a failure at stage~1, a condition which in itself is a deterministic function of~$ (A, B) $. This is where it is crucial that the failure condition is split into stages rather than global.}
    We thus obtain a random variable $ Y_2 $ and a deterministic function~$ a $ such that $ H_1 = a(\hat{H}_1 \cap K^1, Y_2) $ and $ Y_2 $ is independent of \mbox{$ (\hat{H}_1 \cap K^1, H_2, \hat{H}_2, \hat{H}_3) $}.
    Let~$ \mathcal{A} $ denote the set of all elements~$ x $ in the image of~$ Y_2 $ such that \mbox{$ a(\Gamma_1 \cap K^1, x) = \Lambda_1 $}.
    We may then compute
    \begin{align*}
        &\,\mathbb{P}\left[(H_2, \hat{H}_2, \hat{H}_3) \in X \ \middle|\ \hat{H}_1\cap K^1 = \Gamma_1 \cap K^1, H_1 = \Lambda_1\right]\\
        =&\, \mathbb{P}\left[(H_2, \hat{H}_2, \hat{H}_3) \in X \ \middle|\ \hat{H}_1\cap K^1 = \Gamma_1 \cap K^1, Y_2 \in \mathcal{A}\right]\\
        =&\, \mathbb{P}\left[(H_2, \hat{H}_2, \hat{H}_3) \in X \ \middle|\ \hat{H}_1\cap K^1 = \Gamma_1 \cap K^1\right],
    \end{align*}
    as required for \eqref{cond:zak_coupling3_2.1} in \ref{cond:zak_coupling3_2}.

    Similarly, apply \cref{claim:independence1} with $ B = (M_1, M_2) $ and $ A = M_3 $ to obtain $ Y_3 $, which is independent from $ (M_1, M_2) $ and all base variables except $ (M_i)_{i \in [3]} $, such that $ M_3 $ is a deterministic function of~$ (M_1, M_2, Y_3) $.
    Apply \cref{claim:independence2} with $ A = H_2$, $B = (M_1, M_2, H_1, \mybar{H}_2, X_1, X_2)$, \mbox{$F = (\hat{H}_1 \cap K^1, \hat{H}_2)$}, $G = \hat{H}_3$, and $C = (Y_3, H_3, X_3) $ to obtain a random variable $ Y_4 $ and a deterministic function $ b $ such that $ H_2 = b(\hat{H}_1 \cap K^1, \hat{H}_2, Y_4) $ and $ Y_4 $ is independent of $ (\hat{H}_1 \cap K^1, \hat{H}_2, \hat{H}_3) $.
    We may then verify \eqref{cond:zak_coupling3_2.2} in \ref{cond:zak_coupling3_2} similarly.\COMMENT{Fix an arbitrary $Y\subseteq\mathcal{G}$ and $ \Gamma_1,\Gamma_2,\Lambda_2 \in \mathcal{G} $ such that $ e(\Gamma_1) = t_1 $, $ e(\Gamma_2) = t_2 $, and $\mathbb{P}[\hat{H}_1 \cap K^1 = \Gamma_1 \cap K^1, \hat{H}_2 = \Gamma_2, H_2 = \Lambda_2]\neq 0$.
    Let $\mathcal{B}$ denote the set of all elements $y$ in the image of $Y_4$ such that $b(\Gamma_1\cap K^1,\Gamma_2,y)=\Lambda_2$.
    Then
    \begin{align*}
        &\,\mathbb{P}\left[\hat{H}_3\in Y \ \middle|\ \hat{H}_1\cap K^1 = \Gamma_1\cap K^1, \hat{H}_2 = \Gamma_2, H_2 = \Lambda_2\right]\\
        =&\,\mathbb{P}\left[\hat{H}_3\in Y \ \middle|\ \hat{H}_1\cap K^1 = \Gamma_1\cap K^1, \hat{H}_2 = \Gamma_2, Y_4 \in \mathcal{B}\right]\\
        =&\,\mathbb{P}\left[\hat{H}_3\in Y \ \middle|\ \hat{H}_1\cap K^1 = \Gamma_1\cap K^1, \hat{H}_2 = \Gamma_2\right].
    \end{align*}}
\end{proof}

Finally, we can prove \cref{lem:two_stage_indep}, which also completes the proof of \cref{lem:two_stage_real}.

\begin{proof}[Proof of \cref{lem:two_stage_indep}]
    For convenience, for each $i\in[m]$ and $j\in[r_{i^*}]$ set $\mathcal{L}_1^i\coloneqq\mathcal{L}_1\cap\{i^*=i\}$ and $\mathcal{L}_2^j\coloneqq\mathcal{L}_2\cap\{j^*=j\}$.
    By \ref{cond:zak1}, we have that
    \begin{equation}\label{equa:zak0}
        \mathbb{P}[\mathcal{L}_1]\geq1-\delta.
    \end{equation}
    Note that, since the different random graphs we consider are chosen independently, conditioning upon $ \mathcal{L}_1 $ and on the outcome of $i^*$ does not affect the probability of any events defined only over the choice of $H_2$, $H_3$ and $\mybar{H}_2$.
    
    Let $ H_2' \coloneqq H_2 \cup \mybar{H}_2 $.
    Now, for each $i\in[m]$ and each $ j \in [r_i] $, define the events
    \[ \mathcal{C}_{i,j} \coloneqq \{ E_{i, j} \subseteq H_2 \},\quad \mathcal{D}_{i,j} \coloneqq \{ E_{i, 1} \not \subseteq H_2, \ldots, E_{i, j - 1} \not \subseteq H_2 \},\quad \text{ and } \quad\mathcal{A}_{i,j} \coloneqq \mathcal{C}_{i,j} \cap \mathcal{D}_{i,j}. \]
    In particular, note that $\mathcal{A}_{i,j} \cap \mathcal{L}_1^i = \mathcal{L}_2^j \cap \mathcal{L}_1^i$ and that, for each $i\in[m]$, the events in the set $\{\mathcal{A}_{i,j}:j\in[r_i]\}$ are pairwise disjoint.
    With these definitions, and using the independence between $H_1$ and $H_2$ as well as \ref{cond:zak2}, for each $i\in[m]$ we have that
    \begin{equation}\label{equa:zak1}
        \mathbb{P}\left[\mathcal{L}_2\ \middle|\ \mathcal{L}_1^i\right]=\sum_{j\in[r_i]}\mathbb{P}\left[\mathcal{A}_{i,j}\ \middle|\ \mathcal{L}_1^i\right]=\sum_{j\in[r_i]}\mathbb{P}[\mathcal{A}_{i,j}]\geq1-\delta.
    \end{equation}
    Moreover, for each $i\in[m]$ and each $ j \in [r_i] $, let $\mathcal{B}_{i,j}$ be the event that there exists some $F \in \mathcal{F}_{i, j}$ such that $F \setminus E_{i, j} \subseteq H_3 \setminus H_2' $.
    Then \ref{cond:zak3} says that for every $i\in[m]$ and $ j \in [r_i] $ we have that
    \begin{equation}\label{equa:zak2}
        \mathbb{P}[\mathcal{B}_{i,j}] \ge 1 - \delta.
    \end{equation}
    
    We claim that the events $\mathcal{A}_{i,j}$ and $\mathcal{B}_{i,j}$ are positively correlated.

    \begin{claim} \label{claim:zak_cond}
        For every $i\in[m]$ and every $j\in[r_i]$ we have that $ \mathbb{P}[\mathcal{B}_{i,j} \cap \mathcal{A}_{i,j}] \ge \mathbb{P}[\mathcal{B}_{i,j}] \mathbb{P}[\mathcal{A}_{i,j}] $.
    \end{claim}

        \begin{claimproof}[Proof of \cref{claim:zak_cond}]
        Fix $i\in[m]$ and $j\in[r_i]$.
        By the definition of $\mathcal{A}_{i,j}$, we have that 
        \[ \mathbb{P}[\mathcal{B}_{i,j} \cap \mathcal{A}_{i,j}] = \mathbb{P}[\mathcal{B}_{i,j} \cap \mathcal{C}_{i,j} \cap \mathcal{D}_{i,j}]. \]
        Observe that $ \mathcal{C}_{i,j} $ is independent of $ \mathcal{B}_{i,j} $, since they are determined by disjoint sets of edges.
        Intuitively, we would like to use this fact to say that $ \mathbb{P}[\mathcal{B}_{i,j} \cap \mathcal{A}_{i,j}] = \mathbb{P}[\mathcal{B}_{i,j} \cap \mathcal{D}_{i,j}] \mathbb{P}[\mathcal{C}_{i,j}] $.
        However, we must first account for the fact that $ \mathcal{C}_{i,j} $ and $ \mathcal{D}_{i,j} $ may not be independent in general.
        To do so, we define a further event
        \[ \mathcal{M}_{i,j} \coloneqq \bigcap_{k \in [j - 1]} \{ E_{i,k} \setminus E_{i,j} \not \subseteq H_2 \}, \]
        which is independent from $ \mathcal{C}_{i,j} $ by edge-disjointness and satisfies that
        \[ \mathcal{C}_{i,j} \cap \mathcal{M}_{i,j} = \mathcal{C}_{i,j} \cap \mathcal{D}_{i,j}. \]
        Hence, in particular, we may write\COMMENT{Note that it is not true in general that, if an event is independent from two other events, then it is independent of their intersection. Formally, thus, here one needs to show that $\mathcal{B}_{i,j} \cap \mathcal{M}_{i,j}$ is independent of $\mathcal{C}_{i,j}$. This is fine because we have edge-disjointness between the sets of edges which define the events.}
        \begin{equation} \label{eqn:zak5}
            \mathbb{P}[\mathcal{B}_{i,j} \cap \mathcal{A}_{i,j}]
                = \mathbb{P}[\mathcal{B}_{i,j} \cap \mathcal{C}_{i,j} \cap \mathcal{D}_{i,j}]
                = \mathbb{P}[\mathcal{B}_{i,j} \cap \mathcal{C}_{i,j} \cap \mathcal{M}_{i,j}]
                = \mathbb{P}[\mathcal{B}_{i,j} \cap \mathcal{M}_{i,j}] \mathbb{P}[\mathcal{C}_{i,j}].
        \end{equation}

        We now make the intuitive observation that both $ \mathcal{B}_{i,j} $ and $ \mathcal{M}_{i,j} $ are ``decreasing with respect to~$ H_2 $'', in the sense that, if the event $ \mathcal{B}_{i,j} $ holds when $ (H_2, \mybar{H}_2, H_3) = (G_1, G_2, G_3) $ and $ G_1' \subseteq G_1 $, then $ \mathcal{B}_{i,j} $ also holds when $ (H_2, \mybar{H}_2, H_3) = (G_1', G_2, G_3) $, and likewise for $ \mathcal{M}_{i,j} $.
        We therefore aim to apply \cref{lem:fkg_inequality}.
        Recall that $ \mathcal{G} $ denotes the set of all graphs on vertex set $ [n] $.
        Define two functions $ f, g\colon \mathcal{G} \to \mathbb{R} $ by
        \[ f(G) \coloneqq -\mathbb{P}[\mathcal{B}_{i,j} \mid H_2 = G]\qquad \text{ and } \qquad g(G) \coloneqq -\mathds{1}(\mathcal{M}_{i,j} \mid H_2 = G), \]
        noting that $ \mathcal{M}_{i,j} $ is fully determined by $ H_2 $, whereas $ \mathcal{B}_{i,j} $ may also depend upon $ \mybar{H}_2 $ and $ H_3 $.
        By the observation above, both $ f $ and $ g $ are increasing functions.
        We therefore obtain that
        \begin{align}
            \mathbb{P}[\mathcal{B}_{i,j} \cap \mathcal{M}_{i,j}]
                &= \sum_{G \in \mathcal{G}} \mathbb{P}[H_2 = G] \mathds{1}(\mathcal{M}_{i,j} \mid H_2 = G) \mathbb{P}[\mathcal{B}_{i,j} \mid H_2 = G]\nonumber\\
                &= \sum_{G \in \mathcal{G}} \mathbb{P}[H_2 = G] (-1)^2 g(G) f(G)\nonumber\\
                &= \mathbb{E}[g(H_2) f(H_2)]\nonumber\\
                &\ge \mathbb{E}[g(H_2)] \mathbb{E}[f(H_2)]\nonumber\\
                &= (-1)^2 \mathbb{P}[\mathcal{M}_{i,j}] \sum_{G \in \mathcal{G}} \mathbb{P}[H_2 = G] \mathbb{P}[\mathcal{B}_{i,j} \mid H_2 = G]\nonumber\\
                &= \mathbb{P}[\mathcal{M}_{i,j}] \mathbb{P}[\mathcal{B}_{i,j}],\label{eqn:zak55}
        \end{align}
        where the inequality follows from \cref{lem:fkg_inequality}.
        The result now follows by combining \eqref{eqn:zak5} and \eqref{eqn:zak55} with the independence of the defined events:
        \begin{align*}
            \mathbb{P}[\mathcal{B}_{i,j} \cap \mathcal{A}_{i,j}] &= \mathbb{P}[\mathcal{B}_{i,j} \cap \mathcal{M}_{i,j}] \mathbb{P}[\mathcal{C}_{i,j}]\\
            &\ge \mathbb{P}[\mathcal{B}_{i,j}] \mathbb{P}[\mathcal{M}_{i,j}] \mathbb{P}[\mathcal{C}_{i,j}]
            = \mathbb{P}[\mathcal{B}_{i,j}] \mathbb{P}[\mathcal{M}_{i,j} \cap \mathcal{C}_{i,j}]
            = \mathbb{P}[\mathcal{B}_{i,j}] \mathbb{P}[\mathcal{A}_{i,j}].\qedhere
        \end{align*}
    \end{claimproof}
    
    Using the fact that $\mathcal{L}_2=\mathcal{L}_2\cap\mathcal{L}_1$, the definition of $\mathcal{A}_{i,j}$, the independence of the events $\mathcal{A}_{i,j}$ and~$\mathcal{B}_{i,j}$ from $H_1$, and \cref{claim:zak_cond}, for every $i\in[m]$ and $ j \in [r_i] $ we have that
    \begin{align}
        \mathbb{P}\left[\mathcal{B}_{i,j} \ \middle|\ \mathcal{L}_1^i\cap\mathcal{L}_2^j\right] &=\frac{\mathbb{P}\left[\mathcal{B}_{i,j} \cap\mathcal{L}_2^j\ \middle|\  \mathcal{L}_1^i\right]}{\mathbb{P}\left[\mathcal{L}_2^j\ \middle|\ \mathcal{L}_1^i\right]}
        =\frac{\mathbb{P}\left[\mathcal{B}_{i,j} \cap\mathcal{A}_{i,j}\ \middle|\ \mathcal{L}_1^i\right]}{\mathbb{P}\left[\mathcal{A}_{i,j}\ \middle|\ \mathcal{L}_1^i\right]}
        =\frac{\mathbb{P}[\mathcal{B}_{i,j} \cap\mathcal{A}_{i,j}]}{\mathbb{P}[\mathcal{A}_{i,j}]}\nonumber\\
        &\geq\frac{\mathbb{P}[\mathcal{B}_{i,j}]\mathbb{P}[\mathcal{A}_{i,j}]}{\mathbb{P}[\mathcal{A}_{i,j}]}
        =\mathbb{P}[\mathcal{B}_{i,j}].\label{equa:zak3}
    \end{align}

    For simplicity, let us denote by $\mathcal{E}$ the event that $ \mathcal{L}_2 $ occurs and there exists $ F \in \mathcal{F}_{i^*, j^*} $ for which $ F \setminus E_{i^*, j^*} \subseteq H_3 \setminus H_2' $.
    Using \eqref{equa:zak0}, \eqref{equa:zak1}, \eqref{equa:zak2} and~\eqref{equa:zak3} we then get that
    \begin{align*}
        \mathbb{P}[\mathcal{E}]&=\sum_{i=1}^m\sum_{j=1}^{r_i}\mathbb{P}\left[\mathcal{L}_1^i\cap\mathcal{L}_2^j\cap\mathcal{B}_{i,j}\right]
        =\sum_{i=1}^m\sum_{j=1}^{r_i}\mathbb{P}\left[\mathcal{B}_{i,j} \ \middle|\ \mathcal{L}_1^i\cap\mathcal{L}_2^j\right]\mathbb{P}\left[\mathcal{L}_2^j\ \middle|\ \mathcal{L}_1^i\right]\mathbb{P}\left[\mathcal{L}_1^i\right]\\
        &\geq(1-\delta)\sum_{i=1}^m\mathbb{P}\left[\mathcal{L}_1^i\right]\sum_{j=1}^{r_i}\mathbb{P}\left[\mathcal{L}_2^j\ \middle|\ \mathcal{L}_1^i\right]
        =(1-\delta)\sum_{i=1}^m\mathbb{P}\left[\mathcal{L}_1^i\right]\mathbb{P}\left[\mathcal{L}_2\ \middle|\ \mathcal{L}_1^i\right]\\
        &\geq(1-\delta)^3 \ge 1 - 3 \delta.\qedhere
    \end{align*}
\end{proof}


\section{Graph factors: lower bounds for successful strategies}\label{section:factors_low}

In this section, we prove lower bounds on the time and budget required for constructing $F$-factors. 
To do this, we show that there are no successful $(t,b)$-strategies for constructing $ F $-factors (and more generally any graph containing many vertex-disjoint copies of $F$) if $t$ and $b$ are not sufficiently large.
We can formulate such statements by considering some $t$ larger than the (likely) hitting time for $F$-factors and providing a lower bound on the budget $b$ for any successful $(t,b)$-strategies for $F$-factors.
In \cref{section:factors_upper} we will showcase some simple $(t,b)$-strategies which a.a.s.\ result in graphs containing the desired $F$-factors, where the values of $b$ are very close to the lower bounds proved in this section.
This shows that our results are essentially tight.

Let us begin by discussing the intuitive ideas behind our main result of the section (\cref{thm:lowerboundF}), which in particular shows that, for any property which requires that $B_t$ contains linearly many vertex-disjoint copies of a fixed connected graph $F$, the budget~$b$ for any successful $(t,b)$-strategy with $t$ ``close'' to the hitting time must also be ``close'' to the hitting time. 
We will later use this result to prove \cref{thm:factors_lower,thm:hamsquare_lower}.

Consider the simple problem of constructing a triangle in the budget-constrained random graph process.
This requires buying some edge~$ vw $ contained within the neighbourhood of some vertex~$ u $ (where the neighbourhood is formed by already bought edges incident to~$u$).
Roughly and intuitively, if we let $p\coloneqq t/M$ and~$ u $ has degree~$ d \in [n] $ in Builder's graph, then its neighbourhood contains~$ \inbinom{d}{2} $ potential edges, so the expected number of such edges which are presented up to time~$ t $ is approximately $ p \inbinom{d}{2} $.
If we want to be able to construct such a triangle a.a.s., then we need that, with high probability, at least one of these edges is presented by the random graph process, for which we want to ensure that $ p \inbinom{d}{2} = \omega(1) $.
In other words, to construct a triangle in time $ n^{2 - 2 \eps} $, we will probably need to have at least some vertex of degree at least $\omega(n^{\eps})$.

Now suppose instead that we want to construct \emph{many} vertex-disjoint triangles, say linearly many (which is required, for example, for a triangle factor or the square of a Hamilton cycle).
Then, by the same logic, to achieve such a structure a.a.s.\ in time $ n^{2 - 2 \eps} $, we will need to ensure that linearly many vertices have degree at least $ n^{\eps} $, which in particular means we have no hope of achieving this with a linear budget.

The following theorem formalises this idea, and generalises it from triangles to any fixed connected graph $ F $ of constant size.

\begin{theorem}\label{thm:lowerboundF}
    Let $ t=t(n) $ and $ b=b(n) $ be such that $t,b\in[M] $ and $ t = o(n^2) $.
    For all $ k_0 \in \mathbb{N} $ and $ \eps> 0 $, there exist $ \lambda > 0 $ and $ n_0 \in \mathbb{N} $ such that, for all $ n \ge n_0 $, the following holds.
    
    Let $ B_t $ be produced by the budget-constrained random graph process under a $(t, b)$-strategy.
    Then, with probability at least $ 1 - \eps $, there exists a set of vertices $ V \subseteq [n] $ of size $|V|\geq (1 - \eps) n $ satisfying the following property.
    
    For every connected graph~$ F $ on at most~$ k_0 $ vertices and every vertex~$ v \in V $, the number of copies of\/~$F$ in~$B_t[V]$ containing $v$ is at most $ \lambda b^{v(F) - 1} t^{e(F) - v(F) + 1} n^{v(F)-2e(F)-1}$.
\end{theorem}

Before proving \cref{thm:lowerboundF}, let us use it to derive lower bounds for time and budget of successful $(t,b)$-strategies for different properties.
In general, suppose that Builder tries to construct a structure which contains linearly many vertex-disjoint copies of some fixed connected graph $ F $ on $ k $ vertices with $ m \geq k $ edges.
Then, if we take $ \eps $ sufficiently small, every set of size $(1-\eps)n$ must contain some copy of~$ F $.
In particular, by \cref{thm:lowerboundF}, in order for Builder to a.a.s.\ find a copy of the large structure, it is a necessary condition that $ \lambda b^{k - 1} t^{m - k + 1} n^{k - 2m - 1} \ge 1 $.
Since~$ F $ is not a tree, given any budget~$b$ (resp.\ time~$t$), this yields a lower bound on the time~$t$ (resp.\ budget~$b$) required to have a successful $(t,b)$-strategy for such a structure (as showcased, for instance, in \cref{fig:HamPowers,fig:Krfactors}).
In general, we obtain the following, stronger result, which gives a negative answer to \cref{question:FKMgeneral}, posed by \citet{FKM25}, in a strong sense.

\begin{corollary}\label{thm:Ffactors_lower}
    Let\/ $F$ be a (not necessarily connected) fixed graph with at least one edge, and let $\alpha\in(0,1]$.
    Let\/ $t=t(n)$ and $b=b(n)$ be such that there exists a successful $ (t, b) $-strategy for an $\alpha$-$ F $-factor.
    Then we must have $ bt^{d^*(F)-1} = \Omega(n^{2d^*(F)-1}) $.
\end{corollary}

\begin{proof}
    Let $H\subseteq F$ be a subgraph of $F$ with $d^*(F)=d(H)$.
    It can easily be checked that $H$ must be connected\COMMENT{Suppose $H$ is not connected, and suppose without loss of generality that it consists of two vertex-disjoint graphs $H_1$ and $H_2$.
    Say that $H$ has $k$ vertices and $m$ edges, and that $H_i$ has $k_i\geq2$ vertices and $m_i$ edges, with $k=k_1+k_2$ and $m=m_1+m_2$. 
    (If one of the two graphs, say $H_1$, only has $1$ vertex, then its $1$-density is not defined (or could be defined as equal to $0$ by convention), but then the other component would have fewer vertices and the same number of edges, and so have a higher $1$-density than $H$, contradicting the choice of $H$.)
    We are going to show that there is some $i\in[2]$ such that $d(H_i)>d(H)$, contradicting the choice of $H$.\\
    Indeed, suppose this is not the case.
    Then we have that
    \[d(H)\geq d(H_1)\iff\frac{m_1+m_2}{k_1+k_2-1}\geq \frac{m_1}{k_1-1}\iff m_1k_1+m_2k_1-m_1-m_2\geq m_1k_1+m_1k_2-m_1\iff m_2k_1-m_1k_2\geq m_2,\]
    and similarly
    \[d(H)\geq d(H_2)\iff\frac{m_1+m_2}{k_1+k_2-1}\geq \frac{m_2}{k_2-1}\iff m_1k_2+m_2k_2-m_1-m_2\geq m_2k_1+m_2k_2-m_2\iff m_1k_2-m_2k_1\geq m_1.\]
    Adding both expressions, we reach a contradiction: $0\geq m\geq1$.} and have at least one edge, and that, moreover, any graph containing an $\alpha$-$F$-factor must contain an $\alpha'$-$H$-factor, where $\alpha'\coloneqq\alpha v(H)/v(F)$.
    Let $k\coloneqq v(H)$ and $m\coloneqq e(H)$, and suppose that $ B_t $ is the result of the budget-constrained random graph process under some $ (t, b) $-strategy, where $ b^{k-1}t^{m-k+1} = o(n^{2m-k+1}) $.
    If $t=o(n^2)$ and $n$ is sufficiently large, by \cref{thm:lowerboundF} applied with $\eps=\alpha'/2k$, with probability at least~$1/2$,\COMMENT{Simply note that $\eps\leq1/2$ by definition.} there exists a set $ U \subseteq [n] $ of size at least $ (1 - \alpha'/2k) n $ containing no copy of $ H $.
    In particular, every copy of $ H $ has at least one vertex in the set $ [n] \setminus U $, which has size at most $ \alpha' n/2k =\alpha n/2|V(F)|$, so any collection of vertex-disjoint copies of $ F $ covers at most $ \alpha n/2 $ vertices.
    Hence, in this event there does not exist an $\alpha$-$ F $-factor.
    On the other hand, if $t=\Omega(n^2)$, then $b=o(n)$, so there cannot be an $\alpha$-$F$-factor since this requires at least a linear budget.
\end{proof}

Since for every $r\geq2$ we have that $d^*(K_r)=r/2$, \cref{thm:factors_lower} corresponds to the particular case of \cref{thm:Ffactors_lower} when $F=K_r$.
\Cref{thm:hamsquare_lower} follows by noting that powers of Hamilton cycles contain (almost) $F$-factors for different graphs $F$; we include the details next.

\begin{proof}[Proof of \cref{thm:hamsquare_lower}]
For each pair of integers $(k,q)$ with $k\geq2$ and $q\geq k+1$, let $ P_q^k $ denote the $k$-th power of a $ q $-vertex path, which is a connected graph with $ \inbinom{k}{2}+(q-k)k $ edges (see \cref{fig:hamsquare} for examples when $k=2$).
Note that, for any fixed such pair $ (k,q) $ and $ n $ sufficiently large, if a graph~$ G $ on~$ n $ vertices contains the $k$-th power of a Hamilton cycle, then it contains a $(1/2)$-$ P_q^k $-factor.

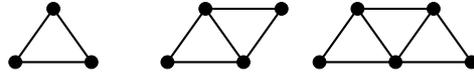
\begin{figure}[ht]
    \begin{tikzpicture}
        \node (x1) at (0, 0) [vtx] {};
        \node (x2) at (1, 0) [vtx] {};
        \node (x3) at (0.5, 0.707) [vtx] {};
        
        \draw[-, thick] (x1) -- (x2);
        \draw[-, thick] (x2) -- (x3);
        \draw[-, thick] (x3) -- (x1);

        \node (y1) at (2, 0) [vtx] {};
        \node (y2) at (3, 0) [vtx] {};
        \node (y3) at (2.5, 0.707) [vtx] {};
        \node (y4) at (3.5, 0.707) [vtx] {};
        
        \draw[-, thick] (y1) -- (y2);
        \draw[-, thick] (y2) -- (y3);
        \draw[-, thick] (y3) -- (y1);
        \draw[-, thick] (y2) -- (y4);
        \draw[-, thick] (y3) -- (y4);

        \node (z1) at (4, 0) [vtx] {};
        \node (z2) at (5, 0) [vtx] {};
        \node (z3) at (4.5, 0.707) [vtx] {};
        \node (z4) at (5.5, 0.707) [vtx] {};
        \node (z5) at (6, 0) [vtx] {};
        
        \draw[-, thick] (z1) -- (z2);
        \draw[-, thick] (z2) -- (z3);
        \draw[-, thick] (z3) -- (z1);
        \draw[-, thick] (z2) -- (z4);
        \draw[-, thick] (z3) -- (z4);
        \draw[-, thick] (z4) -- (z5);
        \draw[-, thick] (z2) -- (z5);
    \end{tikzpicture}
    \caption{The graphs $ P_3^2 $, $ P_4^2 $ and $ P_5^2 $.}\label{fig:hamsquare}
\end{figure}

Fix an integer $k\geq2$.
Let $\eta>0$ be arbitrarily small, and fix $ q \geq k^2/\eta $.
Suppose that $ B_t $ is the result of the budget-constrained random graph process under some $ (t, b) $-strategy with $ b \leq n^{2k-1-\eta} / t^{k-1} $.
Then, since we may assume that $t\geq b=\Omega(n)$, by the choice of $q$ we have that\COMMENT{Since $ b \leq n^{2k-1-\eta} / t^{k-1} $, we have that
\[b^{q - 1} t^{\genfrac{(}{)}{0pt}{2}{k}{2}+(q-k)k-q+1}=\left(\frac{b}{t}\right)^{q - 1} t^{\binom{k}{2}+(q-k)k}\leq\left(\frac{n^{2k-1-\eta}}{t^{k}}\right)^{q - 1} t^{\binom{k}{2}+(q-k)k}=\frac{n^{(2k-1-\eta)(q-1)}}{t^{\binom{k}{2}}}.\]
We wish to show that this expression is $o(n^{k(k-1)+2(q-k)k-q+1})$.
By reordering terms, it suffices to show that
\[n^{2\binom{k}{2}-\eta(q-1)}=o(t^{\binom{k}{2}}).\]
As we may assume that $t\geq b=\Omega(n)$, we will be done if $\eta(q-1)>\inbinom{k}{2}$.
And this is achieved by the choice of $q$. (Being more careful, it would also work for smaller choices of $q$, but this does not affect the result.)}
\[b^{q - 1} t^{\genfrac{(}{)}{0pt}{2}{k}{2}+(q-k)k-q+1} = o(n^{k(k-1)+2(q-k)k-q+1}),\]
which, since $ d^*(P_q^k)\geq d(P_q^k) $, implies that\COMMENT{In fact, we have that $d^*(P_q^k)=d(P_q^k)$, so the result follows immediately, but we do not want to prove that $d^*(P_q^k)=d(P_q^k)$. To avoid this, it suffices to note that
\begin{align*}
    &\,b^{q - 1} t^{\binom{k}{2}+(q-k)k-q+1} = o(n^{k(k-1)+2(q-k)k-q+1})\\
    \iff &\, b t^{d(P_q^k)-1} = o(n^{2d(P_q^k)-1})\\
    \iff &\, \frac{bn}{t} = o\left(\left(\frac{n^2}{t}\right)^{d(P_q^k)}\right)\\
    \implies&\,\frac{bn}{t} = o\left(\left(\frac{n^2}{t}\right)^{d^*(P_q^k)}\right)\\
    \iff &\, b t^{d^*(P_q^k)-1} = o(n^{2d^*(P_q^k)-1}),
\end{align*}
where the implication holds since $t\leq n^2$ and $d(P_q^k)\leq d^*(P_q^k)$.}
\[b t^{d^*(P_q^k)-1} = o(n^{2d^*(P_q^k)-1}).\]
Hence, by \cref{thm:Ffactors_lower}, the strategy is not successful for a $(1/2)$-$P_q^k$-factor, and thus, it is also not successful for the $k$-th power of a Hamilton cycle.
Since $\eta$ was arbitrary, this completes the proof.
\end{proof}

Let us now focus on the proof of \cref{thm:lowerboundF}.

\begin{remark}
    The following intuition may be useful when reading the proof of \cref{thm:lowerboundF}.
    Recall that we sometimes consider $d\coloneqq 2b/n$ as (an upper bound on) the average degree of $B_t$ and $p\coloneqq t/M$ as the probability that any edge is presented throughout the process.
    \Cref{thm:lowerboundF} tells us that the number of copies of $F$ containing a vertex $v\in V$ is $O(d^{v(F) - 1} p^{e(F) - v(F) + 1})$.
    Intuitively, the budget restriction gives us a bound on the number of copies of spanning trees of $ F $ rooted at $v$ in $ B_t $, counted by considering the vertices one at a time in such a way that each subsequent vertex belongs to the neighbourhood of one of the previous ones.
    Given such a spanning tree for which all edges are bought, the time restriction then bounds the probability that the remaining $ e(F) - v(F) + 1 $ edges required to complete the tree to a copy of $ F $ are ever presented.
    The proof itself requires a little more subtlety, as formalising this argument depends upon the order in which the edges of a copy of $ F $ appear in the process.
\end{remark}

\begin{proof}[Proof of \cref{thm:lowerboundF}]
We will make use of the following index set:
\[J\coloneqq\left\{(k,m)\in[k_0]\times\mathbb{Z} : 0\leq m\leq\binom{k}{2}\right\}=\{j_1,\ldots,j_{h}\},\]
where $h\coloneqq\sum_{k=1}^{k_0}\big(1+\inbinom{k}{2}\big)$ and $J$ is labelled in lexicographic order.
This means that, for $j=(k,m)$ and $j'=(k',m')$ with $j,j'\in J$, we have that $j<j'$ if and only if $k<k'$ or $k=k'$ and $m<m'$.
We consider constants $\eta,\lambda'>0$ and, for each $j\in J$, additional constants $\eps_j,\lambda_j>0$ which we will use throughout the proof adhering to the following hierarchy:
\begin{equation}\label{eq:hierarchy}
\frac{1}{n_0},\frac{1}{\lambda}\ll\frac{1}{\lambda'} \ll \frac{1}{\lambda_{j_{h}}}\ll\cdots\ll\frac{1}{\lambda_{j_1}}\ll\eta\ll\eps_{j_1}\ll\cdots\ll\eps_{j_{h}}\ll\frac{1}{k_0},\eps.
\end{equation}
Moreover, $n_0$ is chosen sufficiently large with respect to $t$ that all subsequent inequalities where this is necessary hold for all $n\geq n_0$.\COMMENT{Note we do not believe that $ n_0 $ depends on anything in the hierarchy, only on $ t $. We choose not to change the statement, as it is correct as written.}
Furthermore, for more convenient calculations during the proof, we introduce the density parameters $ p \coloneqq t/ M \in (0, 1) $ and $ d \coloneqq 2b/n \in (0, n - 1] $.
Note that by \eqref{eq:hierarchy}, in particular,
\begin{equation}\label{eq:simplterm}
\lambda' d^{v(F) - 1} p^{e(F) - v(F) + 1}\leq \lambda b^{v(F) - 1} t^{e(F) - v(F) + 1} n^{v(F)-2e(F)-1}. 
\end{equation}

Recall that, for each $i\in[t]$, we write $ e_i $ for the edge received at time~$ i $ in the random graph process.
Recall also the definition of the update history at time $i$, which consists of a pair $(\vect{e}^{(i)},\vect{b}^{(i)})\in E^{(i)}\times\{0,1\}^i$ encoding the sequence of edges presented so far, as well as the choices made by Builder (see \cref{sect:notation}).
We say that a sequence $ U = (U_i)_{i \in [t]} $ of subsets of $ [n] $ which are nested (that is, such that $ U_1 \supseteq \cdots \supseteq U_t $) is a \defi{vertex filtration} in the budget-constrained random graph process (under some strategy) if $U_i$ depends only on the update history at time $i$.
In general, we think of $ U $ as a subset of the vertices which evolves over time, depending both upon the randomly presented edges and the behaviour of the strategy, but in which each $ U_i $ does not depend on anything that happens after time $ i $.
Given $ \nu \in [0, 1] $, we say that a set $ U' \subseteq [n] $ is \defi{$ \nu $-incomplete} if $ |U'| \ge (1 - \nu) n$, and further that a vertex filtration $ U $ is \defi{$ \nu $-incomplete} if $ U_t $ is $ \nu $-incomplete (which, in particular, implies that $ U_i $ is $ \nu $-incomplete for all $i\in[t]$).
Note that, for any positive integer $ \ell \le 1 / \nu $, the intersection of $ \ell $ different $ \nu$-incomplete vertex filtrations is an $ (\ell \nu) $-incomplete vertex filtration.

Now, let $B_t$ be the graph resulting from applying a given $(t,b)$-strategy. 
By \cref{lem:deterministic}, we may assume without loss of generality that the strategy is deterministic.
Let $ F $ be a connected graph on $ k \le k_0 $ vertices with $ k - 1 \le m \le \inbinom{k}{2} $ edges. Suppose $ F' \subseteq B_t $ is a copy of $ F $.
We must have $ E(F') = \{ e_{i_1}, \ldots, e_{i_m} \} $ for some indices $ 1\leq i_1 < \cdots < i_m \leq t $.
We say that the final edge~$ e_{i_m} $ \defi{completes}~$ F' $, or that $ F' $ is \defi{completed at time $ i_m $}.
Given a vertex filtration $ U $, we say that $ F' $ is \mbox{\defi{$ U $-contained}} if $ V(F') \subseteq U_{i_m} $, that is, if $ F' $ is contained in $ U $ at the time of its completion.
At each time step $ i \in [t] $, for each vertex $ v \in [n] $, we write $ a^{F, U}_i(v) $ for the number of $ U $-contained copies of $ F $ containing $ v $ which have been completed at any time $ i' \le i $, that is, completed by Builder purchasing the edge $ e_{i'} $.
Note that the definition of $ U $-contained ensures that, for any fixed $ v \in [n] $, the quantity $ a^{F, U}_i(v) $ is non-decreasing over time $ i $.

We now define vertex filtrations $ U^{k, m} $ for all $ k \in [k_0] $ and $ 0 \le m \le \inbinom{k}{2} $, by using induction on $k$ and, for each value of $k$, induction on $m$.
We start with $ U^{1, 0}$ as a base case, by setting $ U^{1, 0}_i \coloneqq [n] $ for all $ i \in [t] $.
For the inductive step, suppose first that $0\leq m\leq k-2$ and that the vertex filtration $U^{k-1,\genfrac{(}{)}{0pt}{2}{k-1}{2}}$ is already defined. 
In this case, we set $ U^{k, m} \coloneqq U^{k - 1, \genfrac{(}{)}{0pt}{2}{k-1}{2}} $.\COMMENT{The reason we define things like this is that there are no connected graphs on $ k $ vertices with fewer than $ k - 1 $ edges.}
Next, suppose that $k-1\leq m\leq\inbinom{k}{2}$ and that the vertex filtration $U^{k,m-1}$ is already defined.
Then, for each connected graph $ F $ on $ k $ vertices with $ m $ edges and each $ i \in [t] $, we define
\[ U_i^F \coloneqq \left\{ v \in U_i^{k,m-1} : a^{F, U^{k,m-1}}_i(v) \le \lambda_{k, m} d^{k - 1} p^{m - k + 1} \right\}, \]
and then set
\[ U^{k, m}_i \coloneqq \bigcap_F U_i^F, \]
where the intersection is over all connected graphs $ F $ on $ k $ vertices with exactly $ m $ edges.
Since $U^{k,m-1}$ is a vertex filtration and $ a^{F, U^{k,m-1}}_i(v) $ is non-decreasing, each sequence $ (U_i^{F})_{i \in [t]} $ is a vertex filtration too, and therefore so is $ U^{k, m} \coloneqq (U^{k, m}_i)_{i \in [t]} $.

Lastly, we need to define some events.
For each $ (k,m)\in J $ and $i\in[t]$, we write $ \mathcal{E}_{k, m, i} $ for the event that $ U^{k, m}_i $ is $ \eps_{k,m} $-incomplete, and $ \mathcal{E}_{k, m} \coloneqq \mathcal{E}_{k, m, t} $ for the event that $ U^{k, m} $ is $ \eps_{k,m} $-incomplete.
Note that,
\begin{equation}
\label{eq:eventscont}\text{for each }2\leq k \leq k_0\text{ and }0 \le m \le k - 2,\text{ we have that }\mathcal{E}_{k - 1, \genfrac{(}{)}{0pt}{2}{k-1}{2}}\subseteq\mathcal{E}_{k, m},
\end{equation}
since in this case $U^{k, m}=U^{k - 1, \genfrac{(}{)}{0pt}{2}{k-1}{2}}$ by the definition of the vertex filtrations in the previous paragraph and since $\eps_{k - 1, \genfrac{(}{)}{0pt}{2}{k-1}{2}}<\eps_{k,m}$ by~\eqref{eq:hierarchy}.

The main proof step is to establish that each event $ \mathcal{E}_{k, m} $ holds with sufficiently high probability.

\begin{claim}\label{clm:Ekm}
    For each $ k \in [k_0] $ and $0\leq m \leq\inbinom{k}{2}$, the event $ \mathcal{E}_{k, m} $ holds with probability at least $ 1 - \eps_{k, m} $.
\end{claim}

Observe that $\sum_{j\in J}\eps_{j}\leq h\eps_{j_{h}}<\eps$ by~\eqref{eq:hierarchy}.
Therefore, by \cref{clm:Ekm}, with probability at least $1-\eps$ all the events $ \mathcal{E}_{j} $, $j\in J$, hold simultaneously.
In that case, by the definition of $ \mathcal{E}_{j} $, for $V\coloneqq\bigcap_{j\in J}U^j_t = U^{j_h}_t$ we have that $|V|\geq(1-\eps_{j_h})n\geq(1-\eps)n$; moreover, by the definition of the vertex filtrations, we have that for every $v\in V$ the number of copies of any connected graph $F$ on at most $k_0$ vertices containing $v$ is at most $\lambda_{(v(F),e(F))} d^{v(F) - 1} p^{e(F) - v(F) + 1}<\lambda' d^{v(F) - 1} p^{e(F) - v(F) + 1}$, where we again appeal to~\eqref{eq:hierarchy}.
Thus, by~\eqref{eq:simplterm}, in order to prove the theorem it suffices to show \cref{clm:Ekm}.

\begin{claimproof}[Proof of Claim~\ref{clm:Ekm}]
The proof is by induction on $ (k, m) $ in the lexicographic ordering.
We take $ k = 1 $ and $ m = 0 $ as the base case for the induction, and note that this case holds trivially, since \mbox{$U^{1,0}=([n])_{i\in[t]}$} is $\eps_{1,0}$-incomplete with probability~$1$.
Suppose first that $2\leq k\leq k_0$ and $0\leq m\leq k-2$ and assume as the inductive hypothesis that $ \mathcal{E}_{k-1, \genfrac{(}{)}{0pt}{2}{k-1}{2}} $ holds with probability at least $ 1 - \eps_{k-1,\genfrac{(}{)}{0pt}{2}{k-1}{2}} $.
By~\eqref{eq:eventscont} and since $\eps_{k,m}>\eps_{k-1,\genfrac{(}{)}{0pt}{2}{k-1}{2}}$ (by \eqref{eq:hierarchy}), this implies that $ \mathcal{E}_{k, m} $ also holds with probability at least $1-\eps_{k,m}$. 
Now, let $2\leq k\leq k_0$ and $k-1\leq m\leq\inbinom{k}{2}$ and assume as the inductive hypothesis that the filtration $ U = U^{k, m - 1} $ is $ \eps_{k, m - 1} $-incomplete with probability at least $ 1 - \eps_{k, m - 1} $.

Let $ F $ be a connected graph on $ k $ vertices with $ m $ edges. Write
\[ C(F) \coloneq \{ e \in E(F) : F \setminus e \text{ is connected} \} \]
for the set of edges of $F$ which can be removed without disconnecting $ F $, and $ D(F) \coloneqq E(F) \setminus C(F) $ for the edges which disconnect $ F $.
Suppose that a $ U $-contained copy $ F' $ of $ F $ is completed by an edge~$ e_{i_m} $, corresponding to the edge $ e \in E(F) $.
We say that $ F' $ is \defi{$ U $-connected-contained} if $ e \in C(F) $, and \defi{$ U $-disconnected-contained} if $ e \in D(F) $.
Given a vertex $ v \in [n] $ and a time $i\in[t]$, write $ c^{F, U}_i(v) $ and $ d^{F, U}_i(v) $, respectively, for the number of $ U $-connected-contained and $ U $-disconnected-contained copies of $ F $ containing $ v $ which were completed at any time $ i' \le i $, so
\begin{equation}\label{equa:a=c+d_v2}
    a^{F, U}_i(v) = c^{F, U}_i(v) + d^{F, U}_i(v).
\end{equation}
Write $ a^{F, U}_i \coloneqq \sum_{v \in [n]} a^{F, U}_i(v) $, and likewise for $ c $ and $ d $.

Let us first bound the number of $U$-disconnected-contained copies of $F$.
Suppose that a copy~$ F' $ of~$F$ is completed by an edge $ e_i = x'y' $ corresponding to $ e = xy \in D(F) $.
Let $ F_x^e $ and $ F_y^e $ be the connected components of $ F \setminus \{ e \} $ containing~$ x $ and~$ y $, respectively.
Suppose that $ F_x^e $ and $ F_y^e $ have~$ k_x $ and~$ k_y $ vertices and $ m_x $ and $ m_y $ edges, respectively, so 
\begin{equation}\label{eq:ksummsum}
 k_x + k_y = k\qquad \text{ and  }\qquad m_x + m_y + 1 = m .
\end{equation}
Note that, in order for $ F' $ to be $ U $-contained, the copies of $ F_x^e $ and $ F_y^e $ from which it is completed must also be $ U $-contained, by the nested property of the $ U_i $.
As such, given $ x', y' \in U_{i - 1} $, purchasing the edge $ x'y' $ at time $ i $ completes at most
\[ a^{F_x^e, U}_{i - 1}(x') \cdot a^{F_y^e, U}_{i - 1}(y') \]
$ U $-contained copies of $ F $ in which $ x' $ and $ y' $ correspond to $ x $ and $ y $, respectively.
By the definition of~$ U $, \eqref{eq:ksummsum} and \eqref{eq:hierarchy}, this quantity is at most
\[ \lambda_{k_x, m_x } d^{k_x - 1} p^{m_x - k_x + 1} \cdot \lambda_{k_y, m_y} d^{k_y - 1} p^{m_y - k_y + 1} \leq \lambda_{k, m - 1}^2 d^{k - 2} p^{m - k + 1}. \]
Since we buy at most $b = d n / 2 $ edges $ x'y' $ in total, there are at most $ k^2 $ choices for $ (x, y) $ in $V(F)$, and each $U$-contained copy of~$F$ is counted precisely $k$ times when considering $d_t^{F,U}$, this means that
\begin{equation}\label{eqn:lowerbound_bbound_v2}
    d^{F, U}_t \le \frac{\lambda_{k, m - 1}^2 k^3}{2} n d^{k - 1} p^{m - k + 1}.
\end{equation}

Next, we similarly bound the number of $U$-connected-contained copies of $F$.
Note that, if a copy~$ F' $ of~$F$ completed by an edge $ e_i $ for some $ i \in [t] $ corresponding to $ e \in C(F) $ is $U$-contained, then the copy of the connected graph $ F \setminus \{ e \} $ (with $ m - 1 $ edges on $ k $ vertices) from which it arises must also be $ U $-contained.
Thus, by definition, there are at most $ a^{F \setminus \{ e \}, U}_{i - 1} $\COMMENT{Here we could divide by $k$ since each copy is counted $k$ times (once from each vertex). But it is not necessary.} such $ U $-contained copies $ F'' $ of $ F \setminus \{ e \} $ present immediately before the edge~$e_i$ is presented, each with a unique corresponding missing edge $ e'(F'') \in \inbinom{[n]}{2} $.\COMMENT{Note that we might also have copies of $ F \setminus \{ e \} $ for which the ``missing'' edge has already been offered; since we only care about an upper bound, we may simply ignore these, or treat them as if the edge was not missing.}
For each $e\in C(F)$, let~$ \mathcal{F}_e $ be the set of all copies of ~$ F \setminus \{ e \} $ in $K_n$, and let $\mathcal{F}\coloneqq\bigcup_{e\in C(F)}\mathcal{F}_e$.\COMMENT{Note that, since we consider labelled embeddings of labelled graphs, this set may contain multiple copies of the same graph on $K_n$ as copies of different instances of $F\setminus\{e\}$.}
For each~$F'' \in \mathcal{F}$, let~$ \cA_{F'', i} $ denote the event that $ F'' $ is completed to a copy~$ F' $ of~$F$ by~$ e_i $, that is, the event that $E(F'')\subseteq B_{i - 1}$ and $ e_i = e'(F'') $.

For each $i\in[t]$, let us write $ X_i $ for the number of $ U $-connected-contained copies of~$ F $ completed at time~$ i $.
Observe that
\begin{align}
    \mathbb{E} \left[ c^{F, U}_t \ \big|\  \mathcal{E}_{k, m - 1} \right] = \sum_{i = 1}^t \mathbb{E} \left[ X_i \ \big|\  \mathcal{E}_{k, m - 1} \right]
        &= \sum_{i = 1}^t \frac{\mathbb{E} \left[ X_i \mathds{1}[\mathcal{E}_{k, m - 1}] \right]}{\mathbb{P}[\mathcal{E}_{k, m - 1}]}
        \le \frac{1}{1 - \eps_{k, m - 1}} \sum_{i = 1}^t \mathbb{E} \left[ X_i \mathds{1}[\mathcal{E}_{k, m - 1}] \right]\nonumber\\
        &\le \frac{1}{1 - \eps_{k, m - 1}} \sum_{i = 1}^t \mathbb{E} \left[ X_i \mathds{1}[\mathcal{E}_{k, m - 1, i - 1}] \right].\label{equa:U-conn-cont-fix}
\end{align}
Note that, since the $(t,b)$-strategy is deterministic, $ \mathcal{E}_{k, m - 1, i - 1} $ depends only upon the edges $ e_1, \ldots, e_{i - 1} $, and let $ \mathcal{G}_i $ be the set of all possible outcomes corresponding to different values of $ (e_1, \ldots, e_{i - 1}) $ for which $ \mathcal{E}_{k, m - 1, i - 1} $ holds.
Given $ \mathcal{H} \in \mathcal{G}_i $, let $ \mathcal{F}(\mathcal{H}) $ be the set of copies $ F'' \in \mathcal{F} $ with $E(F'')\subseteq B_{i - 1}$ which are $ U $-contained conditional upon the event $ \mathcal{H} $.
By definition and summing over each possible $ e \in C(F) $ and $v\in U_i$, we have that 
\[
    |\mathcal{F}(\mathcal{H})| \le \sum_{e\in C(F)}a^{F \setminus \{ e \}, U}_{i - 1} \le m \cdot n \cdot \lambda_{k, m - 1} d^{k - 1} p^{m - k}.
\]
Moreover, for each $ F'' \in \mathcal{F}(\mathcal{H}) $, observe that\COMMENT{This is using the fact that $i\leq t=o(n^2)$, which guarantees that the number of possible choices for the next edge is $(1-o(1))M$.}
\[
    \mathbb{P} \left[ \mathcal{A}_{F'', i} \ \big|\  \mathcal{H} \right] = \mathbb{P} \left[ \{ e_i = e'(F'') \} \ \big|\  \mathcal{H} \right] = (1 + o(1)) \binom{n}{2}^{-1}.
\]
Hence, for any $ \mathcal{H} \in \mathcal{G}_i $, we have that
\[
    \mathbb{E} \left[ X_i \ \big|\  \mathcal{H} \right] = \sum_{F'' \in \mathcal{F}(\mathcal{H})} \mathbb{P} \left[ \mathcal{A}_{F'', i} \ \big|\  \mathcal{H} \right] \le m \cdot n \cdot \lambda_{k, m - 1} d^{k - 1} p^{m - k} \cdot (1 + o(1)) \binom{n}{2}^{-1},
\]
and so, in particular,
\[
    \mathbb{E} \left[ X_i \mathds{1}[\mathcal{E}_{k, m - 1, i - 1}] \right] \le (1 + o(1)) \lambda_{k, m - 1} m d^{k - 1} p^{m - k} \cdot \frac{n}{\binom{n}{2}}.
\]
Substituting this into \eqref{equa:U-conn-cont-fix} and using \eqref{eq:hierarchy} yields
\[
    \mathbb{E} \left[ c^{F, U}_t \ \big|\  \mathcal{E}_{k, m - 1} \right] \le 2 p \binom{n}{2} \cdot \lambda_{k, m - 1} m d^{k - 1} p^{m - k} \cdot \frac{n}{\binom{n}{2}}.
\]
By Markov's inequality, this means that
\begin{equation}\label{eqn:lowerbound_cbound_v2}
    \mathbb{P}\left[c^{F, U}_t \le \eta^{-1} \cdot 2 \lambda_{k, m - 1} m n d^{k - 1} p^{m - k + 1}\ \big|\ \cE_{k,m-1}\right]\geq1-\eta.
\end{equation}

Combining \eqref{eqn:lowerbound_bbound_v2} and \eqref{eqn:lowerbound_cbound_v2} via \eqref{equa:a=c+d_v2} ensures that, by setting $\beta_{k,m}\coloneqq2\lambda_{k, m - 1} m\eta^{-1}+\lambda_{k,m-1}^2k^3/2$,
\[ \mathbb{P}\left[a^{F, U}_t \le \beta_{k,m} n d^{k - 1} p^{m - k + 1}\ \big|\ \cE_{k,m-1}\right]\geq1-\eta. \]
In particular, if the event above holds, there are at most $ \eps_{k, m - 1} n $ vertices $ v \in U_t $ for which
\[ a^{F, U}_t(v) \ge \beta_{k,m} \eps_{k, m - 1}^{-1} d^{k - 1} p^{m - k + 1}, \]
so, since by~\eqref{eq:hierarchy} we have that
\[ \lambda_{k, m} > (2m\lambda_{k,m-1}\eta^{-1}+\lambda_{k,m-1}^2k^3/2)\eps_{k, m - 1}^{-1} = \beta_{k,m} \eps_{k, m - 1}^{-1},\]
we conclude that, conditional upon $ \mathcal{E}_{k, m - 1} $, we have $|U_t^F| \ge (1 - 2\eps_{k, m - 1}) n$. 
Hence, summing over the at most $ k^{2m} $ many possible graphs~$ F $, we conclude that, conditional upon $ \mathcal{E}_{k, m - 1} $, with probability at least $ 1 - k^{2m} \eta $, the vertex filtration $ U^{k,m} $ is $ \eps_{k,m} $-incomplete, as $ \eps_{k,m} > 2 k^{2m} \eps_{k, m - 1} $ by~\eqref{eq:hierarchy}.
Therefore, by the induction hypothesis, $ \mathcal{E}_{k, m} $ holds with probability at least $1-(k^{2m}\eta + \eps_{k,m-1})\geq 1 - \eps_{k,m} $\COMMENT{We have that $\mathbb{P}[\mathcal{E}_{k,m-1}]\geq1-\eps_{k,m-1}$ by the induction hypothesis and that $\mathbb{P}[\mathcal{E}_{k,m}\mid\mathcal{E}_{k,m-1}]\geq1-k^{2m}\eta$. 
Therefore, by the law of total probability,
\[\mathbb{P}[\mathcal{E}_{k,m}]\geq\mathbb{P}[\mathcal{E}_{k,m}\mid\mathcal{E}_{k,m-1}]\mathbb{P}[\mathcal{E}_{k,m-1}]\geq(1-k^{2m}\eta)(1-\eps_{k,m-1})\geq1-(k^{2m}\eta + \eps_{k,m-1}).\]} (again by~\eqref{eq:hierarchy}), completing the inductive step.
\end{claimproof}

As shown above, this also concludes the main proof.
\end{proof}


\section{Successful strategies for graph factors}\label{section:factors_upper}

In this section we present our successful strategies for $F$-factors, including the proofs of \cref{thm:factor_upper,thm:partial_factor_upper}.
Recall from the introduction that the threshold for $F$-factors in (binomial) random graphs, when $F$ is strictly $1$-balanced, was determined by \citet{JKV08}.
The following version of their result gives an effective bound on the probability that the random graph does not contain an $F$-factor.

\begin{theorem}[{\citet[Theorem~2.3]{JKV08}}]\label{thm:JKVbound}
    Let $F$ be a fixed strictly $1$-balanced graph.
    For every constant $C_1>0$, there exists a constant $C_2>0$ such that, if $p\geq C_2n^{-1/d(F)}\log^{1/e(F)}n$ and~$n$ is a multiple of $v(F)$, then the probability that $G(n,p)$ does not contain an $F$-factor is $O(n^{-C_1})$.
\end{theorem}

In fact, \citet{JKV08} obtained an almost tight upper bound on the threshold for $F$-factors for any fixed graph $F$.
The following is an effective version which follows from their proof.

\begin{theorem}[{\citet[effective version of Theorem~2.2]{JKV08}}]\label{thm:JKVGenCase}
    Let $F$ be a fixed graph and $\eps>0$.
    If $p\geq n^{-1/d^*(F)+\eps}$ and $n$ is a multiple of $v(F)$, then the probability that $G(n,p)$ does not contain an $F$-factor is at most $n^{-\omega(1)}$.
\end{theorem}

Recall also that \citet{Ruc92} was the first to consider (partial) $F$-factors in random graphs, and he obtained the threshold for $G(n,p)$ to contain an $\alpha$-$F$-factor for any $\alpha\in(0,1)$.
In fact, by following his proof, we obtain the following effective bound on the probability that $G(n,p)$ does not contain an $\alpha$-$F$-factor.

\begin{theorem}[{\citet[Theorem~4~(a)]{Ruc92}}]\label{thm:Rucbound}
    Let $F$ be a fixed graph.
    For every constant $\alpha\in(0,1)$, there exists a constant $C>0$ such that, if $p\geq Cn^{-1/d^*(F)}$ and $n$ is sufficiently large, then the probability that $G(n,p)$ does not contain an $\alpha$-$F$-factor is at most $2^{-n}$.
\end{theorem}

In a similar way, we can retrieve an effective bound for the probability that $G(n,p)$ does not contain an $F$-factor when $F$ is not vertex-balanced from the work of \citet{GM15}.
For a fixed graph $F$ and any $v\in V(F)$ we write
\[d^*(v,F)\coloneqq\max\{d(F'):F'\subseteq F, v(F')\geq2,v\in V(F')\}.\]
We say that $F$ is \defi{vertex-balanced} if $d^*(v,F)=d^*(F)$ for every $v\in V(F)$.
Note that every (strictly) $1$-balanced graph is also vertex-balanced, so a result about thresholds for $F$-factors when $F$ is not vertex-balanced is not covered by the work of \citet{JKV08}.
By retracing the proof of \citet{GM15} (applying our \cref{thm:Rucbound} at the end of the proof of \cite[Lemma~3.2]{GM15} and then \cite[Theorem~4.1]{GM15}), we reach the following statement.

\begin{theorem}[\citet{GM15}]\label{thm:GMbound}
    Let $F$ be a fixed non-vertex-balanced graph.
    For every constant $C_1>0$, there exists a constant $C_2>0$ such that, if $p\geq C_2n^{-1/d^*(F)}$ and $v(F)$ divides~$n$, then the probability that $G(n,p)$ does not contain an $F$-factor is at most $O(n^{-C_1})$.\COMMENT{A bit of justification for this theorem: Essentially, the proof starts by first finding an appropriate $\alpha$-$H$-factor, where $H$ is a subgraph of $F$ of maximum density. They simply need to have one a.a.s., but we want to have some control over the probability of failure; we can get one with very high probability by applying our \cref{thm:Rucbound} instead of what they do in the proof of their Lemma~3.2.
    Then, they need to ensure that they can embed the remaining edges, also in a sort of ``factor-like'' way. For this, they essentially use (a generalisation of) the theorem of \citet{JKV08}, for which we have \cref{thm:JKVbound}, which gives us effective bounds on the probability too. In their paper, this is written differently as Theorem~4.1; it should be possible to rewrite this similarly to our \cref{thm:JKVbound} via a standard argument (which can be found in the appendix of \cite{JKV08}).}
\end{theorem}

In the $(t,b)$-strategies that we consider below, we will partition the vertex set of the random graph into smaller subsets.
Our aim will be to find an $F$-factor in each of the subsets, so we will consider $F$-equipartitions (recalling the definition from \cref{sect:notation}).
We will simply purchase all edges that fall inside each one of the sets of the partition.
Roughly speaking, this allows us to treat the graph resulting from each of the parts as an independent binomial random graph whose probability $p$ is ``boosted'' compared to the average density of $B_t$.
Thus, we may apply the results about random graphs to each of these sets to prove that a.a.s.\ this leads to a graph $B_t$ with the desired $F$-factor.

The following result is a simple corollary of \cref{thm:JKVbound} and shows that \cref{thm:lowerboundF} is best possible, up to polylogarithmic factors.
\Cref{thm:factor_upper} is a particular case of this more general result.
We remark that, in \cref{thm:factoruppergeneral}, as well as \cref{thm:partialfactoruppergeneral,thm:factoruppergeneralnon-balanced} below, the constant $ K_1 $ does not really depend on~$ \delta $; this dependency in our statements is a quirk of our proof, and a more general version can be easily deduced by applying our statement twice.

\begin{theorem}\label{thm:factoruppergeneral}
    Let $F$ be a fixed strictly $1$-balanced graph.\COMMENT{Note that the only strictly $1$-balanced tree is the graph consisting of a single edge (so, the result is about perfect matchings). The result we obtain in this case is not very good, but it is correct. All other strictly $1$-balanced graphs contain at least one cycle, and thus have $1$-density larger than $1$ and the optimal $(t,b)$-strategies must have parameters exhibiting the ``slope'' we have for cliques.}
    For every $\delta>0$, there exist constants $K_1,K_2>0$ such that, if $K_1n^{2-1/d(F)}\log^{1/e(F)}n\leq t\leq n^{2-\delta}$ with $n$ restricted to multiples of~$v(F)$, then there exists a successful $(t,K_2t^{1-d(F)}n^{2d(F)-1}\log^{1/(v(F)-1)}n)$-strategy for an $F$-factor.
\end{theorem}

\begin{proof}
    Let $\delta>0$ be fixed.
    Let $C_1\coloneqq 1/(\delta d(F))$, and let $C_2$ denote the constant given by \cref{thm:JKVbound} with inputs $F$ and $C_1$.
    Let $K>2C_2^{d(F)}$ and choose $K_1> K^{1/d(F)}/2$ and $K_2=9K$.
    Let $b\coloneqq K_2t^{1-d(F)}n^{2d(F)-1}\log^{1/(v(F)-1)}n$.
    By \cref{lem:basiccoupling}, there exist a $p=(1+o(1))t/M$ and a coupling $(G_t,G)$ such that $G\sim G(n,p)$ and a.a.s.\ $G_t\subseteq G$; we will use this coupling to bound the number of purchased edges below.
    Set
    \[k=k(n)\coloneqq\left\lfloor\frac{p^{d(F)}n}{K\log^{1/(v(F)-1)}n}\right\rfloor.\]
    Observe that our choice of $K_1$ guarantees that $k\geq 1$.\COMMENT{Actually turns out that we could choose $K_1=C_2$, if $K=2C_2^{d(F)}$ were true (which it is not). Here the calculation: 
    \begin{align*}
    &\;\frac{p^{d(F)}n}{K\log^{1/(v(F)-1)}n}\geq 1\;
    \Leftrightarrow\;p^{d(F)}\geq K\frac{\log^{1/(v(F)-1)}n}{n}\\
    \Leftarrow\;&(1-o(1))t/\binom{n}{2}\geq K^{1/d(F)}\frac{\log^{1/e(F)}n}{n^{1/d(F)}}\;
    \Leftarrow\;(1-o(1))2K_1n^{-1/d(F)}\log^{1/e(F)}n\geq K^{1/d(F)}\frac{\log^{1/e(F)}n}{n^{1/d(F)}}\\
    \Leftarrow\;& K_1>K^{1/d(F)}/2\;.
    \end{align*}
    }
    Consider an arbitrary $F$-equipartition $[n]=A_1\cupdot\cdots\cupdot A_k$ of the set of vertices, so that for all $j\in[k]$ we have that\COMMENT{We simply want general upper and lower bounds. Because of the floors, it may be that the sets have size at most twice as large as $n/(p^{d(F)}n)$, and we write $3$ to avoid any issue with the rounding for the upper bound. For the lower bound, the $2$ is only there to deal with the rounding.} 
    \[\frac{K}{2}\frac{\log^{1/(v(F)-1)}n}{p^{d(F)}}\leq|A_j|=(1\pm o(1))\frac{n}{k}\leq 3K\frac{\log^{1/(v(F)-1)}n}{p^{d(F)}}.\]
    Note, in particular, that $k=o(n^{1-\delta d(F)})$ and $|A_j|=\omega(n^{\delta d(F)})$.
    
    Consider the following $(t,b)$-strategy: at each time $i\in[t]$, when the edge $e_i$ is presented, purchase it if and only if $e(B_{i-1})<b$ and there exists some $j\in[k]$ such that $e_i\subseteq A_j$.
    The coupling of random graphs introduced above ensures that a.a.s.\ $B_t\subseteq\bigcup_{j=1}^kG[A_j]$, where $G[A_j]\sim G(|A_j|,p)$.
    By applying Chernoff's bound (\cref{lem:Chernoff}), we thus have that a.a.s.\COMMENT{We apply Chernoff to each graph. We are bounding the probability of having (roughly) twice the expected number of edges in each $G[A_j]$, which by Chernoff can be bounded by $\exp(-\Theta(|A_j|^2p))=\exp(-\tilde{\Theta}(p^{1-2d(F)}))$. By the upper bound on $t$, we have that $p$ is polynomial in $n$ (meaning, never too close to $1$ asymptotically), so a union bound ensures that a.a.s.\ the bound holds simultaneously for all graphs. This gives the first inequality.
    For the last inequality, we use that $p\geq t/M\geq t/n^2$, though we never made this explicit anywhere (since the exponent of $p$ is negative, to obtain an upper bound on the expression we must use a lower bound on $p$).}
    \begin{align*}
        e(B_t)<\sum_{j=1}^k|A_j|^2p&\leq\frac{p^{d(F)}n}{K\log^{1/(v(F)-1)}n}\cdot9K^2\frac{\log^{2/(v(F)-1)}n}{p^{2d(F)}}\cdot p\\
        &\leq9Kt^{1-d(F)}n^{2d(F)-1}\log^{1/(v(F)-1)}n=b.
    \end{align*}
    In particular, it follows that a.a.s.\ $B_t=\bigcup_{j=1}^kG_t[A_j]$.
    
    Now, by another application of \cref{lem:basiccoupling}, there exists a $p'=(1-o(1))t/M$ and a coupling $(G',G_t)$ such that $G'\sim G(n,p')$ and a.a.s.\ $G'\subseteq G_t$; in particular, a.a.s.\ $\bigcup_{j=1}^kG'[A_j]\subseteq B_t$.
    Note that, for each $j\in[k]$, we have that $G'[A_j]\sim G(|A_j|,p')$.
    For a fixed $j\in[k]$, a straightforward calculation using the definition of $K$ confirms that $p'\geq C_2|A_j|^{-1/d(F)}\log^{1/e(F)}|A_j|$,\COMMENT{We must verify that \[p'\geq C_2|A_j|^{-1/d(F)}\log^{1/e(F)}|A_j|.\]
    Substituting the lower bound on $|A_j|$ in the polynomial term, it suffices to verify that 
    \[p'\geq C_2(2/K)^{1/d(F)}p\log^{-1/e(F)}n\log^{1/e(F)}|A_j|.\]
    But this clearly holds, since 
    \[C_2(2/K)^{1/d(F)}\log^{-1/e(F)}n\log^{1/e(F)}|A_j|\leq C_2(2/K)^{1/d(F)}<1-o(1)=p'/p\]
    by the definition of $K$.} so by \cref{thm:JKVbound} the probability that $G'[A_j]$ does not contain an $F$-factor is $O(|A_j|^{-C_1})=o(n^{-\delta d(F)C_1})=o(1/n)$.
    Thus, by a union bound over all $j\in[k]$, a.a.s.\ all $G'[A_j]\sim G(|A_j|,p')$ contain an $F$-factor, so we conclude that $B_t$ does too.
    Therefore, the statement holds for our choice of $K_1$ and $K_2$.
\end{proof}

By using \cref{thm:Rucbound} instead of \cref{thm:JKVbound}, we obtain the following result, which shows that \cref{thm:lowerboundF} is best possible, up to constant factors.
Since the proof is analogous to that of \cref{thm:factoruppergeneral}, we omit the details.
\Cref{thm:partial_factor_upper} is a particular case of this more general result.

\begin{theorem}\label{thm:partialfactoruppergeneral}
    Let $F$ be a fixed graph and $\alpha\in(0,1)$.
    For every $\delta>0$, there exist constants $K_1,K_2>0$ such that, if $K_1n^{2-1/d^*(F)}\leq t\leq n^{2-\delta}$,\COMMENT{Based on the proof, I think here this $n^{2-\delta}$ (with $\delta$ being a constant) can be improved to anything where $n^\delta$ is larger than any polylogarithm, i.e., to $\delta=\delta(n)=\omega(\log\log n/\log n)$. (This is because the effective bound in \cref{thm:Rucbound} is much stronger than the bound in \cref{thm:JKVbound}.)} then there exists a successful $(t,K_2t^{1-d^*(F)}n^{2d^*(F)-1})$-strategy for an $\alpha$-$F$-factor.
\end{theorem}

Lastly, following the same proof but applying \cref{thm:GMbound} instead of \cref{thm:JKVbound}, we also obtain a successful strategy for $F$-factors, where $F$ is not vertex-balanced, which again matches the lower bound provided by \cref{thm:lowerboundF}.
In fact, this shows that, in general, the bounds in \cref{thm:Ffactors_lower} cannot be improved by more than a constant factor when considering complete $F$-factors.
We again omit the details of the proof.\COMMENT{We define $C_1$ analogously as in the proof of \cref{thm:factoruppergeneral}, let $C_2$ be the constant given by \cref{thm:GMbound} with inputs $F$ and $\delta$, and define $K$ analogously as in the proof of \cref{thm:factoruppergeneral}. We then define $k$ (and thus an equipartition) analogously as in the proof of \cref{thm:partialfactoruppergeneral}. The rest proceeds analogously.}

\begin{theorem}\label{thm:factoruppergeneralnon-balanced}
    Let $F$ be a fixed non-vertex-balanced graph.
    For every $\delta>0$, there exist constants $K_1,K_2>0$ such that, for $n$ restricted to multiples of $v(F)$, if $K_1n^{2-1/d^*(F)}\leq t\leq n^{2-\delta}$, then there exists a successful $(t,K_2t^{1-d^*(F)}n^{2d^*(F)-1})$-strategy for an $F$-factor.
\end{theorem}

We remark that, in more generality, using \cref{thm:JKVGenCase} instead of \cref{thm:JKVbound}, the same proof strategy as for \cref{thm:factoruppergeneral} yields successful $(t,b)$-strategies for $F$-factors, for any fixed graph $F$, where the upper bound on $b$ is tight (with respect to that given in \cref{thm:Ffactors_lower}) up to subpolynomial factors.
Any improvement on \cref{thm:JKVGenCase} for any family of graphs $F$ will lead to corresponding improvements for successful $(t,b)$-strategies.
In fact, when considering couplings in the next section, it will be useful to have the following lemma for binomial random graphs, which is a by-product of the proof of this general upper bound for successful strategies.
As this lemma will be relevant later, we include its proof for completeness.

\begin{lemma}\label{lem:factorGnp}
    Let $F$ be a fixed graph.
    For every $\delta>0$, for all $n^{-1/d^*(F)+\delta}\leq p\leq n^{-\delta}$ with $n$ restricted to multiples of\/~$v(F)$, there exist $\theta=\theta(n)\in\mathbb{N}$ with $\theta = \Theta(p^{d^*(F)}n^{1-\delta d^*(F)/2})$ and $\delta'=\delta'(n)=o(1)$ such that the following statements hold for any $F$-equipartition $[n]=T_1\cupdot\cdots\cupdot T_\theta$.
    \begin{enumerate}[label={\ensuremath{(\mathrm{F}\arabic*)}}]
        \item\label{lem:factorGnpitem1} If\/ $G_1\sim G(n,(1-o(1))p)$, then with probability at least $1-\delta'$ we have that $G_1$ contains an $F$-factor which is entirely contained in $\bigcup_{i\in[\theta]} G_1[T_i]$.
        \item\label{lem:factorGnpitem2} If\/ $G_2\sim G(n,(1+o(1))p)$, then a.a.s.\ $\lvert\bigcup_{i\in[\theta]} E(G_2[T_i])\rvert<p^{1-d^*(F)}n^{1+\delta d^*(F)/2}$.
    \end{enumerate}
\end{lemma}

\begin{proof}
Let $\delta>0$ be fixed and sufficiently small.
Set $\theta=\theta(n)\coloneqq p^{d^*(F)}n^{1-\delta d^*(F)/2}$.\COMMENT{Note that $\theta=\omega(1)$ so we do not need floors here.}
Consider an arbitrary $F$-equipartition $[n]=T_1\cupdot\cdots\cupdot T_\theta$ of the set of vertices, so that for all $j\in[\theta]$ we have that
\[|T_j|=(1\pm o(1))\frac{n}{\theta}=(1\pm o(1))n^{\delta d^*(F)/2}p^{-d^*(F)}.\]
Note, in particular, that $\theta=O(n^{1-3\delta d^*(F)/2})$ and $|T_j|=\Omega(n^{3\delta d^*(F)/2})$.

By Chernoff's inequality (\cref{lem:Chernoff}), we have that a.a.s.\COMMENT{Note that the expected number of edges in each $G_2[T_j]$ is roughly of the order of magnitude of 
\[|T_j|^2p=n^{\delta d^*(F)}p^{1-2d^*(F)}\geq n^{\delta d^*(F)}\]
where in the equality we are ignoring smaller order terms, and the inequality uses the fact that $d^*(F)\geq1$ for every $F$, so $1-2d^*(F)<0$ and, to obtain a lower bound on the expression, we make use of the trivial upper bound $p\leq1$.
Therefore, when applying Chernoff bounds, we have enough concentration for a union bound.}
\[\left\lvert\bigcup_{j\in[\theta]} E(G_2[T_j])\right\rvert<\sum_{j=1}^\theta\frac23|T_j|^2p\leq p^{d^*(F)}n^{1-\delta d^*(F)/2}\cdot n^{\delta d^*(F)}p^{-2d^*(F)}\cdot p=p^{1-d^*(F)}n^{1+\delta d^*(F)/2},\]
thus proving \ref{lem:factorGnpitem2}.

On the other hand, letting $p'=(1-o(1))p$ so that $G_1\sim G(n,p')$, for each $j\in[\theta]$ we have that $G_1[T_j]\sim G(|T_j|,p')$.
For a fixed $j\in[\theta]$, a straightforward calculation confirms that $p'\geq |T_j|^{-1/d^*(F)+\eta}$ for some constant $\eta>0$.\COMMENT{Indeed, given the bounds on $|T_j|$ and $p'$, it suffices to verify that \[p\geq2\left(n^{\delta d^*(F)/2}p^{-d^*(F)}\right)^{-1/d^*(F)+\eta}=2\left(\frac{p}{n^{\delta/2}}\right)^{1-\eta d^*(F)}\iff p^{\eta d^*(F)}\geq2n^{-\delta(1-\eta d^*(F))/2}.\]
Since $p\geq n^{-1/d^*(F)+\delta}$, it suffices to verify that
\[n^{-\eta d^*(F)(1/d^*(F)-\delta)}\geq2n^{-\delta(1-\eta d^*(F))/2},\]
which holds for $n$ sufficiently large if
\begin{align*}
    \eta d^*(F)\left(\frac{1}{d^*(F)}-\delta\right)<\frac{\delta}{2}(1-\eta d^*(F))&\iff\eta d^*(F)\left(\frac{1}{d^*(F)}-\frac{\delta}{2}\right)<\frac{\delta}{2}\\
    &\iff\eta<\frac{\delta}{2d^*(F)\left(\frac{1}{d^*(F)}-\frac{\delta}{2}\right)},
\end{align*}
and such positive $\eta$ exist by choosing $\delta$ sufficiently small, as we wanted to see.
(Note that all the places where one would need $\delta$ sufficiently small actually correspond to cases where, for larger values of $\delta$, we would fall outside the allowed range for $p$, and thus the statement would be vacuously true.)}
Thus, by \cref{thm:JKVGenCase} with $\eta$ playing the role of $\eps$, the probability that $G_1[T_j]$ does not contain an $F$-factor is at most $|T_j|^{-\omega(1)}=n^{-\omega(1)}$.
Thus, by a union bound over all $j\in[\theta]$, a.a.s.\ each $G_1[T_j]\sim G(|T_j|,p')$ contains an $F$-factor, which results in the desired global $F$-factor with probability at least $1-\delta'$, where $\delta'=n^{-\omega(1)}$, completing the proof of \ref{lem:factorGnpitem1}.
\end{proof}

\section{Successful strategies for powers of Hamilton cycles}\label{sect:powers}

The successful $(t,b)$-strategies showcased in the proofs of \cref{thm:factoruppergeneral,thm:partialfactoruppergeneral,thm:factoruppergeneralnon-balanced} rely on the simple fact that, by our coupling to $ G(n, p) $, if we split the vertex set $ [n] $ into sets of size $ m = o(n) $ and purchase all edges which fall inside each of the sets, we may think of the graph induced on each of the sets as a binomial random graph $ G(m, p'(m)) $, where $ p'(m) = p(n) = o(p(m)) $.
By choosing the sizes $ m $ so that $ p'(m) $ is slightly above the threshold for having an $F$-factor, we obtain an $F$-factor in each of the smaller graphs and, thus, a complete $F$-factor (while reducing the total number of purchased edges).
We believe this represents the simplest case of what may be a more general phenomenon: that if there exists a constructive proof of the existence of a spanning structure $H$ in $G(n,p)$ where the construction relies on joining small local structures (where by ``constructive'' we do not necessarily mean that the small local structures can be built explicitly, but only that the global structure can be built explicitly by joining the small local structures), then it is possible to adapt such a proof to obtain successful $(t,b)$-strategies for Builder to construct a graph containing a copy of $H$.
As one further example of this phenomenon, we consider powers of Hamilton cycles.

Recall that, for $k\geq2$, the threshold for containing the $k$-th power of a Hamilton cycle is $n^{-1/k}$~\cite{riordan00,KO12,KNP21}.
The arguments of \citet{riordan00} and \citet{KNP21} do not provide a description of the construction of such a structure.
\citet{KO12} provided a constructive proof slightly above the threshold, showing that a.a.s.\ $G(n,n^{-1/k+\eps})$ contains the $k$-th power of a Hamilton cycle, for all fixed $\eps>0$.
Here, we adapt the proof strategy of \citet{KO12} to obtain nearly optimal successful $(t,b)$-strategies for powers of Hamilton cycles.
We will prove the following statement, which is equivalent to \cref{thm:hamsquare_upper}.

\begin{theorem}\label{thm:upperboundPowers}
    Fix an integer $k\geq2$.
    For every $\eps>0$ and every $t=t(n)=\omega(n^{2-1/k})$,
    there exists $b\leq n^{2k-1+\eps}/t^{k-1}$ such that there is a successful $(t,b)$-strategy for the $k$-th power of a Hamilton cycle.
\end{theorem}

\subsection{Notation and proof sketch}

Let us describe the notation we need for proving this theorem.
Given an $\ell$-vertex path $P$ (which we in general denote by~$P_\ell$), we denote its vertices in a sequence $P=x_1\ldots x_\ell$, and view $x_1$ as the first vertex of $P$ and $x_\ell$ as its last vertex.
Suppose that $\ell\geq k\geq2$.
When considering the $k$-th power~$P^k$ of~$P$, we refer to $x_1,\ldots,x_k$ as the \defi{initial endsequence} of $P^k$, and to $x_{\ell-k+1},\ldots,x_\ell$ as its \defi{final endsequence}.

The proof of \cref{thm:upperboundPowers} relies on the absorption method.
We will construct a set of \emph{absorbers} which can be used to incorporate (or \emph{absorb}) a few ``leftover'' vertices into the $k$-th power of a large cycle.
To be more precise, for integers $k\geq2$, $j\geq3$ and $\ell\geq2k$, we consider a \defi{$(j,\ell,k)$-absorber} $\mathcal{A}_{j,\ell,k}$, which is a graph consisting of the union of the $k$-th powers of two paths $P$ and $Q$, where $P$ is a path on $s\coloneqq j(2\ell+4)+\ell$ vertices which we call the \defi{spine}, $V(Q)=V(P)\cup\{v\}$ where $v\notin V(P)$ is the \defi{absorption vertex}, and the initial endsequences and the final endsequences of $P^k$ and $Q^k$ coincide.
The path $Q$ is called the \defi{augmented path}.
For an explicit description of $(j,\ell,k)$-absorbers, see \cite[Section~3]{KO12}.
While we do not need to present the explicit definition of $(j,\ell,k)$-absorbers, we will want to use the following property that they satisfy.

\begin{lemma}[{\cite[Lemma~6.1]{KO12}}]\label{lem:absorberdensity}
    For every $k\geq2$ and $\delta>0$, there exist $j=j(k,\delta)\geq3$ and \mbox{$\ell_0=\ell_0(k,\delta)\geq2k$} such that, for all $\ell\geq\ell_0$, we have that $d^*(\mathcal{A}_{j,\ell,k})\leq k+\delta$.
\end{lemma}

Crucially, if a graph $G$ contains a $(j,\ell,k)$-absorber with spine $P$, augmented path $Q$, and absorption vertex $v$, and contains the $k$-th power of a cycle which itself contains $P^k$ but avoids $v$, then, by replacing $P^k$ by $Q^k$, we obtain the $k$-th power of a cycle which now contains $v$ and otherwise has the same vertices as the original.
This process can be iterated to absorb multiple vertices into a long power of a cycle.
With this property in mind, the proof of \cref{thm:upperboundPowers} (as well as the main theorem of \citet{KO12}) goes roughly as follows.
We consider four stages.
In the first stage, for suitable parameters $j$ and $\ell$, Builder finds a $(1/3)$-$\mathcal{A}_{j,\ell,k}$-factor.
We let $A$ denote the set of absorption vertices from all the copies of $\mathcal{A}_{j,\ell,k}$ Builder constructed, and consider the $k$-th power of the spine of each of them.
In the second stage, Builder joins all the powers of the spines into a single $k$-th power of a path, by finding \emph{linking structures} between them (see \cref{sect:linking} for the definition).
These linking structures will use vertices not covered in the first stage, and in total cover a small constant fraction of the vertices.
In the third stage, Builder covers all the remaining vertices with vertex-disjoint copies of the $k$-th power of a path of a suitable length.
In the fourth stage, analogously to what she did in the second, Builder finds linking structures between all the current powers of paths, yielding the $k$-th power of a very long cycle.
In this stage, all the new vertices of the linking structures that are used belong to $A$.
At this point, any vertices of $A$ that do not belong to the $k$-th power of the cycle can be absorbed into it by replacing the $k$-th power of their respective spines by the $k$-th power of the augmented path in each corresponding absorber.
In each of the rounds described above, Builder will actually find the desired structures restricting the purchased edges to those spanned by sets of an appropriate size, so that the total number of edges that she considers does not exceed the budget.

As discussed in \cref{sect:couplings}, to prove that our strategy succeeds, we will apply \cref{lem:two_stage_real} to show that, in the third stage, a.a.s.\ we can obtain a factor of $ k $-th powers of paths covering all remaining vertices, after conditioning upon having found a particular set of linking structures in the second stage.
The very general setup of \cref{lem:two_stage_real} is important here since we use (one of) the main results of \citet{JKV08} (\cref{thm:JKVGenCase}) as a black box to find factors (of $ k $-th powers of paths) in binomial random graphs.
However, the major drawback of this approach is that, with our formulation, it is inherently limited to strategies in which at most two stages interact with each other in a non-edge-disjoint way (in this case, the second and third stages).
As such, we require a different approach to prove the success of the fourth stage in finding the desired linking structures, conditional upon having obtained a given collection of powers of paths in previous stages; we explain this below in \cref{sect:linking}.

\subsection{Linking structures} \label{sect:linking}

Let us focus first on the linking structures mentioned above.
Given \mbox{$ k \ge 2 $} and $ V \subseteq [n] $, define a \defi{$ k $-endsequence pair in~$ V $} to be a pair $ (A, B) = ((a_1,\ldots,a_k), (b_1,\ldots,b_k)) $ such that all $2k$ vertices are pairwise distinct and lie in~$ V $ (if the set~$V$ is not specified, one should interpret that we refer to $k$-endsequence pairs in~$[n]$).
Given further $ \eta \in \mathbb{N} $, define a \defi{$(k, \eta)$-endsequence family in $ V $} to be a set of $ \eta $ pairwise disjoint $ k $-endsequence pairs $ (A_i, B_i)_{i \in [\eta]} $ in $ V $.
Now, for any integer $r \geq 0$, an \defi{$(A,B)$-linkage} $L$ of \defi{length} $r$ is defined as follows.
First, we consider the $k$-th power $L'$ of a path on $r+2k$ vertices whose initial and final endsequences are $A$ and $B$, respectively.
Then we obtain~$L$ by removing all edges of $L'$ spanned by $A$ and by $B$.
We refer to the vertices of an $(A,B)$-linkage which are contained neither in $A$ nor in $B$ as its \defi{internal vertices} (so the length of an $(A,B)$-linkage coincides with its number of internal vertices).
Given $ r \in \mathbb{N} $, a $(k, \eta)$-endsequence family $ \mathcal{L} $ in $V$ and a set $ U \subseteq [n]\setminus V$, an \defi{$ \mathcal{L} $-linkage of length $ r $ in $ U $} is a set consisting of one $ (A_i, B_i) $-linkage for each $i\in[\eta]$, all of length~$ r $ and pairwise vertex-disjoint, with internal vertices in~$ U $.

We devote this section to stating and proving some auxiliary results which will guarantee we can find the desired linkages for the proof of \cref{thm:upperboundPowers}.
In order to state these results, we must first introduce some more ideas required to formally prove that our strategy succeeds.
As discussed previously, in order to prove \cref{thm:upperboundPowers} we will employ a strategy consisting of four stages, with each subsequent stage building upon the structures found during the previous stages.
Recall that our coupling (\cref{lemma:coupling_stages}) allows us to think of the edge set presented during the fourth stage of the process as containing $ G \setminus G' $, where $ G' \sim G(n, p') $ contains all edges presented in the first three stages, and $ G $ is an independent binomial random graph.
When considering $ p' = o(1) $, the effect of subtracting $ G' $ on the probability of certain structures appearing should be negligible.
One way to formalise this idea is to prove that a.a.s.\ $ G' $ is ``sparse'' in some appropriate sense, and separately prove that, given an arbitrary ``sparse'' graph $ H $, the required structure appears a.a.s.\ in $ G \setminus H $.
We now define the particular characterisation of ``sparseness'' which we need in our application.

\begin{definition}\label{def:sparse-partitionable}
    Let $ n, m, r \in \mathbb{N} $ and $ \lambda > 0 $, let $ R \subseteq [n] $, and let $H$ be a graph on vertex set $[n]$.
    We say that a set $ S \subseteq [n] $ is \defi{$ R $-independent} in $H$ if $ H[S \cup R] $ contains no edges incident to any vertex in~$ S $.
    We say that a set $ V \subseteq [n] $ is \defi{$ (R, r, m, \lambda) $-sparse} in $ H $ if every subset $ U \subseteq V $ of size $ |U| \ge m $ contains at least $ \lambda |U|^r $ subsets $ S \subseteq U $ of size $ |S| = r $ which are $ R $-independent in $H$.
    Given integers $ k, \eta \in\mathbb{N} $ and a $(k, \eta)$-endsequence family $ \mathcal{L} = (A_i, B_i)_{i \in [\eta]} $, we say that $ \mathcal{L} $ is \defi{$ (r, m, \lambda) $-linkable over $ V $ despite~$ H $} if~$ V $ is $ (A_i \cup B_i, r, m, \lambda) $-sparse in $ H $ for every $ i \in[\eta] $.

    Now let $ k, r, \nu, \xi, \zeta \in \mathbb{N} $ and $ \chi, \lambda > 0 $.
    We say that $ H $ is \defi{$ (k, r, \nu, \xi, \zeta, \chi, \lambda) $-sparse-partitionable} if the following holds.
    Let $ X_1, \ldots, X_{\nu} $ and $ Y_1, \ldots, Y_{\xi} $ be any collection of pairwise disjoint subsets of~$ [n] $, with $ |X_i| = k $ and $ |Y_a| = \zeta $ for every $ i \in [\nu] $ and $ a \in [\xi] $.
    Then there exists an equipartition $ [\nu] = J_1 \cupdot \cdots \cupdot J_{\xi} $ such that, for every $ a \in [\xi] $ and $ i \in J_a $, the set $ Y_a $ is $ (X_i, r, \chi \zeta, \lambda) $-sparse in $ H $.
\end{definition}

To motivate the definition of sparse-partitionable, suppose that Builder has a $ (k/2, \nu) $-endsequence family $ (A_i, B_i)_{i \in [\nu]} $, for which she seeks to find linkages of length $ r $, while only buying edges contained within the sets $ \{ Y_a \}_{a \in [\xi]} $.
In general, if $ |Y_a| = \zeta $ is small (say, $ o(pn) $), then we cannot guarantee that~$ Y_a $ is $ (X_i, r, \chi \zeta, \lambda) $-sparse in $ H $ for arbitrary $ i \in [\nu] $ and $ a \in [\xi] $, because it could be the case that $ Y_a \subseteq N_H(X_i) $.
However, if $ |\bigcup_{a \in [\xi]} Y_a| = \xi \zeta $ is linear and $ p = o(1) $, then given $ i \in [\nu] $ this must hold for most $ a \in [\xi] $; furthermore, we can match these~$ i $ and~$ a $ to give the desired equipartition.
We can later use this matching to find the desired linkages.

\begin{remark}\label{rmk:sparse-partitionable}
    The property of being sparse-partitionable for any given parameters is decreasing, in the sense that it is inherited by subgraphs.
    Furthermore, our definition of sparseness has the following monotonicity property: if $ V_1 $ is $ (R_1, r, m_1, \lambda_1) $-sparse in $ H_1 $ and $ H_2 \subseteq H_1$, $ V_2 \subseteq V_1 $, $ R_2 \subseteq R_1 $, $ m_2 \ge m_1 $, and $ \lambda_2 \le \lambda_1 $, then $ V_2 $ is $ (R_2, r, m_2, \lambda_2) $-sparse in $ H_2 $.
    Moreover, if $ H $ is empty, then trivially~$ V $ is $ (R, r, m, \lambda) $-sparse for any $ V, R \subseteq [n] $, $ r, m \in \mathbb{N} $ with $ r \le m $, and $ \lambda \ll 1 / r $.
\end{remark}

We may now proceed to state the lemmas which we will use in the proof of \cref{thm:upperboundPowers} to show that the desired linkages exist.
Specifically, we use a version of the connecting lemma of \citet[Lemma~3.3]{NS19}.
Our version allows us to remove a sparse graph~$ H $ from the random graph~$ G $ and still derive the desired conclusion.
However, our version is simplified substantially by considering only the particular case which we need.

\begin{lemma} \label{lem:linkage}
	Let $ 1/n \ll 1/r \ll \delta, 1/k \leq 1/2 $.
    Let $ H $ be an arbitrary graph on $ [n] $ and $ V \subseteq [n] $ be a subset of size $|V| \ge n / \log{n} $.
    If $ G \sim G(n, p) $ for $ p \ge n^{-1/k + \delta} $, then with probability at least $1 - n^{-\log{n}}$ the graph $ G \setminus H $ satisfies the following property.
    
    For every $\eta \in \mathbb{N}$ such that $ 4 r \eta \le |V|$ and every $(k, \eta)$-endsequence family $ \mathcal{L} $ in $ [n] \setminus V $ which is $ (r, n / \log^4{n}, 1 / \log{n}) $-linkable over $ V $ despite~$ H $, the graph $G \setminus H $ contains an $ \mathcal{L} $-linkage of length~$ r $ in~$ V $.
\end{lemma}

We omit the proof of \cref{lem:linkage}, since it is identical to that given in \cite{NS19}, restricting to the graph case and a particular range of $ |V| $ and $ p $, and taking a linkage $ L $ (as defined at the start of this section) in place of their general graph $ F $.
The only substantial change is to replace their Claim~3.4 with the following analogous statement to account for the removal of the sparse graph $ H $.

\begin{claim} \label{claim:linkage_single}
	Let $ 1/n \ll 1/r \ll \delta, 1/k \leq 1/2 $.
    Let $ H $ be an arbitrary graph on $ [n] $ and $V \subseteq [n]$ be a subset of size $|V| \ge n / \log^3{n}$.
    If $ G \sim G(n, p) $ for $ p \ge n^{-1/k + \delta} $, then with probability at least $1 - n^{-\log^2{n}}$ the graph $ G \setminus H $ satisfies the following property.
    
    For every $\eta \in \mathbb{N} $ such that $ 2 r \eta \le |V| $, every subset $U \subseteq V$ of size $|U| \le r \eta $, and every $(k, \eta)$-endsequence family $ \mathcal{L} $ in $ [n] \setminus V $ which is $ (r, n / \log^4{n}, 1 / \log{n}) $-linkable over $ V $ despite $ H $, there is some $i \in [\eta]$ such that $ G \setminus H $ contains an $ (A_i, B_i) $-linkage of length $ r $ in $ V \setminus U $.
\end{claim}

This follows via Janson's inequality very similarly to \cite[Claim~3.4]{NS19} (in fact, it is substantially simpler in our special case), but using the sparseness of $ H $ in the estimate for $ \mu $.
We include the proof in \cref{appen:linkage} for completeness.

We close this section by showing that binomial random graphs are indeed sparse-partitionable, for an appropriate range of the parameters.

\begin{lemma} \label{lem:upperBoundPowers_sparse}
    Given $ r \in \mathbb{N} $ with $ r \ge 2 $, there exists $ \lambda > 0 $ such that the following holds for all $ k, h \in \mathbb{N} $.
    Let $ \nu, \xi, \zeta \in \mathbb{N} $ and $ \chi, p > 0 $ all be functions of $ n $ such that
    \begin{enumerate}[$(\mathrm{S}\arabic*)$]
        \item\label{lem:upperBoundPowers_sparse_item1} $ p = \omega(\log{n} / n) $ and $ p = o(n^{-2/h}) $,
        \item\label{lem:upperBoundPowers_sparse_item2} $ \chi \xi \zeta = \omega(n) $,
        \item\label{lem:upperBoundPowers_sparse_item3} $ \chi p = \omega(\nu^{-1/h}\log n) $,
        \item\label{lem:upperBoundPowers_sparse_item4} $ \zeta \ge \nu^{1/h} $, and
        \item\label{lem:upperBoundPowers_sparse_item5} $\nu\geq\xi$.
    \end{enumerate}
    Then a.a.s.\ $ G(n, p) $ is $ (k, r, \nu, \xi, \zeta, 2 \chi p, \lambda) $-sparse-partitionable.
\end{lemma}

The proof of \cref{lem:upperBoundPowers_sparse} makes use of the following version of the well-known K\H{o}v\'ari--S\'os--Tur\'an theorem~\cite{KST54}.

\begin{lemma}[\citet{KST54}] \label{lem:kst_thm}
    For any $ t \in \mathbb{N} $, there exists $ C > 0 $ such that, for any positive integers $ m\ge n $, every subgraph $ G $ of the complete bipartite graph $ K_{m, n} $ with at least $ C (m^{1 - 1/t} n + m) $ edges contains a copy of the complete bipartite graph $ K_{t, t} $ as a subgraph.
\end{lemma}

\begin{proof}[Proof of \cref{lem:upperBoundPowers_sparse}]
    Let $ \lambda \ll \delta \ll 1/r $, and let $ G \sim G(n,p) $.
    We claim that a.a.s.\ $ G $ satisfies the following three properties:
    \begin{enumerate}[label={\ensuremath{(\mathrm{R}\arabic*)}}]
        \item\label{lem:upperBoundPowers_sparse_item1bis} $ |N_G(v)| \le 2pn $ for all $ v \in [n] $,
        \item\label{lem:upperBoundPowers_sparse_item2bis} $ G $ contains no copy of $ K_{h,h} $, and
        \item\label{lem:upperBoundPowers_sparse_item3bis} $ [n] $ is $ (\varnothing, r, \chi p \zeta, \delta) $-sparse in $ G $.
    \end{enumerate}
    Indeed, \ref{lem:upperBoundPowers_sparse_item1bis} follows from a straightforward application of Chernoff's bound (\cref{lem:Chernoff}) and a union bound over all $ v \in [n] $ by \ref{lem:upperBoundPowers_sparse_item1}.
    Moreover, again by \ref{lem:upperBoundPowers_sparse_item1}, the expected number of copies of~$ K_{h, h} $ in~$ G $ is at most $ o(n^{2h - 2h^2 / h}) = o(1) $, so \ref{lem:upperBoundPowers_sparse_item2bis} follows by Markov's inequality.
    Lastly, we focus on~\ref{lem:upperBoundPowers_sparse_item3bis}.
    Fix any set $ V \subseteq [n] $ of size $ |V| \ge \chi p \zeta $.
    We aim to apply \cref{lem:janson} to bound the number of independent sets of size $ r $ contained in $ V $, so recall the definitions of $ \mu $ and $ \Delta $.\COMMENT{Note that $\mu$ and $\Delta$ are actually defined on the complement of $G$.}
    For $ S \subseteq V $ of size~$ r $, the probability that $ S $ is independent in $G$ is $ (1 - p)^{\binom{r}{2}} = 1 - o(1) $, so the expected number~$ \mu $ of such independent $ S $ satisfies $ \mu \ge 2 \delta |V|^r $.
    Observe also that $ \Delta = O(|V|^{2r - 2}) $, since the number of choices for two sets $ S_1, S_2 \subseteq V $ of size $r$ with at least two vertices in their intersection is $ O(|V|^{2r - 2}) $.
    Hence, by \cref{lem:janson}, with probability at least $ 1 - \nume^{-\Omega(|V|^2)} $, there are at least $ \delta |V|^{r} $ many independent sets $ S \subseteq V $ of size $r$.\COMMENT{Let $X$ denote the number of independent sets of size $r$ in $V$.
    We have that
    \[\mathbb{P}[X<\delta |V|^{r}]\leq\mathbb{P}[X<\mu-\delta|V|^r]\leq\nume^{-\mu^2/2(\mu+2\Delta)}\leq\nume^{-\Theta(|V|^{2r}/|V|^{2r-2})}=\nume^{-\Theta(|V|^{2})},\]
    where the third inequality follows e.g.\ from the fact that $\mu<|V|^r=O(|V|^{2r-2})$ since $r\geq2$.}
    By taking a union bound over all possible choices for~$ V $ and using \ref{lem:upperBoundPowers_sparse_item3} and \ref{lem:upperBoundPowers_sparse_item4}, the probability that every set $V\subseteq [n]$ of size $ |V| \geq \chi p \zeta $ contains at least $\delta|V|^r$ independent sets of size~$r$ is at least $ 1 - \sum_{i=\chi p \zeta}^n\inbinom{n}{i}\nume^{-\Omega(i^2)} = 1-n^{-\omega(1)} $,\COMMENT{For sufficiently large $n$,
    \[1 - \sum_{i=\chi p \zeta}^n\binom{n}{i}\nume^{-\Omega(i^2)} \geq 1 - \sum_{i=\chi p \zeta}^n n^i\nume^{-\Omega(i^2)} = 1 - \sum_{i=\chi p \zeta}^n \nume^{i\log n-\Omega(i^2)} = 1 - \sum_{i=\chi p \zeta}^n \nume^{-\Omega(i^2)} \geq 1 - n \nume^{-\Omega(\chi^2 p^2 \zeta^2)} = 1-n^{-\omega(1)}.\]} so \ref{lem:upperBoundPowers_sparse_item3bis} holds.
    
    Henceforth, we assume that \ref{lem:upperBoundPowers_sparse_item1bis}--\ref{lem:upperBoundPowers_sparse_item3bis} hold.
    We aim to show that any graph $ G $ satisfying these conditions is $ (k, r, \nu, \xi, \zeta, 2 \chi p, \lambda) $-sparse-partitionable.
    Let $ X_1, \ldots, X_\nu $ and $ Y_1, \ldots, Y_\xi $ be as in \cref{def:sparse-partitionable}.
    We first find an equipartition $ [\nu] = J_1 \cupdot \cdots \cupdot J_{\xi} $ such that, for every $ a \in [\xi] $ and $ i \in J_a $, we have $ |N_{G}(X_i) \cap Y_a| \le \chi p \zeta $.
    To do so, consider first an arbitrary equipartition $ [\nu] = J_1' \cupdot \cdots \cupdot J_{\xi}' $ (say, into intervals) and, for each $i\in[\nu]$, let $a(i)$ denote the (unique) $a\in[\xi]$ such that $i\in J_a'$.
    Consider now an auxiliary bipartite graph $ H $ with parts $ A \coloneqq \{ X_i: i \in [\nu] \} $ and $ B \coloneqq [\nu] $, where for each $ X \in A $ and $ i \in B $ we have $ Xi \in E(H) $ if and only if $ |N_{G}(X) \cap Y_{a(i)}| \le \chi p \zeta $.
    Note that any perfect matching~$ M $ in~$ H $ corresponds exactly to a desired equipartition $ [\nu] = J_1 \cupdot \cdots \cupdot J_{\xi} $ in which, for every $ a \in [\xi] $ and $ i \in J_a $, we have $ |N_{G}(X_i) \cap Y_a| \le \chi p \zeta $ (indeed, this equipartition can be obtained by letting $J_a\coloneqq\{i\in[\nu]:X_i\in N_M(J_{a}')\}$ for each $a\in[\xi]$).
    If all degrees in $ H $ are at least $ \nu / 2 $, the existence of one such perfect matching will follow from Hall's theorem, so we need only check that this condition holds.

    Note firstly that, by \ref{lem:upperBoundPowers_sparse_item1bis} and \ref{lem:upperBoundPowers_sparse_item2}, for any $ i \in [\nu] $ we have that $ |N_{G}(X_i)| \le 2kpn = o(\chi p \zeta \xi) $.
    Since the sets $ Y_1, \ldots, Y_\xi $ are pairwise disjoint, it follows that the number of $ a \in [\xi] $ for which $ |N_{G}(X_i) \cap Y_a| \ge \chi p \zeta $ is at most $ o(\xi) $.
    Thus, by \ref{lem:upperBoundPowers_sparse_item5} the minimum degree in $ H $ of any $ X \in A $ is $ (1 - o(1)) \nu $.
    Secondly, we claim that the minimum degree in $ H $ of any $ i \in B $ is at least $ \nu / 2 $.
    To prove this, suppose for a contradiction that there exist $ i \in B $ and a set $ A' \subseteq A $ of size $ |A'| \ge \nu / 2 $ for which $ e_H(A', i) = 0 $.
    Then, for each $ X \in A' $, there exists some $ x \in X $ with $ |N_{G}(x) \cap Y_{a(i)}| \ge \chi p \zeta / k $, so we obtain a set~$ Z $ of $ |Z| \ge \nu / 2 $ vertices $ x \in [n] $ with $ |N_{G}(x) \cap Y_{a(i)}| \ge \chi p \zeta / k $.
    It follows using \ref{lem:upperBoundPowers_sparse_item3} and then \ref{lem:upperBoundPowers_sparse_item4} that
    \[ e_G(Z, Y_{a(i)}) \ge |Z| \cdot \chi p \zeta / k = \omega(|Z| \cdot \nu^{-1/h} |Y_{a(i)}|) = \omega(|Z|^{1 - 1/h} |Y_{a(i)}| + |Z|). \]
    Hence, by \cref{lem:kst_thm}, $ G[Z, Y_{a(i)}] $ contains a copy of the complete bipartite graph $ K_{h, h} $, contradicting \ref{lem:upperBoundPowers_sparse_item2bis}.
    We thus conclude that the minimum degree in $ B $ is at least $ \nu / 2 $, as desired.\COMMENT{In fact, the exact same proof shows that the minimum degree is $(1-o(1))\nu$ again.}

    To finish the proof, fix $ a \in [\xi] $ and $ i \in J_a $.
    Since $ [n] $ is $ (\varnothing, r, \chi p \zeta, \delta) $-sparse in $ G $ by \ref{lem:upperBoundPowers_sparse_item3bis}, this is also true of the set $ Y_a $.
    We aim to show that it is in fact $ (X_i, r, 2\chi p \zeta, \lambda) $-sparse.
    By our choice of equipartition, we have that $ |N_G(X_i) \cap Y_a| \le \chi p \zeta $.
    Hence, given any set $ V \subseteq Y_a $ of size $ |V| \ge 2 \chi p \zeta $, the set $ U \coloneqq V \setminus N_G(X_i) $ has size $ |U| \ge |V| / 2 \ge \chi p \zeta $.
    Now by the original sparseness property, we obtain at least $ \delta |U|^r \ge \lambda |V|^r $ independent sets $ S \subseteq U $ of size $ |S| = r $ in the graph $ G $.
    By definition, such sets are also $ X_i $-independent and contained in $ V $, so this implies that $ Y_a $ is $ (X_i, r, 2\chi p \zeta, \lambda) $-sparse in~$ G $, as required.
\end{proof}

We can now proceed to exhibit our strategy and prove that it is successful.

\subsection{Proof of \texorpdfstring{\cref{thm:upperboundPowers}}{Theorem 6.1}}

If $t\leq n^{2-1/k+\eps/k}$, we may take $b=t$ and the statement holds trivially by simply purchasing every edge given by the random graph process, since in this case a.a.s.~$G_t$ contains the $k$-th power of a Hamilton cycle~\cite{riordan00,KO12,KNP21}.
Moreover, if for every $\eps>0$ we can prove the statement for $t\leq n^{2-\eps/k}$, then it follows in full generality for $t=O(n^2)$ by applying the result with $\eps/k$ playing the role of $\eps$ and considering $t=n^{2-\eps/k}$.\COMMENT{This is using the monotonicity observation made in the introduction (and assuming $\eps<1$, say). 
Indeed, by applying the (assumed) result with $\eps/k$ playing the role of $\eps$, we conclude that there exists a successful $(t,b)$-strategy for the $k$-th power of a Hamilton cycle for all $n^{2-\eps/k^2}\geq t=\omega(n^{2-1/k})$ with $b\leq n^{2k-1+\eps/k}/t^{k-1}$.
In particular, for $t=n^{2-\eps/k}\leq n^{2-\eps/k^2}$, we have a $(t,b)$-strategy with $b\leq n^{1+\eps}$.
By the monotonicity, for all $t\geq n^{2-\eps/k}$ there exists a successful $(t,b)$-strategy for the $k$-th power of a Hamilton cycle with $b\leq n^{1+\eps}$.
Now note that $n^{1+\eps}\leq n^{2k-1+\eps}/t^{k-1}$ for all $t\leq n^2$; thus, the successful strategy above covers all the missing range.}
Thus, we may assume that $n^{2-1/k+\eps/k} \le t \le n^{2-\eps/k} $ and prove the statement for $b\coloneqq n^{2k-1+\eps}/t^{k-1}$.

We split the strategy into four stages.
For each $ i \in [4] $, let $t_i\coloneqq t/4$; during the $i$-th stage we run the random graph process for $t_i$ steps.
We shall denote by $\hat{G}_i$ the (random) graph consisting of all edges presented during stage $i$, and by $\hat{B}_i$ the (random) graph consisting of all edges purchased by Builder.
In each stage, Builder aims to obtain a certain structure; if in some stage this does not succeed, then we terminate and say that the strategy has failed; we do the same if at any point Builder attempts to buy more than $ b $ edges.
For simplicity, we will show that a.a.s.\ each $ \hat{B}_i $ contains an appropriate structure for stage $ i $ to succeed, and separately that a.a.s.\ the budget is not exceeded; this is clearly sufficient.
\vspace{0.5em}

\noindent
\textbf{Preparation.\ }
We will work with parameters adhering to the following hierarchy:
\begin{align}
    1/n \ll 1/q \ll 1/\ell \ll 1/r \ll \eps' \ll \eps \ll 1/k \leq 1/2.\label{equa:hierarchy1}
\end{align}
We further specify some of these (and other) parameters next.
The parameter $r$ represents the length of each linkage to be found in the second and fourth stages of the strategy.
Apply \cref{lem:absorberdensity} with~$\eps/3$ in place of~$\delta$ to obtain $j=j(k,\eps/3)$ such that\COMMENT{This is already using that $1/\ell\ll\eps,1/k$.}
\begin{equation}\label{eq:densabsstr}
    d^*(\mathcal{A}_{j,\ell,k})\leq k+\eps/3.
\end{equation}
Set $ s\coloneqq j(2\ell+4)+\ell $, noting that $ s + 1 = v(\mathcal{A}_{j,\ell,k}) $ is the order of each $(j,\ell,k)$-absorber to be found in the first stage.
For simplicity, from here on we will refer to $(j,\ell,k)$-absorbers simply as \emph{absorbers}.
The parameter $q$ is chosen to be a prime and such that $ 1/q \ll 1/s $, and represents the order of the $k$-th powers of paths with which to cover all remaining vertices in the third stage.
Choose $\eta\in\{\lfloor n/(3(s+1))\rfloor-q+1,\ldots,\lfloor n/(3(s+1))\rfloor\}$ such that $n-\eta(s+1)-(\eta-1)r = \nu q $ for some $ \nu \in \mathbb{N} $ (note that this is always possible, since $q$ is prime and greater than $s+1+r$).
Here, $\eta$ represents the number of absorbers to be bought in the first stage, and~$\nu$ represents the number of powers of paths~$P_q^k$ to be bought in the third stage.

Fix an arbitrary set $ U_1 \subseteq [n] $ of size $ |U_1| = \eta(s+1) $ (which is roughly $n / 3$), in which we will buy absorbers in the first stage.
Let~$ \pi $ and~$ \delta_1 $ be the values of~$ \theta $ and~$ \delta' $, respectively, given by \cref{lem:factorGnp} with input $(|U_1|,\mathcal{A}_{j,\ell,k},\eps',t_1/M)$ in place of $(n,F,\delta,p)$, where the required bounds on $ p $ follow from \eqref{equa:hierarchy1}.\COMMENT{Note that, in order to be able to apply \cref{lem:factorGnp}, we must ensure that the $p$ we use is within the prescribed interval.
Using the assumed bounds on $t$, we have that $t_1/M\gtrsim n^{-(1-\eps)/k}/2$, and in order to apply the lemma we must have $t_1/M\geq|U_1|^{-1/d^*(\mathcal{A}_{j,\ell,k})+\eps'}\sim (n/3)^{-1/d^*(\mathcal{A}_{j,\ell,k})+\eps'}$.
Since $2$ and $3^{-1/d^*(\mathcal{A}_{j,\ell,k})+\eps'}$ are constants, the desired inequality will hold (for sufficiently large $n$) if 
\[-\frac{1-\eps}{k}>-\frac{1}{d^*(\mathcal{A}_{j,\ell,k})}+\eps'\iff\frac{1}{d^*(\mathcal{A}_{j,\ell,k})}>\frac{1-\eps+k\eps'}{k}\iff d^*(\mathcal{A}_{j,\ell,k})<\frac{k}{1-\eps+k\eps'},\]
which holds since, by \eqref{eq:densabsstr}, $d^*(\mathcal{A}_{j,\ell,k})\leq k+\eps/3<k+\eps$ and
\[k+\eps<\frac{k}{1-\eps+k\eps'}\iff -(k-1+\eps-k\eps')\eps+k^2\eps'<0,\]
which holds since $k\geq2$ by the choice that $\eps'\ll\eps\ll1/k$ in \eqref{equa:hierarchy1}.\\
For the upper bound, we have that $t_1/M\lesssim n^{-\eps/k}/2$ and must have $t_1/M\leq |U_1|^{-\eps'}\sim(n/3)^{-\eps'}$, which holds for $n$ sufficiently large if $-\eps/k<-\eps'$, i.e., $\eps'<\eps/k$, which holds by choice from \eqref{equa:hierarchy1}.}
Choose an arbitrary $\mathcal{A}_{j,\ell,k}$-equipartition $U_1=V_1\cupdot\cdots\cupdot V_{\pi}$.
Let $U_2\subseteq[n]\setminus U_1$ be an arbitrary set of size $|U_2|=n/3$, in which we will buy linkages in the second stage.
Choose an arbitrary equipartition $U_2=W_1\cupdot\cdots\cupdot W_{\xi}$, for
\begin{equation} \label{eqn:powers_xidefn}
    \xi=\xi(n)\coloneqq\left\lfloor\left(\frac{t}{n^2}\right)^{\frac{k}{1-\eps/2}}n\right\rfloor.
\end{equation}
Similarly, let $I_1\cupdot\cdots\cupdot I_{\xi}$ be an arbitrary equipartition of the set $[\eta-1]$ of indices of the endsequences between which we aim to buy linkages.
Finally, let $ \sigma $ and $ \delta_3 $ be the values of~$ \theta $ and~$ \delta' $ given by \cref{lem:factorGnp} with input $(\nu q,P_{q}^k,\eps',t_3/M)$ in place of $(n,F,\delta,p)$ (again using \eqref{equa:hierarchy1}).\COMMENT{The verification that the probability is within range is like above, the only difference is in some constants that do not change the validity of the asymptotic arguments.
Indeed, note that, since $1/\ell\ll1/r$, we have that $\nu q=n-\eta (s+1)-(\eta-1)r\geq n/2$.}
\vspace{0.5em}

\noindent
\textbf{Strategy.\ }
We now define our strategy in four separate stages, as discussed.
In order to define each stage, we assume that all previous stages have succeeded and that we have access to their output; we will prove later that the overall probability of failure is $ o(1) $.
\begin{enumerate}[label=Stage~\Roman*,leftmargin=5.5em,labelsep=1.5em]
    \item \label{item:upperboundPowers_stage1} Mimic the strategy of \cref{thm:factoruppergeneral} to obtain an $\mathcal{A}_{j,\ell,k}$-factor on $ U_1 $. 
        That is, for time~$t_1$ buy every edge from $\bigcup_{a\in[\pi]}\inbinom{V_a}{2}$ which is presented.
        Assume that the strategy succeeds, so there exists an $\mathcal{A}_{j,\ell,k}$-factor on $ U_1 $; Builder now chooses one in a particular (deterministic) way, which will be specified later in \cref{claim:upperboundPowers_stages1-3}, and fixes an arbitrary labelling $\mathcal{A}_1,\ldots,\mathcal{A}_\eta$ of its vertex-disjoint copies of $\mathcal{A}_{j,\ell,k}$.
    \item \label{item:upperboundPowers_stage2} For each $i\in[\eta]$, let $S_{i-1}$ and $T_i$ denote the initial and final endsequences of $\mathcal{A}_i$, respectively.
        For each $a\in[\xi]$, let $W_a' \coloneqq W_a\cup\bigcup_{i\in I_a}(S_i\cup T_i)$ and $\mathcal{L}_a\coloneqq\{(S_i,T_i):i\in I_a\}$.
        Now, for time $ t_2 $, purchase every edge presented which is contained in one of the sets~$W_a'$.
        Assume that, for each $ a \in [\xi] $, the graph $ \hat{B}_2[W_a'] $ contains an $ \mathcal{L}_a $-linkage of length~$r$ in~$W_a$.
        Builder now chooses one such $ \mathcal{L}_a $-linkage $ (L_i)_{i\in I_a} $; again, the way Builder makes this (deterministic) choice is specified in \cref{claim:upperboundPowers_stages1-3}.
        Then the union of the linkages $ (L_i)_{i \in[\eta - 1]} $ and the $k$-th powers of the spines of the absorbers $ (\mathcal{A}_i)_{i\in[\eta]} $ forms the $k$-th power of one long path, which we denote by $Q_0$.
    \item \label{item:upperboundPowers_stage3} Let $ U_2' \coloneqq \bigcup_{i \in [\eta - 1]} V(L_i) $, so $ |U_2'\setminus U_1| = (\eta - 1) r $.\COMMENT{Note that $U_1$ and $U_2'$ are NOT disjoint, as the endsequences of the linking structures are contained in $U_1$.}
        Let $U_3 \coloneqq[n]\setminus(U_1\cup U_2')$, and recall that $|U_3|=n-\eta(s+1)-(\eta-1)r = \nu q $, where $\nu\in\mathbb{N}$.
        Choose an arbitrary $P_q^k$-equipartition $U_3=X_1\cupdot\cdots\cupdot X_{\sigma}$.
        Now mimic the strategy of \cref{thm:factoruppergeneral} to obtain a $P_{q}^k$-factor on $ U_3 $; that is, for time $t_3$, purchase exactly the presented edges which are contained within $ X_a $ for some $ a \in [\sigma] $.
        Assume that this succeeds, in the sense that Builder may (arbitrarily) choose pairwise vertex-disjoint paths $Q_1,\ldots,Q_{\nu}$ of order $q$ whose $ k$-th powers are contained in $ \hat{B}_3 $.
    \item \label{item:upperboundPowers_stage4} 
        For each $i\in\{ 0 \} \cup [\nu]$, let $S_{i-1}'$ and $T_i'$ denote the initial and final endsequences of $Q_i^k$, respectively, with indices taken modulo $ \nu + 1 $.
        Let $A\subseteq U_1$ denote the set consisting of the absorption vertices of all absorbers $ \mathcal{A}_i $ ($i \in [\eta]$).
        Let $A''\subseteq A$ be an arbitrary set of size $|A''| < \xi $ with $|A''|\equiv|A|\pmod{\xi}$ and let $A'\coloneqq A\setminus A''$.
        Choose an arbitrary equipartition of $A'$ into $\xi$ sets, $A'=Y_1\cupdot\cdots\cupdot Y_{\xi}$.
        Then choose a specific equipartition $\{0\}\cup[\nu]=J_1\cupdot\cdots\cupdot J_{\xi}$, to be specified later (see \cref{claim:upperboundPowers_stage4}).
        For each $a\in[\xi]$, let $Y_a' \coloneqq Y_a\cup\bigcup_{i\in J_a}(S_i'\cup T_i')$ and $\mathcal{L}_a'\coloneqq\{(S_i',T_i'):i\in J_a\}$.
        Now, for time $ t_4 $, purchase every presented edge which is contained in one of the sets~$Y_a'$.
        Assume that, for each $ a \in [\xi] $, the graph $ \hat{B}_4[Y_a'] $ contains an $ \mathcal{L}_a' $-linkage $ (L_i')_{i\in J_a} $ of length~$r$ in $Y_a$.
\end{enumerate}
\vspace{0.5em}

\noindent
\textbf{Absorption.\ }
At this point, assuming for now that all stages of the strategy succeeded, the union of all~$Q_i^k$ and~$L_i'$ for $i\in\{0\}\cup[\nu]$ results in the $k$-th power $\cC^k$ of a cycle $\cC$ which contains all vertices in~$[n]$ except some leftover vertices in $A$ (note that $ A' $ is large enough to contain the required number of linkages in \ref{item:upperboundPowers_stage4}, by \eqref{equa:hierarchy1}).  
Moreover, $\cC$ contains the spine of every absorber $\mathcal{A}_i$ ($i\in[\eta]$).
Now, since all the $\mathcal{A}_i$ are vertex-disjoint, for each absorption vertex $v\in A\setminus V(\cC)$, we may replace the $k$-th power of the spine of its corresponding absorber by the $k$-th power of its augmented path, which now contains $v$.
This results in the $k$-th power of a Hamilton cycle in $B_t$, as desired.
\vspace{0.5em}

\noindent
\textbf{Proof of the success of the strategy.\ }
It now suffices to show that a.a.s.\ Builder succeeds in purchasing the specified structures in each of the four stages, and that the total number of edges purchased throughout does not exceed the budget.
Since this will be shown via coupling arguments, we prove first that the desired structures appear in corresponding binomial random graphs.

\begin{claim} \label{claim:upperboundPowers_stage2}
    Let $ p_2 = (1 - o(1)) t_2 / M $ and $ H_2 \sim G(n, p_2) $.
    Then there exists $\delta_2 = o(1) $ such that the following holds.
    Let $ \mybar{\mathcal{L}}=\{(\mybar{R}_i, \mybar{S}_i)\}_{i \in [\eta - 1]} $ be a (fixed, arbitrary) $(k,\eta-1)$-endsequence family in $ U_1 $.
    For each $ a \in [\xi] $, define $ \mybar{W}_a \coloneqq W_a \cup \bigcup_{i\in I_a} (\mybar{R}_i\cup \mybar{S}_i) $ and $ \mybar{\mathcal{L}}_a\coloneqq\{(\mybar{R}_i, \mybar{S}_i):i\in I_a\} $.
    With probability at least $1-\delta_2$, for every $ a \in [\xi] $, the graph $ H_2[\mybar{W}_a] $ contains an $ \mybar{\mathcal{L}}_a $-linkage of length $r$ in $W_a$.
\end{claim}

\begin{claimproof}
    We aim to apply \cref{lem:linkage} to $H_2[\mybar{W}_a]$ (taking $ H $ to be empty; recall that this satisfies the desired condition by \cref{def:sparse-partitionable} and \cref{rmk:sparse-partitionable}) for each $a\in[\xi]$, with $(\eps',\mybar{W}_a,W_a,|I_a|)$ in place of $(\delta,[n],V,\eta)$.
    By \eqref{eqn:powers_xidefn} and \eqref{equa:hierarchy1} and since $ |\mybar{W}_a| = n / 3 \xi + 2k $, it is a simple task to check that \mbox{$p_2=\Omega(|\mybar{W}_a|^{-(1-\eps/2)/k})\geq|\mybar{W}_a|^{-1/k+\eps'}$} for all $a\in[\xi]$.\COMMENT{Indeed,
    \[|\mybar{W}_a|^{-(1-\eps/2)/k}=\Theta\left(\left(\frac{n}{\xi}\right)^{-(1-\eps/2)/k}\right)=\Theta\left(\frac{t}{n^2}\right)\]
    by the definition of $\xi$.}
    Moreover, again by \eqref{equa:hierarchy1}, observe that $ 4 r |I_a| = 4 r \eta / \xi < n / 3 \xi = |W_a| $.
    Therefore, by \cref{lem:linkage} and a union bound over all $a\in[\xi]$, we conclude that, with probability at least $1-n^{-\sqrt{\log{n}}} $,\COMMENT{Since $ n/\xi $ is a fixed polynomial in terms of $ \eps, k $, the failure probabilities we get are $ n^{-\Omega(\log{n})} $, and we union bound over $ n^{O(1)} $ many elements.} for every $a\in[\xi]$ the graph $ H_2[\mybar{W}_a] $ contains an $\mybar{\mathcal{L}}_a$-linkage of length~$r$ in~$W_a$.
\end{claimproof}

\begin{claim} \label{claim:upperboundPowers_stage3}
    Let $ p_2' = (1 + o(1)) t_2 / M$, $ p_3 = (1 - o(1)) t_3 / M $, and $ H_2' \sim G(n, p_2'), H_3 \sim G(n, p_3) $ be sampled independently.
    Let\/ $ \mybar{U}_3 \subseteq [n]\setminus U_1 $ have size $ |\mybar{U}_3| = \nu q $, and let\/ $ \mybar{X}_1\cupdot \cdots\cupdot \mybar{X}_{\sigma} $ be a $P_q^k$-equipartition of\/~$ \mybar{U}_3 $.
    Then, with probability at least $1-\delta_3$ (recall that $\delta_3$ was defined already in the preparation phase), we have that $ H_3 \setminus H_2' $ contains a $ P^k_{q} $-factor on $ \mybar{U}_3 $ consisting only of edges in $ \bigcup_{a \in [\sigma]} \inbinom{\overline{X}_a}{2} $.
\end{claim}

\begin{claimproof}
    Note that $ H_3 \setminus H_2' \sim G(n, p') $ for some $ p' = (1 - o(1)) t_3 / M $.\COMMENT{This holds since $t=o(n^2)$ by assumption.}
    Furthermore, $d^*(P^k_{q})< k$ and therefore, using \eqref{equa:hierarchy1}, 
    \[|\mybar{U}_3|^{-1/d^*(P_q^k)+\eps'}\leq\frac{n^{-1/k+\eps/k}}{4}\leq p'\leq n^{-\eps/k}\leq |\mybar{U}_3|^{-\eps'}.\]
    Hence, by \cref{lem:factorGnp}~\ref{lem:factorGnpitem1} with $(\mybar{U}_3,P_q^k,\eps',t_3/M)$ playing the roles of $([n],F,\delta,p)$, we conclude that with probability at least $1-\delta_3$ the binomial random graph $ H_3 \setminus H_2' $ contains a $ P^k_{q} $-factor on~$ \mybar{U}_3 $ consisting only of edges in $ \bigcup_{a \in [\sigma]} \inbinom{\overline{X}_a}{2} $.
\end{claimproof}

We can now use these claims together with \cref{lem:two_stage_real} to show that a.a.s.\ Builder will successfully construct the desired structures in the first three stages of the strategy.

\begin{claim} \label{claim:upperboundPowers_stages1-3}
    Builder can choose $ (\mathcal{A}_i)_{i \in [\eta]} $ in \ref{item:upperboundPowers_stage1} and $ (L_i)_{i \in[\eta - 1]} $ in \ref{item:upperboundPowers_stage2} in such a way that a.a.s.\ the first three stages succeed.
\end{claim}

\begin{claimproof}
    We aim to apply \cref{lem:two_stage_real}.
    Let $ \mathcal{E}^1 = \{ E_1, \ldots, E_m \} $ be the set of all possible $\mathcal{A}_{j,\ell,k}$-factors on $ U_1 $ consisting of edges belonging to $\bigcup_{a\in[\pi]}\inbinom{V_a}{2}$.
    For each $ i \in [m] $, fix an arbitrary labelling of the copies of $\mathcal{A}_{j,\ell,k}$ which constitute~$E_i$ (which corresponds to the labelling considered at the end of \ref{item:upperboundPowers_stage1}), let $\mathcal{L}\coloneqq\bigcup_{a\in[\xi]}\mathcal{L}_a$, and let $ \mathcal{E}^2_i = \{ E_{i, 1}, \ldots, E_{i, r_i} \} $ be the collection of all possible $\mathcal{L}$-linkages of length $r$ which consist of edges belonging to $ \bigcup_{a \in [\xi]} \inbinom{W_a'}{2} $.
    Then, for each $i\in[m]$ and each $ h \in [r_i] $, given the set $U_3$ determined by  $ E_{i, h} $ (and after having fixed the arbitrary $P_q^k$-equipartition $U_3=X_1\cupdot\cdots\cupdot X_\sigma$ in \ref{item:upperboundPowers_stage3}), let $ \mathcal{F}_{i, h} $ be the collection of all possible $ P^k_{q} $-factors on $ U_3 $ consisting of edges belonging to $ \bigcup_{a \in [\sigma]} \inbinom{X_a}{2} $.
    Note that $E(F)\cap E(E_{i,h})=\varnothing$ for all $F\in\mathcal{F}_{i,h}$.
    Now let $ K^1 $ and $ K^2 $ be as defined in \eqref{eqn:zak6}.
    We must verify that all the conditions required to apply \cref{lem:two_stage_real} hold.
    
    Note that from \ref{item:upperboundPowers_stage1} we have $E_i\subseteq\inbinom{U_1}{2}$ for every $i\in[m]$, and therefore $ K^1 \subseteq \inbinom{U_1}{2} $.
    On the other hand, note that in \ref{item:upperboundPowers_stage2} each $ W_a $ is disjoint from $ U_1 $, and a linkage does not involve any edges contained within its endsequences, so no potential linkage $E_{i,h}$ involves any edge contained within the set $~U_1 $.
    Since in \ref{item:upperboundPowers_stage3} the potential $P^k_{q} $-factors $ \mathcal{F}_{i, h} $ are contained in $U_3$ and $U_3\cap U_1=\varnothing$, it follows that $ K^2 $ is edge-disjoint from $ K^1 $, so \eqref{eqn:zak6} is satisfied.
    Moreover, note that $e(K^1)=\Theta(\pi(n/\pi)^2)=\Theta(n^2/\pi)$, hence
    \[\frac{M}{e(K^1)}=\Theta(\pi)=\Theta\left(\left(\frac{t_1}{M}\right)^{d^*(\mathcal{A}_{j,\ell,k})}n^{1-\eps'd^*(\mathcal{A}_{j,\ell,k})/2}\right)=o(n),\]
    and thus $t_i=\omega(n) = \omega(M/e(K^1))$ for all $i\in[3]$.
    Similarly,\COMMENT{The reason why here we write $O$ instead of $\Theta$ is that we have not (nor want to) consider the edges coming from \ref{item:upperboundPowers_stage2}, so the calculations are done taking solely \ref{item:upperboundPowers_stage3} into account.}
    \[\frac{M}{e(K^2)}=O(\sigma)=O\left(\left(\frac{t_1}{M}\right)^{d^*(P_q^k)}n^{1-\eps'd^*(P_q^k)/2}\right)=o(n),\]
    so $t_i=\omega(n) = \omega(M/e(K^2))$ for all $i\in[3]$.
    
    By using \eqref{equa:hierarchy1} and \eqref{eq:densabsstr} we have that\COMMENT{For the second inequality it suffices to check that 
    \[-\frac{1}{k+\eps/3}+2\eps'<-\frac{1-\eps}{k}\iff\frac{3k}{3k+\eps}+2\eps'k>1-\eps.\]
    In particular, it suffices to have $3k>(3k+\eps)(1-\eps)=3k-(3k-1)\eps-\eps^2$, which clearly holds.}
    \[|U_1|^{-1/d^*(\mathcal{A}_{j,\ell,k})+\eps'} \leq n^{-1/(k+\eps/3)+2\eps'} \leq \frac{n^{-1/k+\eps/k}}{4} \leq \frac{t_1}{M} \leq n^{-\eps/k} \leq |U_1|^{-\eps'},\]
    and therefore \cref{lem:factorGnp}~\ref{lem:factorGnpitem1} applied with $(U_1,\mathcal{A}_{j,\ell,k},\eps',t_1/M)$ playing the roles of $([n],F,\delta,p)$ ensures that with probability at least $1-\delta_1$, for any $p_1=(1 - o(1)) t_1 / M$, the graph $ G(n, p_1) $ contains an $\mathcal{A}_{j,\ell,k}$-factor on~$ U_1 $ consisting of edges which lie in $\bigcup_{a\in[\pi]}\inbinom{V_a}{2}$; this in particular means that condition~\ref{cond:zak1} in the setup for \cref{lem:two_stage_real} holds.
    The hypotheses \ref{cond:zak2} and \ref{cond:zak3} follow immediately from \cref{claim:upperboundPowers_stage2} and \cref{claim:upperboundPowers_stage3}, respectively.

    We may thus apply \cref{lem:two_stage_real}, which gives deterministic functions $ i \colon \mathcal{G}_1 \to [m], j \colon \mathcal{H} \to \mathbb{N} $, where $ i(\hat{G}_1) $ corresponds to a factor $\mathcal{A}_1,\ldots,\mathcal{A}_\eta$ in \ref{item:upperboundPowers_stage1}, and $ j(\hat{G}_1, \hat{G}_2) $ corresponds to a set of linkages $ L_1, \ldots, L_{\eta - 1} $ in \ref{item:upperboundPowers_stage2}.
    This determines the aforementioned choices of these objects which Builder must make in each stage.
    By the conclusion of \cref{lem:two_stage_real}, with these choices, a.a.s.\ the first three stages all succeed.
\end{claimproof}

For the remainder of the proof, consider instead the coupling provided by \cref{lemma:coupling_stages}: for each \mbox{$i\in[4]$} there exist $p_i=(1-o(1))t_i/M$ and $\mybar{p}_i=o(t_i/M)$ and a coupling of random graphs $(H_i,\hat{G}_i,\mybar{H}_i)_{i\in[4]}$ satisfying properties \ref{lem:couplingproperty1}--\ref{lem:couplingproperty3} of \cref{lemma:coupling_stages}; write $ H_i' \coloneqq H_i \cup \mybar{H}_i $.

\begin{claim} \label{claim:upperboundPowers_stage4}
    Assume that the first three stages succeed, and further that $ H \coloneqq H_1' \cup H_2' \cup H_3' $ is \mbox{$ (2k, r, \nu + 1, \xi, \lfloor\eta / \xi\rfloor, 1 / \log^5{n}, 1/\log^{1/2}{n}) $}-sparse-partitionable.
    Then there exists a choice of equipartition $ \{0\} \cup [\nu] = J_1 \cupdot \cdots \cupdot J_{\xi} $ such that a.a.s.\ \ref{item:upperboundPowers_stage4} also succeeds.
\end{claim}

\begin{claimproof}
    We assume that all information from the first three stages (in particular, $ (\hat{G}_i)_{i \in [3]} $ and $ H $) has already been revealed, that the containments in \cref{lemma:coupling_stages}~\ref{lem:couplingproperty3} occur for $i\in[3]$, and that $ H $ is indeed $ (2k, r, \nu + 1, \xi, \lfloor\eta / \xi\rfloor, 1 / \log^5{n}, 1/\log^{1/2}{n}) $-sparse-partitionable, which implies the same about $ \hat{G} \coloneqq \hat{G}_1 \cup \hat{G}_2 \cup \hat{G}_3 \subseteq H $\COMMENT{The containment holds since we just assumed that \cref{lemma:coupling_stages}~\ref{lem:couplingproperty3} holds for all $i\in[3]$.} by \cref{rmk:sparse-partitionable}.
    Therefore (recall \cref{def:sparse-partitionable}), Builder may choose some equipartition $ \{0\}\cup[\nu]=J_1\cupdot\cdots\cupdot J_{\xi} $ such that, for every $ a \in [\xi] $ and $ i \in J_a $, the set $ Y_a $\COMMENT{Recall from \ref{item:upperboundPowers_stage4} that we ensured all the sets $Y_a$ have size exactly $\lfloor\eta/\xi\rfloor$.} is \mbox{$ (S_i' \cup T_i', r, \eta / (\xi \log^5{n}), 1 / \log^{1/2}{n}) $}-sparse in $ \hat{G} $; in other words, for every $a\in[\xi]$, the endsequence family $ \mathcal{L}_a' $ is $ (r, \eta / (\xi \log^5{n}), 1 / \log^{1/2}{n}) $-linkable over $ Y_a $ despite $ \hat{G} $.
    
    We now aim to apply \cref{lem:linkage} to the graph $(H_4\setminus \hat{G})[Y_a']$ for each $ a \in [\xi] $, with \mbox{$ |Y_a'| = \eta / \xi + 2k $} in place of $ n $, $\eps'$ in place of $\delta$, $ Y_a $ in place of $ V $, and $ |J_a| = (\nu + 1) / \xi $ in place of $ \eta $.
    It is easy to check that $p_4=(1-o(1))t_4/M=\Omega(|Y_a|^{-(1-\eps/2)/k})\geq|Y_a|^{-1/k+\eps'}$ for all $a\in[\xi]$ by \eqref{eqn:powers_xidefn}.
    Note also that \mbox{$ 4 r (\nu + 1) / \xi \le \eta / \xi $} since $ 1 / q \ll 1/s, 1 / r $, and that $ 1/\log^5{n} = o ( 1/\log^4(\eta / \xi) ) $ and \mbox{$ 1/\log(\eta / \xi) = o ( 1 / \log^{1/2}{n} ) $}.
    Therefore, by \cref{rmk:sparse-partitionable}, we may apply \cref{lem:linkage} to each $(H_4\setminus \hat{G})[Y_a']$ to conclude that, with probability at least $ 1 - n^{-\log^{1/2}{n}} $, it contains an $\mathcal{L}_a'$-linkage $(L_i')_{i\in J_a}$ of length~$ r $ in~$ Y_a $.
    The conclusion now follows by a union bound over all $a\in[\xi]$, and since a.a.s.\ $ H_4 \setminus \hat{G} \subseteq \hat{G}_4 $ by \cref{lemma:coupling_stages}~\ref{lem:couplingproperty3} and thus $ (H_4 \setminus \hat{G})[Y_a'] \subseteq \hat{B}_4 $.
\end{claimproof}

We may now combine all the previous claims.
Firstly, by \cref{claim:upperboundPowers_stages1-3}, a.a.s.\ the first three stages succeed.
By applying \cref{lem:upperBoundPowers_sparse} with $ h = \lceil1 / \eps'\rceil $, we conclude that a.a.s.\ the graph $H$ defined in \cref{claim:upperboundPowers_stage4} is $ (2k, r, \nu + 1, \xi, \lfloor\eta / \xi\rfloor, 1 / \log^5{n}, 1/\log^{1/2}{n}) $-sparse-partitionable; the required conditions are easy to check.\COMMENT{Here:
\begin{enumerate}
    \item $ p \le n^{-\eps / k} \ll n^{-2\eps'} \leq n^{-2/h}$, lower bound is clear;
    \item $ \chi \xi \zeta = \eta \chi $ and note $ \eta = \Omega(n) $ and $\chi=1/(2p\log^5n)=\omega(1)$ by the upper bound above;
    \item $ \chi p = 1 / 2 \log^5{n} $ and $ \nu $ is polynomial;
    \item Have $ \eta = \Omega(n) $ and $ \xi \le p^k n \le n^{1-\eps} $ so $ \zeta = \eta / \xi = \Omega(n^{\eps}) $ whereas $ \nu^{1/h} \le n^{\eps'} \ll n^{\eps} $.
\end{enumerate}}
This means that, by \cref{claim:upperboundPowers_stage4}, a.a.s.\ Builder also succeeds in buying the desired structures in the fourth stage.

All that remains to show is that a.a.s.\ the budget is not exceeded.
It follows from \cref{lem:factorGnp}~\ref{lem:factorGnpitem2} (with input $(U_1,\mathcal{A}_{j,\ell,k},\eps',t_1/M)$ in place of $([n],F,\delta,p)$) and by using \cref{lemma:coupling_stages}~\ref{lem:couplingproperty3}, \eqref{equa:hierarchy1} and~\eqref{eq:densabsstr} that a.a.s.\ the number of edges purchased in \ref{item:upperboundPowers_stage1} is
\begin{align*}
    e(\hat{B}_1) \leq e\left(\bigcup_{a\in[\pi]}H_1'[V_a]\right)&\leq \left(\frac{t_1}{M}\right)^{1-d^*(\mathcal{A}_{j,\ell,k})}n^{1+\eps'd^*(\mathcal{A}_{j,\ell,k})/2}\\
    &=O\left(\left(\frac{n^2}{t}\right)^{k+\eps/3-1}n^{1+\eps/5}\right)=o(n^{2k-1+\eps}/t^{k-1})<b/4.
\end{align*}
In a very similar way (but using now that $d^*(P_q^k)<k$), we have that a.a.s.\ $e(\hat{B}_3) < b/4$.\COMMENT{We have that
\[e(\hat{B}_3) \leq e\left(\bigcup_{a\in[\sigma]}H_3'[X_a]\right)\leq \left(\frac{t_3}{M}\right)^{1-d^*(P_q^k)}n^{1+\eps'd^*(P_q^k)/2}=O\left(\left(\frac{n^2}{t}\right)^{k-1}n^{1+\eps/5}\right)=o(n^{2k-1+\eps}/t^{k-1})<b/4.\]}
Now, similarly as \cref{lem:factorGnp}~\ref{lem:factorGnpitem2} is proved through Chernoff's inequality, the same can be used to bound the number of purchased edges in \ref{item:upperboundPowers_stage2} and \ref{item:upperboundPowers_stage4}.
For \ref{item:upperboundPowers_stage2}, by \cref{lemma:coupling_stages}~\ref{lem:couplingproperty3}, Chernoff's inequality (\cref{lem:Chernoff}), and a union bound over all $a\in[\xi]$, we have that a.a.s.\COMMENT{The application of Chernoff's bound should be fine. For the last inequality, we must verify that
\[\frac{t}{\xi}=o(b)=o\left(\frac{n^{2k-1+\eps}}{t^{k-1}}\right)\iff t^k=o\left(n^{2k+\eps}\left(\frac{t}{n^2}\right)^{\frac{k}{1-\eps/2}}\right)\iff\left(\frac{t}{n^2}\right)^{k-\frac{k}{1-\eps/2}}=o(n^\eps).\]
Using that $k-k/(1-\eps/2)=-\eps k/(2(1-\eps/2))<0$ and $t\geq n^{2-1/k}$ by assumption (we could use the stronger lower bound here, but it is not needed), we have that
\[\left(\frac{t}{n^2}\right)^{k-\frac{k}{1-\eps/2}}\leq\left(n^{-1/k}\right)^{-\eps k/(2(1-\eps/2))}=n^{\eps/(2-\eps)}=o(n^\eps),\]
as we wanted (since we assume $\eps<1$).
}
\[e(\hat{B}_2)\leq \xi p_2'\left(\frac{n}{3\xi}\right)^2=\Theta\left(\frac{t}{\xi}\right)=o(b),\]
so a.a.s.\ the budget is not exceeded in \ref{item:upperboundPowers_stage2}.
Similarly, a.a.s.\ $e(\hat{B}_4)=O(t/\xi)=o(b)$.
It follows that a.a.s.\ $e(B_t)<b$, as desired.
\qed

\begin{remark}
    It seems plausible that one could improve the upper bound on the budget in \cref{thm:upperboundPowers} by considering the strategy of \citet{NS19} instead of the one we have used here.
    As this would not lead to optimal bounds for the budget, we have not explored this direction.
\end{remark}


\section{Final remarks and open problems}\label{sect:problems}

As we already discussed in the introduction, our results give a negative answer to the general question of whether any bounded-degree $n$-vertex graph $H$ can be constructed in the budget-constrained random graph process with linear budget when slightly above the hitting time for containing a copy of~$H$ (\cref{question:FKMgeneral}).
In more generality, we believe that the examples we know of so far give evidence towards the opposite behaviour to what \citet{FKM25} proposed: when considering perfect matchings, Hamilton cycles or connectivity, the budget can be improved by a logarithmic factor; for $K_r$-factors, it can be improved by at most some polylogarithmic factor with exponent less than $1$; for $\alpha$-$K_r$-factors and fixed trees or cycles, it cannot be improved by more than a constant factor.
We suggest that this may be a general phenomenon for bounded-degree $H$: when $t$ is of the order of the hitting time for $H$, no successful strategies exist with a budget substantially lower than the hitting time itself.
To address this problem, we propose the following concrete question.
We say that a monotone increasing property $\mathcal{P}$ has \defi{bounded degree} if there exists some $\Delta\in\mathbb{N}$ which does not depend on $n$ such that every edge-minimal graph $G\in\mathcal{P}$ satisfies that $\Delta(G)\leq\Delta$.

\begin{problem}
    Does there exist a bounded-degree monotone property $\mathcal{P}$ for which, if~$t$ is of the order of the hitting time for $\mathcal{P}$, there exists a successful $(t,b)$-strategy for $\mathcal{P}$ with $b=o(t/\log n)$?\COMMENT{Note that, a priori, one should also enforce that a.a.s.\ the hitting time for $\mathcal{P}$ is subquadratic, as otherwise the answer would trivially be `yes' (for instance, the answer to the question is `yes' for the containment of a spanning hypercube, which does not have bounded degree). However, there do not exist any bounded-degree properties whose hitting time is quadratic (this follows, e.g., from the work on the threshold for $\Delta$-universality). In fact, a.a.s.\ the hitting time for any such property is polynomially smaller than $n^2$.}
\end{problem}

We note that the bounded degree condition is necessary.
Indeed, suppose that Builder wishes to construct, say, a star with $\sqrt{n}$ leaves (so $H$ consists of a star with $\sqrt{n}$ leaves and $n-\sqrt{n}-1$ isolated vertices).
A very simple greedy strategy gives an essentially optimal result here: fix an arbitrary vertex $v$ and purchase only those edges incident to $v$ until the desired star is created.
Due to the high concentration in the degrees, a.a.s.\ the hitting time for this property is $(1-o(1))n^{3/2}/2$, and the strategy above is a successful $((1+o(1))n^{3/2}/2,\sqrt{n})$-strategy, which is clearly optimal.

We also believe that it would be interesting to understand what time is required in order to have successful strategies that use an ``almost optimal'' budget.
Our results for $F$-factors with $d^*(F)>1$ show that the time needs to be quadratic for a linear budget to suffice (\cref{thm:Ffactors_lower}). 
Is this a general phenomenon for other graph properties?
Let us say that the \defi{size} $ s = s(n) $ of a monotone property $\mathcal{P}$ is the minimum number of edges of a graph in $\mathcal{P}$, and that $ \mathcal{P} $ is \defi{affordable} if there exists a successful $(t,b)$-strategy for $ \mathcal{P} $ with $ b = s n^{o(1)} $ and $t\leq n^{2-\Omega(1)}$.
A general problem would be to classify which properties are affordable.

Note that all properties with $ \tau = s n^{o(1)} $, such as Hamiltonicity or connectivity, are trivially affordable.
In the case that $ s $ is constant, \citet{FKM25} showed that there exist strategies for fixed trees with constant budget and linear time (so this is an affordable property), whereas for small cycles any strategy with constant budget requires quadratic time (so this is not).
As another example, \cref{thm:Ffactors_lower} tells us that, if all graphs satisfying $ \mathcal{P} $ contain linearly many disjoint cycles of bounded size, then $ \mathcal{P} $ is not affordable.
This is indeed the case for many natural bounded-degree properties $ \mathcal{P} $ with $ s = \Omega(n) $ and $ \tau = n^{1 + \Omega(1)} $; it may therefore be tempting to conjecture that no such properties are affordable.
However, this is not true in general: for example, consider the property of containing both a Hamilton cycle and a copy of $ K_4 $.
The Hamilton cycle ensures that $ s = \Omega(n) $ and the clique ensures that $ \tau = n^{1 + \Omega(1)} $.
Nonetheless, there are $ (O(n\log n), O(n)) $-strategies for Hamiltonicity, and we can construct a $ K_4 $ in time $ n^{3/2 + o(1)} $ with budget $ n^{1/2 + o(1)} $ with a simple two-stage strategy.
Namely, in the first stage, buy any edges incident to a single fixed vertex and, in the second, buy all edges presented which are contained in its neighbourhood.
This example avoids the core of the problem by taking a non-affordable property and artificially making it affordable, by intersecting it with a trivially affordable property of larger size.
In general, this idea leads to a huge collection of affordable properties, but remains rather dissatisfying as a counterexample; we therefore ask the following question.

\begin{problem}
    Does there exist a `natural' (in particular, which is not an intersection similar to the one described above) bounded-degree property $\mathcal{P}$ of linear size which is affordable in a non-trivial way?
\end{problem}

Note that here we focus only on the bounded-degree case, because our upcoming work~\cite{EGNS25b} will show that triangle covers are an example of a `natural' affordable property (of unbounded degree).

There also remain more specific open questions related to our results.
For partial $F$-factors, \cref{thm:Ffactors_lower,thm:partialfactoruppergeneral} show that the corresponding upper and lower bounds on the budget~$b$ needed for successful $(t,b)$-strategies for each $t$ are tight (up to some constant factor which we make no effort to optimise).
However, for complete factors, this is not always the case.
\Cref{thm:Ffactors_lower,thm:factoruppergeneral} leave a polylogarithmic gap for $F$-factors when $F$ is strictly $1$-balanced, which in particular includes clique factors (as in \cref{thm:factors_lower,thm:factor_upper}).
Given that for perfect matchings ($K_2$-factors) the logarithmic term can be removed from the budget for all possible times~\cite{FKM25,Anastos22}, one could hope that these polylogarithmic terms can be removed for all $F$-factors for strictly $1$-balanced~$F$ when the time is of the order of the (expected) hitting time.
We believe that closing this gap is a very interesting open problem, even for the simplest case of triangle factors.

\begin{problem}\label{prob:close_gap}
    Does there exist a successful $(t,b)$-strategy for triangle factors with $t=\Theta(n^{4/3}\log^{1/3}n)$ and $ b = o(n^{4/3} \log^{1/3}n)$?
\end{problem}

One possible approach to tackling this problem is to consider the simpler task of purchasing a triangle cover (that is, the property that every vertex is contained in some triangle) in the context of the budget-constrained random graph process.
Having a triangle cover is clearly a necessary condition for having a triangle factor.
In the random graph process, it also turns out to be ``essentially'' sufficient: indeed, the hitting time result of \citet{HKMP23} states that a.a.s.\ the hitting times for these two properties coincide (and in fact gives the analogous statement for general cliques $K_r$ with $r\geq3$).
Therefore, one might hope that understanding triangle covers would at least allow for a conjectural answer to \cref{prob:close_gap}.
This, however, turns out not to be the case: there exists a successful $(t,b)$-strategy for triangle covers with $t$ of the order of the hitting time and a budget $b$ which is asymptotically smaller than the lower bound $ b = \Omega(n^{4/3} \log^{-1/6}n) $ for triangle factors provided by \cref{thm:factors_lower}.
We will detail our contributions towards clique covers in an upcoming work~\cite{EGNS25b}.

In a similar direction to \cref{prob:close_gap}, it would be interesting to close the gap between the lower and the upper bound for powers of Hamilton cycles (\cref{thm:hamsquare_lower,thm:hamsquare_upper}).

Lastly, as we discussed in \cref{sect:toolsintro}, we have introduced some general tools which may be applicable for different multi-stage strategies.
Our first approach was to obtain a general statement which ensures that, if one expects each stage to succeed a.a.s.\ ``independently'', then they do indeed a.a.s.\ succeed successively.
However, our \cref{lem:two_stage_real} is currently limited to applications in which only two stages can overlap (in the sense that the edge sets considered for each of them are non-disjoint).

\begin{problem} \label{problem:multi_stage}
    Is it possible to develop a version of \cref{lem:two_stage_real} which holds for arbitrary multi-stage strategies with a constant number of stages?
\end{problem}

Our second approach was to show that the removal of an arbitrary ``sparse'' graph from the edge-set presented in the process does not substantially affect the hitting times for the properties of interest.
This we applied in the very specific setting of finding linkages in the proof of \cref{thm:upperboundPowers} (see \cref{lem:linkage}).
It would be useful to have some form of ``meta-theorem'' which ensures this is the case in general, so as to avoid the need to prove such results for each specific property of interest.
We thus propose the following (somewhat vague) question.

\begin{problem} \label{problem:multi_stage2}
    Given a function $ p^* = p^*(n) $, does there exist a class $ \mathcal{C} $ of monotone properties, which is large enough to include all natural properties with threshold at most $ p^* $ which one might consider in the multi-stage process, as well as a notion of $p^*$-sparseness, such that a.a.s.\ $ G(n, p^* n^{o(1)}) $ is $p^*$-sparse and, for any property $ \mathcal{P} \in \mathcal{C} $ and any $p^*$-sparse graph $ H $, a.a.s.\ $ G(n, p) \setminus H $ satisfies $ \mathcal{P} $ whenever $ p = \omega(p^*) $?
\end{problem}

Even describing such a general class of properties appears to pose a significant challenge.
Note that, in general, one cannot simply choose $ \mathcal{C} $ to be the class of all properties with threshold at most~$ p^* $.
Indeed, suppose $ p^* = \omega(1/n) $, and given a notion of sparseness, fix a sparse graph $ H $ with $ |E(H)| = \Theta(p^* n^2) $ (which must exist), and consider the property $ \mathcal{P} $ of having non-empty intersection with $ H $.
Then, the threshold for $ \mathcal{P} $ is $ n^{o(1)} / |E(H)| \le p^* $, but clearly $ G(n, p) \setminus H $ never satisfies $ \mathcal{P} $.\COMMENT{The sparseness property should be the same for every $\mathcal{P}\in\mathcal{C}$. In our proof each $\mathcal{P}$ represents finding a linkage collection for a particular partition given in the previous stages.}

We remark that a positive answer to either one of these two problems in sufficient generality would be able to handle many multi-stage strategies, including ours.
Such general tools would be particularly helpful to allow us to reuse known results (such as those of \citet{JKV08}) as a black box, without having to reprove them in the multi-stage context.

\section*{Acknowledgements}

We would like to thank the referees for their careful reading of our manuscript, as well as for their many constructive comments.

We made use of ChatGPT Pro 5.5 to discover reference~\cite{GK11} related to \cref{claim:independence1}, as well as to find some typographical errors in a previous version of this work.
It also noticed that one condition in the statement of an earlier version of \cref{lem:deterministic} was not necessary.

\bibliographystyle{mystyle}
\bibliography{bib}


\appendix

\section{Proofs of \texorpdfstring{\cref{lemma:coupling_stages,claim:independence1,claim:independence2}}{Lemmas 3.3, 3.7 and 3.8}}\label{appen:multi_stage}

We present here the proofs which were left out in \cref{sect:couplings}.

\begin{proof}[Proof of \cref{lemma:coupling_stages}]
    For each $i\in[k]$, let $p_i=(1 - o(1)) t_i / M$ and $\mybar{p}_i=o(t_i / M)$ (the actual functions will be chosen below) and sample two graphs $H_i\sim G(n,p_i)$ and $\mybar{H}_i\sim G(n,\mybar{p}_i)$ independently, so~\ref{lem:couplingproperty1} and~\ref{lem:couplingproperty2} hold by definition.
    Let $p_i'\coloneqq p_i+\mybar{p}_i-p_i\mybar{p}_i$.
    Note that, in this case, $H_i'\coloneqq H_i\cup \mybar{H}_i\sim G(n,p_i')$.
    Now, for each $i\in[k]$, we sample $\hat{H}_i$ iteratively as follows.

    Suppose we have already generated $\hat{H}_1,\ldots,\hat{H}_{i-1}$.
    If $e(H_i\setminus\bigcup_{j=1}^{i-1}\hat{H}_j)>t_i$ or $e(H_i'\setminus\bigcup_{j=1}^{i-1}\hat{H}_j)<t_i$, we say that there is a \defi{failure at step $i$}.
    We split the definition of the coupling into two cases.
    If there is a failure in step $i$, then we completely stop the iterative process, discarding all its previous steps, and sample a random graph process $G_0,G_1,G_2\ldots$ independently from everything that came before, thus also generating a sequence of graphs $(\hat{H}_i)_{i\in[k]}\sim(\hat{G}_i)_{i\in[k]}$.
    If there is no failure at step $i$, then we sample $\hat{H}_i$ by taking a uniformly random set $\hat{E}_i$ of $t_i-e(H_i\setminus\bigcup_{j=1}^{i-1}\hat{H}_j)$ edges from $E(H_i'\setminus(H_i\cup\bigcup_{j=1}^{i-1}\hat{H}_j))$ and letting $\hat{H}_i\coloneqq (H_i\setminus\bigcup_{j=1}^{i-1}\hat{H}_j)\cup \hat{E}_i$.
    
    Note that, by Chernoff's bound (\cref{lem:Chernoff}), for an appropriate choice of $p_i=(1-o(1))t_i/M$ and of $\mybar{p}_i$ such that $p_i'=(1+o(1))t_i/M$, a.a.s.\ there is no failure at step $i$ (and thus, by a union bound over all steps, a.a.s.\ there will be no failure at any step).
    We remark that it is for these calculations that the assumptions on the $t_i$ are used.
    Moreover, if is there is no failure at any step, then property \ref{lem:couplingproperty3} holds by construction.
    Thus, it only remains to verify that \ref{lem:couplingproperty1.5} holds.
    Property \ref{lem:couplingproperty1.5} holds by definition conditional upon there being a failure at any step, so let us assume that there is no failure.

    Observe that, when considering the random graph process, for any $i\in[k]$ and conditional on the outcome of the (already sampled) graphs $\hat{G}_1,\ldots,\hat{G}_{i-1}$, the graph $\hat{G}_i$ is a graph chosen uniformly at random among those which have exactly $t_i$ edges and are edge-disjoint from $\bigcup_{j=1}^{i-1}\hat{G}_j$.
    Thus, we must prove that the graphs $\hat{H}_i$ sampled by our process also have this distribution.
    For each $i\in[k]$, let us denote by $\mathcal{C}_i$ the event that the coupling does not fail at any step at most $i$; moreover, let $\mathcal{C}_0$ denote the (trivial) event of probability $1$.
    Fix some $i\in[k]$ and condition on $\mathcal{C}_{i-1}$ and on the outcomes of $(H_j,\hat{H}_j,H_j')_{j\in[i-1]}$, which have already been sampled.
    If we write $\mathbb{P}^*$ to denote probabilities in this conditional space, for an arbitrary graph $G\subseteq K_n\setminus\bigcup_{j=1}^{i-1}\hat{H}_j$ with $e(G)=t_i$, it is not hard to verify that $\mathbb{P}^*[\hat{H}_i=G\mid\mathcal{C}_i]$ is a function of $ t_i $, $ p_i $, $ p_i' $ and $ s_i $ which does not depend on the graph $ G $ nor on the specific $\hat{H}_1,\ldots,\hat{H}_{i-1}$ that we conditioned on.
    Therefore, the sampled graph is chosen uniformly from all the possible options, just as we wanted to see.
\end{proof}

\begin{proof}[Proof of \cref{claim:independence1}]
    Let $\mathcal{A}$ and $\mathcal{B}$ denote the image sets (that is, the set of elements taken with non-zero probability) of $A$ and $B$, respectively.
    By considering an arbitrary injection, we may assume that $\mathcal{A}\subseteq\mathbb{N}$.
    For each $x\in\mathcal{A}$, let $a^-$ denote the maximum $a'\in\{0\}\cup\mathcal{A}$ such that $a'<a$.
    For each $y\in\mathcal{B}$, consider the cumulative distribution function $F_y\colon\{0\}\cup\mathcal{A}\to\mathbb{R}$ of~$A$ conditional on $B=y$, obtained by setting $F_y(x)\coloneqq\mathbb{P}[A\leq x\mid B=y]$.

    Now, given the pair of random variables $(A,B)$, sample $C\sim\mathcal{U}((F_B(A^-), F_B(A)])$.
    We define $f\colon\mathcal{B}\times(0,1]\to\mathcal{A}$ by setting $f(b,c)\coloneqq\min\{a\in\mathcal{A}:F_b(a)\geq c\}$.
    Then, given $(B,C)$, we have that $A=f(B,C)$.
    It remains to show that $C$ is independent of $B$.
    For every $c\in(0,1]$ and $b\in\mathcal{B}$ we have that
    \COMMENT{In the second equality we use first the fact that, if $A<f(b,c)$ then $C\leq F_b(f(b,c)^-)<c$, and second the fact that, if $A>f(b,c)$, then $A^-\geq f(b,c)$ and thus by definition we have that $C>F_b(f(b,c))\geq c$.
    Note that this implies that, in fact, $C\sim\mathcal{U}((0,1])$ (though it is not independent of $A$).}
    \begin{align*}
        \mathbb{P}[C\leq c\mid B=b]&=\mathbb{P}[C\leq c\mid B=b,A<f(b,c)]\mathbb{P}[A<f(b,c)\mid B=b]\\
        &\quad+\mathbb{P}[C\leq c\mid B=b,A=f(b,c)]\mathbb{P}[A=f(b,c)\mid B=b]\\
        &\quad+\mathbb{P}[C\leq c\mid B=b,A>f(b,c)]\mathbb{P}[A>f(b,c)\mid B=b]\\
        &=\mathbb{P}[A<f(b,c)\mid B=b]\\
        &\quad+\mathbb{P}[C\leq c\mid B=b,A=f(b,c)]\mathbb{P}[A=f(b,c)\mid B=b]+0\\
        &=F_b(f(b,c)^-)+\frac{c-F_b(f(b,c)^-)}{F_b(f(b,c))-F_b(f(b,c)^-)}(F_b(f(b,c))-F_b(f(b,c)^-))\\
        &=c.
    \end{align*}
    As this holds for every $b\in\mathcal{B}$, $C$ is indeed independent of $B$.
\end{proof}

\COMMENT{Proof by Zak under the added hypothesis that the variables take finitely many values and all probabilities in the underlying probability space are rationals with denominator at most $ q \in \mathbb{N} $. This was done before reading \cite{GK11}.
\begin{proof}
    Let $ N \coloneqq q! $.
    For each $x\in X$ and each $ y \in Y $ such that $\mathbb{P}[B=y]\neq0$, let $ p^y_x \coloneqq \mathbb{P}[A = x \mid B = y] $, and let $ \{I^y_x\}_{x \in X} $ be a partition of $ [N] $ into intervals of sizes $ p^y_x N $, respectively, noting these are all integers and sum to~$ N $. \aed{Sum to $N$ yes, but unclear that they would be integers. We currently have no information about the conditional probabilities. To change to a restriction that all probabilities in some underlying probability space are rationals. Plus change current setup for the coupling (since some variables are continuous).}
    Define $ C $ to be uniformly distributed among $ I^B_A $.
    For each $ i \in [N] $ and $y\in Y$, write $ x_i(y) $ for the unique $ x \in X $ such that $ i \in I^y_x $.
    Now, given any $ i \in [N] $ and $ y \in Y $, we see that 
    \[\mathbb{P}[C = i \mid B = y] = \mathbb{P}[C = i \mid A = x_i(y), B = y] \mathbb{P}[A = x_i(y) \mid B = y] = \frac{1}{|I^y_{x_i(y)}|} \cdot \frac{|I^y_{x_i(y)}|}{N} = 1 / N,\]
    which means that $ C $ is uniformly distributed over $ [N] $, and independent of $ B $.
    Furthermore, observe that, given $ B $ and $ C $, we have that $ A $ is the unique $ x \in X $ such that $ C \in I^B_x $, which proves the claim.
\end{proof}
}

\begin{proof}[Proof of \cref{claim:independence2}]
    Let $\Lambda$ denote a set which contains the images of $A,B,C,F,G$.
    We may assume that $\Lambda\subseteq\mathbb{N}$ (again, by taking an arbitrary injection).
    We first note that $A$ is independent of $G$ conditional upon $F$.
    Indeed, fix any $(x,y)\in\Lambda^2$ such that $\mathbb{P}[F=x,G=y]\neq0$, and let 
    \[\Phi\coloneqq\{v\in\Lambda:\mathbb{P}[C=v\mid F=x,G=y]\neq0\}.\]
    Now, using the fact that $G=g(F,C)$, as well as the independence assumptions, for any $\lambda\in\Lambda$ we have that
    \begin{align}\label{equa:zakconditionalindep}
        \mathbb{P}[A=\lambda\mid F=x, G=y]&=\sum_{v \in \Phi}\mathbb{P}[A=\lambda\mid F=x, C=v]\mathbb{P}[C=v\mid F=x, G=y]\nonumber\\
        &=\sum_{v\in\Phi}\mathbb{P}[A=\lambda\mid F=x]\mathbb{P}[C=v\mid F=x, G=y]\nonumber\\
        &=\mathbb{P}[A=\lambda\mid F=x].
    \end{align}

    For each $\lambda\in\Lambda$, let $\lambda^-\coloneqq\max\{\lambda'\in\{0\}\cup\Lambda:\lambda'<\lambda\}$.
    For each $x\in\Lambda$ such that $\mathbb{P}[F=x]\neq0$, let $H_{x}\colon\{0\}\cup\Lambda\to\mathbb{R}$ be the cumulative distribution function of $A$ conditional on $F=x$, defined by setting $H_x(\lambda)=\mathbb{P}[A\leq\lambda\mid F=x]$.
    Given the triple $(A,F,G)$, sample $D\sim\mathcal{U}((H_{F}(A^-),H_{F}(A)])$.
    Letting \[\Xi\coloneqq\{(x,z)\in\Lambda\times(0,1]:\mathbb{P}[F=x]\neq0\},\] we define a function $a\colon\Xi\to\Lambda$ by setting $a(x,z)\coloneqq\min\{\lambda\in\Lambda:H_{x}(\lambda)\geq z\}$.
    With this definition, we have that $A=a(F,D)$, as desired.
    
    Lastly, we verify that $D$ is independent of~$(F,G)$.
    Indeed, for any $(x,y,z)\in\Lambda^2\times(0,1]$ such that $\mathbb{P}[F=x,G=y]\neq0$ and using \eqref{equa:zakconditionalindep}, we have that
    \COMMENT{In the second equality we use first the fact that, if $A<a(x,z)$, then $D\leq H_x(a(x,z)^-)<z$, and second the fact that, if $A>a(x,z)$, then $A^-\geq a(x,z)$ and thus by definition we have that $D>H_x(a(x,z))\geq z$; and in the third equality we use the observation about conditional independence made above.
    Note that this implies that, in fact, $D\sim\mathcal{U}((0,1])$ (though it is not independent of $A$).}
    \begin{align*}
        \mathbb{P}[D\leq z\mid F=x,G=y]&=\mathbb{P}[D\leq z\mid F=x,G=y,A<a(x,z)]\mathbb{P}[A<a(x,z)\mid F=x,G=y]\\
        &\quad+\mathbb{P}[D\leq z\mid F=x,G=y,A=a(x,z)]\mathbb{P}[A=a(x,z)\mid F=x,G=y]\\
        &\quad+\mathbb{P}[D\leq z\mid F=x,G=y,A>a(x,z)]\mathbb{P}[A>a(x,z)\mid F=x,G=y]\\
        &=\mathbb{P}[A<a(x,z)\mid F=x,G=y]\\
        &\quad+\mathbb{P}[D\leq z\mid F=x,G=y,A=a(x,z)]\mathbb{P}[A=a(x,z)\mid F=x,G=y]+0\\
        &=H_x(a(x,z)^-)+\frac{z-H_x(a(x,z)^-)}{H_x(a(x,z))-H_x(a(x,z)^-)}(H_x(a(x,z))-H_x(a(x,z)^-))\\
        &=z,
    \end{align*}
    so $D$ is indeed independent of~$(F,G)$.
\end{proof}

\section{Proof of \texorpdfstring{\cref{claim:linkage_single}}{Claim 6.6}} \label{appen:linkage}

We consider parameters given by $ 1/n \ll 1/C_1 \ll 1 / r \ll \delta, 1/k \leq 1/2 $.
Fix $ \eta \in \mathbb{N} $, $ U \subseteq V $ and $ \mathcal{L} = (A_i, B_i)_{i \in [\eta]} $ as in the statement.
We claim that it suffices to show that, with probability at least $ 1 - n^{-\eta r \log^2{n}} $, there is some $i \in [\eta]$ for which $ G $ contains an $ (A_i, B_i) $-linkage $L$ of length~$ r $ in~$ V \setminus U $ whose internal vertices form an $ (A_i \cup B_i) $-independent set in~$ H $.
Indeed, in this case we would have $ L \subseteq G \setminus H $, and since there are at most $ n $ choices for $ \eta $, at most $ n^{\eta r+1} $ choices for $ U $, and at most $ n^{2k\eta} $ choices for $ \mathcal{L} $, by a union bound we then obtain the desired statement for all such choices with probability at least $ 1 - n^{-\eta r \log^2{n}} \cdot n^{2 + \eta r + 2k\eta} \ge 1 - n^{-\log^2{n}} $.\COMMENT{To get the last inequality we are using that $1/r\ll1/k$ (though there is a lot of leeway here).}
To prove this, we will apply \cref{lem:janson} to the set of all possible $ (A_i, B_i) $-linkages of length $r$ in $ V \setminus U $ whose internal vertices form an $ (A_i \cup B_i) $-independent set in $ H $, over all possible choices for $ i \in [\eta] $.

Recall the definitions of $ \mu $ and $ \Delta $.
Note that, for any $k$-endsequence pair $(A,B)$, an $ (A, B) $-linkage of length $ r $ has $ kr + \inbinom{k+1}{2}$ edges.
Recall also that, by assumption (see \cref{def:sparse-partitionable}), for each $ i \in [\eta] $, the set~$ V $ is $ (A_i \cup B_i, r, n/\log^4{n}, 1/\log{n}) $-sparse in~$ H $.
This means that there are at least $ |V \setminus U|^{r} / \log{n} \ge n^{r} \log^{-3r - 2}{n} $ sets $ X \subseteq V \setminus U $ of size~$ r $ which are $(A_i \cup B_i) $-independent in~$ H $.\COMMENT{This holds since, by the assumption on $\eta$, we have that $|U|\leq|V|/2$ so $|V\setminus U|\geq|V|/2\geq n/2\log^3n\geq n/\log^4n$.}
Thus, using the fact that $1/r\ll\delta$, we may bound\COMMENT{Sufficient to assume, for example, $ r \ge 2 / \delta $.}
\begin{equation}\label{equa:lastappendixboundmu}
    \mu \ge \eta \cdot n^{r} \log^{-3r - 2}{n} \cdot p^{kr + \genfrac{(}{)}{0pt}{2}{k+1}{2}} \ge \eta n^{- (k+1)/2 + \delta(kr + \genfrac{(}{)}{0pt}{2}{k+1}{2})} / \log^{3r + 2}{n} \ge \eta n.
\end{equation}

To bound $ \Delta $, we must consider pairs of linkages $ L_1, L_2 $; let $ i_1, i_2 \in [\eta] $ be such that $ L_j $ is an $ (A_{i_j}, B_{i_j}) $-linkage for each $ j \in [2] $.
Write $ \Delta = \Delta_1 + \Delta_2 $, where $ \Delta_1 $ is the sum over those pairs $ L_1, L_2 $ for which $ i_1 \ne i_2 $, and $ \Delta_2 $ is the sum over pairs with $ i_1 = i_2 $.
We start by bounding $ \Delta_1 $.
Observe that, if $ i_1 \ne i_2 $ and $ |V(L_1) \cap V(L_2)| = j \ge 2 $, then $ L_1 $ and $ L_2 $ certainly share at most $ k(j-1) $ edges.\COMMENT{Certainly at most $ \inbinom{j}{2} $, so if $ j \le 2k $ then we are done. 
Otherwise, consider the vertices in path order (say, with respect to $L_1$).
Then each vertex has at most $ k $ edges going forwards, and the last vertex has no edges going forwards.}
We may therefore bound\COMMENT{The term $\eta^2$ is a bound on the number of choices for $i_1$ and $i_2$.
The constant $C_1$ bounds, given that the intersection has exactly $j$ vertices, the number of ways in which this could happen (by considering different positions of such $j$ points within the order of the two linkages).
For the second inequality, by substituting $ \mu \ge \eta \cdot n^{r} \log^{-3r - 2}{n} \cdot p^{kr + \genfrac{(}{)}{0pt}{2}{k+1}{2}} $, we obtain the bound $\mu^2\sum_{j = 2}^{r} C_1 n^{-j} p^{-k(j-1)}\log^{6r+4}n$, and one can readily verify that the first term is the largest (and then use the lower bound on $p$ to obtain the next bound).
For the last inequality, simply note that $\eta$ is linear.}
\begin{equation}\label{equa:lastappendixbounddelta1}
    \Delta_1 \le \eta^2 \sum_{j = 2}^{r} C_1 n^{2r - j} p^{2kr + 2\genfrac{(}{)}{0pt}{2}{k+1}{2} - k(j-1)} \le 2C_1 \mu^2 n^{-1 - k\delta} \log^{6r + 4}{n} \le n^{-k \delta / 2} \mu^2 / \eta.
\end{equation}
In the case of $ \Delta_2 $, observe that if $ i_1 = i_2 = i $ and $ |(V(L_1) \cap V(L_2)) \setminus (A_i \cup B_i)| = j $ then $ L_1 $ and $ L_2 $ share at most $ kj + \inbinom{k+1}{2} $ edges, and in fact at most $ kj $ in the case that $ j \le r - k $.\COMMENT{Order the vertices in a path from $ A $ to $ B $ and start by counting at most $ kj $ edges forward, which misses at most $ \inbinom{k+1}{2} $ edges incident to $ A $.
In the case $ j \le r - k $, then it is easy to see that this overcounts at least $ \inbinom{k+1}{2} $ edges.
Indeed, observe that this counts at least $ k - i + 1 $ edges between the $ i $-th vertex which does not belong to the intersection and any vertices (including both those in $ A $ and in the intersection) to the left of it.
By the bound on $ j $, there are at least $ k $ such vertices so in total we have overcounted at least $ k + \cdots + 1 = \inbinom{k+1}{2}$ edges.}
We now similarly bound\COMMENT{If $ j \le r / 2 $, then we obtain $ (\mu^2/\eta) \cdot n^{-j} p^{-kj}\log^{6r + 4}{n} $ and $ n^{-j} p^{-kj} \le n^{-\delta kj} \le n^{-\delta k} $.
If $ j \ge r / 2 $, then we obtain $ n^{-j} p^{-kj - \inbinom{k+1}{2}} \le n^{(k+1)/2 -\delta k r / 2} \le n^{-\delta k} $, whenever (for example) $ r \ge 3 / \delta $.
Then absorb all constant and log terms by halving the exponent.}
\begin{equation}\label{equa:lastappendixbounddelta2}
    \Delta_2 \le \eta \sum_{j = 1}^{r} C_1 n^{2r - j} p^{2kr + 2\genfrac{(}{)}{0pt}{2}{k+1}{2} - kj - \genfrac{(}{)}{0pt}{2}{k+1}{2}\mathds{1}(j \ge r - k)} \le n^{- k \delta / 2} \mu^2 / \eta. 
\end{equation}

Combining \eqref{equa:lastappendixboundmu}--\eqref{equa:lastappendixbounddelta2}, we obtain that\COMMENT{Intermediate steps hidden: \[\ge \frac{\mu^2}{2 \mu + 8n^{- k \delta / 2} \mu^2 / \eta} \ge \frac{\eta}{2 \eta / \mu + 8n^{- k \delta / 2}} \ge \frac{\eta}{2 n^{-1} + 8n^{- k \delta / 2}}\]}
\[ \frac{\mu^2}{2\mu + 4\Delta} \ge \eta n^{k \delta / 3}. \]
Hence, by \cref{lem:janson}, there exists $ i \in [\eta] $ such that $ G $ contains an $ (A_i, B_i) $-linkage of length~$ r $ in~$ V \setminus U $, whose internal vertices form an $ (A_i \cup B_i) $-independent set in~$ H $, with probability at least
\[ 1 - \nume^{-\eta n^{k \delta / 3}} \ge 1 - n^{-\eta r \log^2{n}}, \]
as required.
\qed

\end{document}